\documentclass[oneside]{amsart}

\usepackage{amsthm}
\usepackage{amsmath}
\usepackage{amssymb}
\usepackage{enumerate}
\usepackage{graphicx}
\usepackage{booktabs}
\usepackage{color}
\usepackage[dvipsnames]{xcolor}
\usepackage{import}
\usepackage{tikz-cd}

\AtBeginDocument{%
   \def\MR#1{}
}

\usepackage{marginnote}
\long\def\@savemarbox#1#2{\global\setbox#1\vtop{\hsize\marginparwidth 
  \@parboxrestore\tiny\raggedright #2}}
\numberwithin{equation}{section}
\theoremstyle{plain}
\newtheorem{theorem}[equation]{Theorem}
\newtheorem{corollary}[equation]{Corollary}
\newtheorem{lemma}[equation]{Lemma}

\newtheorem{proposition}[equation]{Proposition}

\newtheorem*{namedtheorem}{\theoremname}
\newcommand{\theoremname}{testing}

\theoremstyle{definition}
\newtheorem{definition}[equation]{Definition}
\newtheorem{remark}[equation]{Remark}

\newtheorem{example}[equation]{Example}

\newcommand{\calT}{{\mathcal{T}}}

\newcommand{\cut}{\backslash\backslash}
\newcommand{\calH}{\mathcal{H}}
\newcommand{\calB}{\mathcal{B}}

\newcommand{\calF}{\mathcal{F}}

\makeatletter

\def\chaptermark#1{}

\def\chapter{%
  \if@openright\cleardoublepage\else\clearpage\fi
  \thispagestyle{plain}\global\@topnum\z@
  \@afterindenttrue \secdef\@chapter\@schapter}

\def\@chapter[#1]#2{\refstepcounter{chapter}%
  \ifnum\c@secnumdepth<\z@ \let\@secnumber\@empty
  \else \let\@secnumber\thechapter \fi
  \typeout{\chaptername\space\@secnumber}%
  \def\@toclevel{0}%
  \ifx\chaptername\appendixname \@tocwriteb\tocappendix{chapter}{#2}%
  \else \@tocwriteb\tocchapter{chapter}{#2}\fi
  \chaptermark{#1}%
  \addtocontents{lof}{\protect\addvspace{10\p@}}%
  \addtocontents{lot}{\protect\addvspace{10\p@}}%
  \@makechapterhead{#2}\@afterheading}
\def\@schapter#1{\typeout{#1}%
  \let\@secnumber\@empty
  \def\@toclevel{0}%
  \ifx\chaptername\appendixname \@tocwriteb\tocappendix{chapter}{#1}%
  \else \@tocwriteb\tocchapter{chapter}{#1}\fi
  \chaptermark{#1}%
  \addtocontents{lof}{\protect\addvspace{10\p@}}%
  \addtocontents{lot}{\protect\addvspace{10\p@}}%
  \@makeschapterhead{#1}\@afterheading}
\newcommand\chaptername{Chapter}

\def\@makechapterhead#1{\global\topskip 7.5pc\relax
  \begingroup
  \fontsize{\@xivpt}{18}\bfseries\centering
    \ifnum\c@secnumdepth>\m@ne
      \leavevmode \hskip-\leftskip
      \rlap{\vbox to\z@{\vss
          \centerline{\normalsize\mdseries
              \uppercase\@xp{\chaptername}\enspace\thechapter}
          \vskip 3pc}}\hskip\leftskip\fi
     #1\par \endgroup
  \skip@34\p@ \advance\skip@-\normalbaselineskip
  \vskip\skip@ }
\def\@makeschapterhead#1{\global\topskip 7.5pc\relax
  \begingroup
  \fontsize{\@xivpt}{18}\bfseries\centering
  #1\par \endgroup
  \skip@34\p@ \advance\skip@-\normalbaselineskip
  \vskip\skip@ }
\def\appendix{\par
  \c@chapter\z@ \c@section\z@
  \let\chaptername\appendixname
  \def\thechapter{\@Alph\c@chapter}}

\newcounter{chapter}

\newif\if@openright

\makeatother

\title[Incompressible surfaces, hierarchies and unknot recognition]{Incompressible surfaces, hierarchies \\ and unknot recognition}
\author{Marc Lackenby}
\address{Mathematical Institute, University of Oxford, \newline Woodstock Road, Oxford OX2 6GG, United Kingdom}

\begin{document}

\begin{abstract} \hspace{0.2cm}
We present a new algorithm to determine whether a compact orientable surface properly embedded in
a compact orientable 3-manifold is incompressible.  As a special case, this provides a new
algorithm to detect the unknot. The central technique is the use of hierarchies. Unlike previous algorithms, no challenging invariants need
to be computed and no lengthy search procedures are required. This algorithm leads to a new proof
that determining whether a compact orientable 3-manifold has incompressible boundary 
is in NP.
\end{abstract}
\maketitle
\tableofcontents

\section{Introduction}\label{Sec:Intro}
Incompressible orientable surfaces occupy a central position in 3-manifold theory. It is therefore very natural to ask whether a given orientable surface properly embedded in a 3-manifold is incompressible. The following result provides an algorithm to determine this.

\begin{theorem}
\label{Thm:AlgorithmIncompressible}
Let $\mathcal{T}$ be a triangulation of a compact orientable irreducible 3-manifold $M$, possibly with boundary. Let $S$ be a compact orientable normal surface properly embedded in $M$. Then there is an algorithm to determine whether $S$ is incompressible.
\end{theorem}

This theorem is not new, but the algorithm is. The first algorithmic solution to the problem in Theorem \ref{Thm:AlgorithmIncompressible} was given by Haken  \cite{Haken:Normal}. His procedure was to cut the manifold $M$ along $S$ and then triangulate the new 3-manifold $M \cut S$. Then he searched for possible normal compression discs for the two copies of $S$ in $M \cut S$. In order to make this a finite procedure, he showed that it was possible to restrict to `fundamental' normal discs. A more recent algorithm, due to the author \cite{Lackenby:EfficientCertification}, is more complicated, and uses sutured manifold hierarchies. Both these algorithms are useful and have interesting theoretical consequences. But they both have the downside that they require a search through a possibly very extensive search space, for example, the set of all fundamental normal surfaces.
The algorithm presented here has the significant advantage that it does not require examination of a large search space.
An interesting application of Theorem \ref{Thm:AlgorithmIncompressible} is the following.

\begin{corollary}
\label{Cor:UnknotRecognition}
There is an algorithm to determine whether a given knot in the 3-sphere is trivial.
\end{corollary}

This follows  from Theorem \ref{Thm:AlgorithmIncompressible} because the exterior of a knot
in the 3-sphere has incompressible boundary if and only if the knot is non-trivial.
Again, this result is not new. There are now many known
algorithms to detect the unknot, the first of which was due to Haken \cite{Haken:Normal}. Some require the computation
of knot invariants, such as Khovanov homology \cite{KronheimerMrowka} or Heegaard Floer homology \cite{OzsvathSzabo, SarkarWang}. Others require a possibly extensive search: a search for a spanning disc \cite{HassLagariasPippenger, Lackenby:EfficientCertification}, a homomorphism from
the fundamental group of the knot complement to a finite group with non-abelian image \cite{Kuperberg:Knottedness}, sequences of Reidemeister moves \cite{Lackenby:Reidemeister} or sequences of grid diagrams \cite{Dynnikov}.
As a result, none of these algorithms
are likely to run in polynomial time. Indeed, subject to some standard conjectures in complexity theory,
there is no sub-exponential time algorithm to compute Khovanov homology \cite{JaegerVertiganWelsh}.
(Here, we are using the definition of sub-exponential time requiring that the algorithm runs
faster than $O(2^{n^\epsilon})$ for each $\epsilon > 0$, where $n$ is the size of the input,
which in this case is the number of crossings.)

The algorithm that we give seems amenable to possible speed-up, which we will explore
in future work. Furthermore, in a companion paper \cite{Lackenby:LinkHyperbolicity}, we will show how this algorithm can be extended
to determine whether a given link in the 3-sphere is hyperbolic.

The algorithm generalises to 3-manifolds with a boundary pattern. Recall that
a \emph{boundary pattern} for a compact orientable 3-manifold $M$ is a collection $P$ of disjoint simple closed curves
and trivalent graphs embedded in $\partial M$. A boundary pattern $P$ is \emph{essential}  if, for every disc
$D$ properly embedded in $M$ that intersects $P$ transversely in at most three points away from the vertices of $P$,
there is a disc $D'$ in $\partial M$ with $\partial D' = \partial D$ and that intersects $P$ in the empty set, a single arc, or a graph with a single 3-valent vertex
in the interior of $D'$ and three 1-valent vertices on $\partial D'$. (See Figure~\ref{Fig:EssentialBoundPatt}.)
If there is a disc $D$ properly embedded in $M$ intersecting $P$ transversely at most three times away from the vertices, but 
there is no corresponding disc $D'$, then $D$ is termed a \emph{violating disc}. In the final case, 
when $D'$ intersects $P$ in a graph with a single 3-valent vertex
in the interior of $D'$ and three 1-valent vertices on $\partial D'$, this graph is known as a \emph{tripod}.
Boundary patterns will be the focus of this paper. At this stage, we observe that when $P$ is the empty pattern,
then $P$ is essential if and only if $\partial M$ is incompressible.

\begin{figure}[h]
\centering
\includegraphics[width=0.6\textwidth]{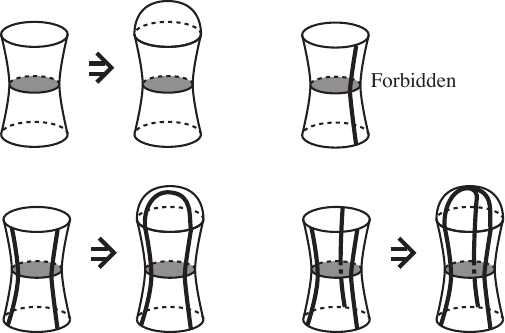}
\caption{Essential boundary pattern} \label{Fig:EssentialBoundPatt}
\end{figure}

\begin{theorem}
\label{Thm:AlgorithmEssentialPattern}
Let $(M,P)$ be a compact orientable irreducible 3-manifold with boundary pattern. Let $\mathcal{T}$ be a triangulation of $M$ in
which $P$ is simplicial. Then there is an algorithm to determine whether $P$ is essential. Moreover, if $P$ is not
essential, the algorithm provides a violating disc.
\end{theorem}

When $P$ is essential and $M$ has non-empty boundary, the algorithm outputs an essential hierarchy for $(M,P)$. Such hierarchies, which are defined in 
Section \ref{Sec:HierarchiesBPs}, will play a central role in this paper. They are a finite sequence of decompositions along properly embedded orientable surfaces where the final
manifold is a collection of 3-balls with an essential boundary pattern. 

In the later parts of the paper, we will establish that the algorithms given above can be used to show
that the decision problems are in NP. Moreover, we will drop the hypothesis that the 3-manifold is irreducible.
Specifically, we will consider the following problems.

The decision problem {\sc Incompressible boundary} takes, as its input a triangulation of a compact
orientable 3-manifold $M$ and it asks whether $\partial M$ is incompressible.

\begin{theorem} The decision problem {\sc Incompressible boundary} is in NP.
\end{theorem}

This theorem was proved by the author in \cite{Lackenby:EfficientCertification}, using sutured
manifold hierarchies, based on an argument of Agol. The problem was also shown to be in co-NP by Ivanov \cite{Ivanov}.

This has applications for the famous problem {\sc Knottedness}. This takes, as its input, either 
a diagram of a knot $K$ or a triangulation of its exterior, and it asks whether $K$ is a non-trivial knot.

\begin{corollary} 
The decision problem {\sc Knottedness} is in NP.
\end{corollary}

This was one of the main theorems in \cite{Lackenby:EfficientCertification}. Previously, Kuperberg \cite{Kuperberg:Knottedness}
showed that {\sc Knottedness} is in NP, conditionally on the Generalised Riemann Hypothesis. Very recently, Baldwin and Sivek \cite{BaldwinSivek}
have shown how to adapt Kuperberg's proof and remove its dependence on GRH, using the solution to the Quasi-Riemann Hypothesis \cite{OpenAI:QRH1, OpenAI:QRH2}.

The decision problem {\sc Essential boundary pattern} takes, as its input, a triangulation of a compact orientable
3-manifold $M$ with boundary pattern $P$ that is simplicial, and it asks whether $P$ is essential.
The following is a new result.

\begin{theorem} 
\label{Thm:EssentialNPcoNP}
The decision problem {\sc Essential boundary pattern} is in NP and co-NP.
\end{theorem}

When $P$ is inessential, the certificate is a violating disc in normal form. When $P$ is essential and $M$ is irreducible and has non-empty boundary,
the certificate is an essential hierarchy for $(M,P)$. The fact that essential hierarchies can be encoded efficiently
and can be certified as essential will be extremely useful in future work \cite{Lackenby:LinkHyperbolicity}.

Much of the theory of hierarchies as explained in Section \ref{Sec:HierarchiesBPs} is due to Johannson \cite{Johannson}. The presentation here
follows that of a graduate course  \cite{Lackenby:Notes} given by the author in 1999. One of the goals of that course was to provide a new proof of the loop theorem using hierarchies. The fact that the theory also has algorithmic applications was observed in a conversation with Ian Agol in 2013. The author developed and expanded the theory over the intervening years, resulting in this paper and its sequel \cite{Lackenby:LinkHyperbolicity}.

The structure of the paper is as follows. In Section \ref{Sec:HierarchiesBPs}, we recall some basic material on hierarchies and boundary patterns. In Section \ref{Sec:EssHierachies}, we explore essential hierarchies and prove one direction of Theorem \ref{Thm:EssentialHierarchy}, which asserts that when a 3-manifold admits an essential hierarchy, then it is irreducible and has essential boundary pattern. In Section \ref{Sec:HandleNormal}, we introduce normal surfaces, in the context of handle structures with a boundary pattern. In Section \ref{Sec:DecomposingHS}, we examine how handle structures behave when decomposed along a normal surface. The goal here is to establish an upper bound on the number of decompositions in a hierarchy. This gives the other direction of Theorem \ref{Thm:EssentialHierarchy}, which gives a sufficient condition for a 3-manifold with boundary pattern to admit an essential hierarchy. In Section \ref{Sec:AlgorithmCompressible}, we present the algorithm required by Theorem \ref{Thm:AlgorithmEssentialPattern}, and in Section \ref{Sec:Example}, an implementation of the algorithm in a specific example is given. The algorithm to determine whether a surface is incompressible, as required by Theorem \ref{Thm:AlgorithmIncompressible}, is given in Section \ref{Sec:AlgorithmIncompressible}. A crude analysis of the number of steps in the algorithm is given in Section \ref{Sec:NumberSteps}. The crucial quantity that controls an upper bound on the running time is the number of decompositions in the hierarchy. Therefore, in Section \ref{Sec:AlgorithmParallelity}, we give a more complicated algorithm for constructing a hierarchy, that has the advantage that the number of decompositions is explicitly controlled. In Section \ref{Sec:UniformType}, we show that the handle structures created in the hierarchy have uniform type, which means that each handle meets the pattern and its neighbouring handles in one of finitely many types, independent of the manifold. This is important for the later parts of the paper, in Sections \ref{Sec:CertifyingInessential}, \ref{Sec:Fundamental} and \ref{Sec:CertifyingEssentialHierarchy}, where we show that the decision problem in Theorem \ref{Thm:EssentialNPcoNP}, deciding whether a pattern is essential, is in NP and co-NP.

\begin{remark} The author was partially supported by EPSRC grant EP/Y004256/1. For the purpose of open access, the author has applied a CC BY public copyright licence to any author accepted manuscript arising from this submission. No AI was used in the creation of this article.
\end{remark}

\section{Hierarchies and boundary patterns}
\label{Sec:HierarchiesBPs}

We will use the following terminology, which has become standard.

\begin{definition}
\label{Def:MCutS}
Let $M$ be a manifold and let $S$ be either a codimension one submanifold properly embedded in $M$ or a compact codimension zero submanifold.
Then the manifold $M \cut S$ \emph{obtained by cutting $M$ along $S$} is defined as follows. Pick a Riemannian metric on $M$. Then $M - S$ inherits a Riemannian metric and $M \cut S$ is defined to be the metric completion of $M - S$. We also permit $S$ not to be a subset of $M$ as long as $S$ and $M$ are both subsets of some larger space and $S \cap M$ satisfies the above conditions, in which case $M \cut S$ is defined to be $M \cut (S \cap M)$. When $M$ is compact and $S$ is a codimension one submanifold, then the inclusion $M - S \rightarrow M$ induces a quotient map $M \cut S \rightarrow M$.
\end{definition}

Our algorithm makes use of hierarchies. Here, one starts with a compact orientable 3-manifold, and then one cuts along a properly embedded orientable surface, forming a new 3-manifold. In this manifold, one finds another properly embedded orientable surface, and one cuts along this. The process continues until the final manifold is a collection of 3-balls. The manifolds in the hierarchy will inherit a boundary pattern, as follows.

\begin{definition}
\label{Def:Decomposition}
Let $(M,P)$ be a compact orientable 3-manifold with boundary pattern. Then a compact orientable surface $S$ properly embedded in $M$ is \emph{transverse} to $P$ if it is disjoint from the vertices and intersects the edges transversely. We then say that $S$ is a \emph{decomposing surface}. The manifold $M \cut S$ inherits a boundary pattern $P'$, where $P'$ is the union of $P \cap (M \cut S)$ and the copies of $\partial S$ in $\partial (M \cut S)$. We write
$$(M, P) \xrightarrow{S} (M \cut S, P')$$
for this decomposition. We also denote $(M \cut S, P')$ by $(M,P) \cut S$. (See Figure \ref{Fig:DecomposeSurface}.)
\end{definition}

\begin{figure}[h]
\centering
\includegraphics[width=0.75\textwidth]{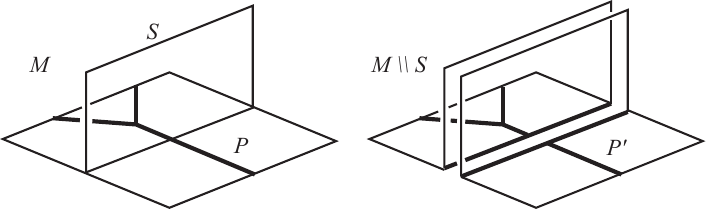}
\caption{Decomposing $(M,P)$ along a surface $S$} \label{Fig:DecomposeSurface}
\end{figure}

\begin{definition}
\label{Def:Hierarchy}
Let $(M,P)$ be a compact orientable 3-manifold with boundary pattern. Then a \emph{partial hierarchy} for $(M,P)$ is a sequence of decompositions
$$(M, P)  = (M_1,P_1) \xrightarrow{S_1} (M_2, P_2) \xrightarrow{S_2} \dots \xrightarrow{S_\ell} (M_{\ell+1}, P_{\ell+1}).$$
It is a \emph{hierarchy} if $M_{\ell+1}$ is a disjoint union of 3-balls. 
\end{definition}

This is a modification of the standard terminology, since we do not require the surfaces to be incompressible (although they will end up being so in the cases
of interest to us).

Recall, from Section \ref{Sec:Intro}, the definition of an essential boundary pattern.
At first sight, this is a strange definition, because it is not at all clear why one should be interested only in simple closed curves
that intersect the pattern in at most three points. We will give more motivation for it in the results that follow. 

\begin{definition}
\label{Def:EssentialHierarchy} 
An \emph{essential hierarchy} for a compact orientable 3-manifold $M$ with boundary pattern $P$ is a 
sequence of decompositions
$$(M, P)  = (M_1,P_1) \xrightarrow{S_1} (M_2, P_2) \xrightarrow{S_2} \dots \xrightarrow{S_\ell} (M_{\ell+1}, P_{\ell+1}),$$
where $(M_{\ell+1}, P_{\ell+1})$ is a collection of balls with essential boundary pattern and each $S_i$ has no 2-sphere components.
\end{definition}

Again, it is not immediately clear why one should only be concerned with whether the final pattern is essential. But this is
explained by the following result.

\begin{theorem}[Essential hierarchy iff essential pattern and irreducible]
\label{Thm:EssentialHierarchy}
Let $M$ be a compact connected orientable 3-manifold with a boundary pattern $P$.
Then $(M,P)$ admits an essential hierarchy if and only if the following all hold:
\begin{enumerate}
\item either $M$ contains a properly embedded orientable incompressible surface with no 2-sphere components, or $\partial M$ is non-empty;
\item $P$ is essential;
\item $M$ is irreducible.
\end{enumerate}
\end{theorem}

It is relatively straightforward to determine quickly whether a boundary pattern for a 3-ball is essential.
This can be rephrased in standard graph-theoretic terms as follows. Recall that a graph is \emph{simple}
if it has no edge loops and does not have multiple edges between pairs of vertices. A simple graph is \emph{$n$-connected} for some
positive integer $n$ if it has at least $n+1$ vertices and there is no set of $n-1$ vertices whose removal disconnects
the graph.

\begin{proposition}[Essential pattern for a ball]
\label{Prop:EssentialBall}
Let $M$ be a 3-ball, and let $P$ be a boundary pattern for $M$.
Then $P$ is essential if and only if $P$ is empty, a simple closed curve or a connected
graph, and in the latter case, the planar graph dual to $P$ is either a simple 4-connected graph or the complete graph on 3 or 4 vertices.
\end{proposition}

\begin{proof}
We first show that the given conditions on $P$ imply that $P$ is essential.
If $P$ is empty or a simple closed curve, then clearly $P$ is essential. So suppose that $P$ is a connected graph and that the graph $G$ dual to $P$ is simple and 4-connected or the complete graph on 3 or 4 vertices. Consider a simple closed curve $C$ in $\partial M$ that intersects $P$ transversely in at most 3 points. If $C$ is disjoint from $P$, then it must bound a disc in $\partial M$ disjoint from $P$, since $P$ is connected and $\partial M$ is a sphere. If $C$ intersects $P$ at one point, this specifies an edge loop in $G$, contrary to our assumption that $G$ is simple. If $C$ intersects $P$ at two points, then these two points must lie on the same edge of $P$, as otherwise we obtain two edges between the same pair of vertices of $G$. Hence, $C$ bounds a disc in $\partial M$ intersecting $P$ in an arc. If $C$ intersects $P$ at three points, then this specifies a closed walk of length 3 in $G$. This is a cycle because otherwise $G$ is not simple. Removing the vertices in this cycle does not separate $G$, because $G$ is 4-connected or the complete graph on 3 or 4 vertices. Hence, $C$ must bound a disc intersecting no regions of $\partial M - P$ other than the three that $C$ visits. This implies that $P$ intersects this disc in a tripod. Thus we have shown that $P$ is essential.

Conversely, suppose that $P$ is essential. We may suppose that $P$ is non-empty and not a simple closed curve, as the implication is immediate in those cases. Then $P$ is clearly connected as otherwise there is a simple closed curve in $\partial M - P$ that does not bound a disc in $\partial M$ disjoint from $P$.  Let $G$ be its dual planar graph. An edge loop in $G$ would give a simple closed curve intersecting the pattern once, contradicting the assumption that $P$ is essential. More than one edge between a pair of vertices would give a simple closed curve in $\partial M$ intersecting $P$ in exactly two points that lie in different edges of $P$. Again this contradicts the assumption that $P$ is essential. Consider a collection of at most three vertices in $G$. We will show that these do not separate $G$. Each vertex corresponds to a region of $\partial M - P$, and for these vertices to disconnect $G$, the union of these regions must disconnect $\partial M$. If they do then there must be a simple closed curve lying in the union of these regions that separates at least two other regions of $\partial M - P$. Since $P$ is essential, this curve bounds a disc in $\partial M$ intersecting $P$ in an arc or a tripod. This implies that these three vertices did not in fact separate $G$. Note that removing these three vertices of $G$ may create the empty graph, but in this case, $G$ is the complete graph on 3 vertices. Hence, we have shown that $G$ is 4-connected or the complete graph on 3 or 4 vertices.
\end{proof}

Note that it can be determined in polynomial time (as a function of the number of vertices and edges) whether a graph is simple and 4-connected or the complete graph on 3 or 4 vertices.

Theorem \ref{Thm:EssentialHierarchy} is an extremely powerful tool for showing that a pattern is essential, as the following example demonstrates.

\begin{figure}[h]
\centering
\includegraphics[width=4in]{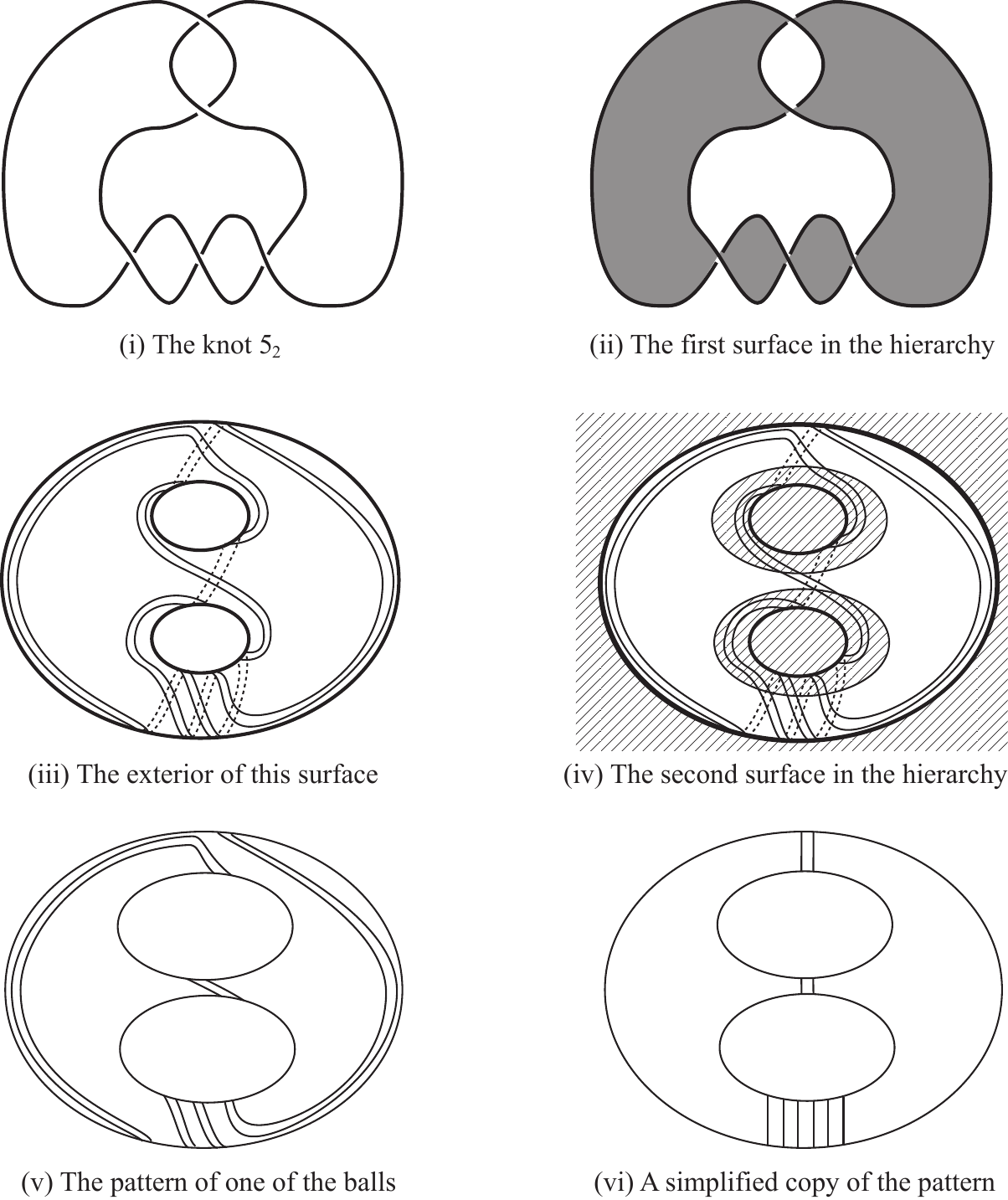}
\caption{An essential hierarchy} \label{Fig:KnotEssentialHierarchy}
\end{figure}

\begin{example}
The knot $5_2$ is shown in Figure \ref{Fig:KnotEssentialHierarchy}. Let $M$ be its exterior and let $P$ be the empty pattern. A hierarchy for $(M,P)$ is given. The final manifold consists of two 3-balls, and the boundary pattern of one of them is shown. It is easy to check that this pattern is essential, and hence by Theorem \ref{Thm:EssentialHierarchy}, $(M,P)$ is essential and so the knot $5_2$ is not the unknot.
\end{example}

Theorem \ref{Thm:EssentialHierarchy} is not new. It is implicit in work of Johannson \cite{Johannson} and a version of it was used by the author in lecture notes in 1999 \cite{Lackenby:Notes}. However, because it is so central to our work and because we will need details of the argument later in the paper, we give a full proof. 

\section{Essential hierarchies}\label{Sec:EssHierachies}

Our aim is to prove 
Theorem \ref{Thm:EssentialHierarchy}.
We first show that if $(M,P)$ admits an essential hierarchy, then $P$ is essential and $M$ is irreducible. We will also establish that
the decomposing surfaces have the following important property.

\begin{definition}\label{Def:PatternCompDisc}
Let $M$ be a compact orientable 3-manifold and let $P$ be a boundary pattern. Let $S$ be a compact surface properly
embedded in $M$. Then a \emph{pattern-compression disc} for $S$ is a disc $D$ embedded in $M$ such that
\begin{enumerate}
\item $D \cap S$ is an arc $\alpha$ in $\partial D$;
\item $\partial D \cut \alpha  = D \cap \partial M$ intersects $P$ at most once; and
\item $\alpha$ does not separate off a disc from $S$ intersecting $P$ at most once.
\end{enumerate}
If no such pattern-compression disc exists, then $S$ is \emph{pattern-incompressible}. (See Figure \ref{Fig:PatternIncomp}.) If $D$ is a pattern-compression disc for $S$, 
then the surface obtained from $S$ by cutting along $S \cap D$ and then attaching two parallel copies of $D$ is the surface
obtained by \emph{pattern-compressing} along $D$. 
\end{definition}

\begin{figure}[h]
\centering
\includegraphics[width=0.9\textwidth]{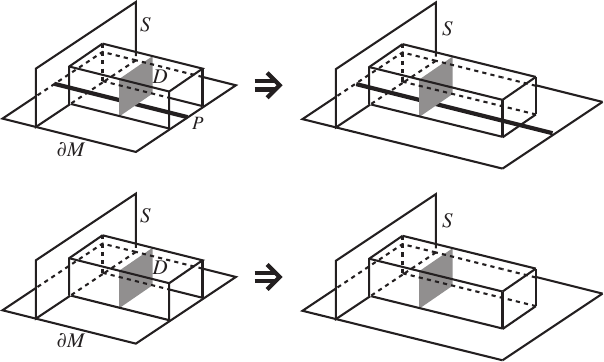}
\caption{A pattern-incompressible surface. (For the sake of clarity, we have assumed here that $M$ is irreducible and $P$ is essential.)} \label{Fig:PatternIncomp}
\end{figure}

A key step in the proof of Theorem \ref{Thm:EssentialHierarchy} is the following.

\begin{lemma}[Essential pattern pulls back]
\label{Lem:EssentialPullsBack}
Let 
$$(M, P) \xrightarrow{S} (M', P')$$
be a decomposition and suppose that $P'$ is essential. Then $P$ is essential and $S$ is incompressible and pattern-incompressible.
\end{lemma}

\begin{proof}
We first show that $P$ is essential.
Let $D$ be a disc properly embedded in $M$ that is transverse to $P$ and that intersects $P$ at most three times. We wish to show that $D$ is not a violating disc.

Perform a small isotopy to $S$, leaving $P$ invariant, so that afterwards $S \cap D$ is a collection of simple closed curves and arcs properly embedded in $D$. 

If there is a simple closed curve of $S \cap D$, there is one that is innermost in $D$. It bounds a disc $D'$ in $M'$ that is disjoint from $P'$. Since we are assuming that $P'$ is essential, $\partial D'$ bounds a disc $D''$ in $\partial M'$ that is disjoint  from the pattern. This disc is a subset of $S$. We can remove $D'$ from $D$ and replace it by a parallel copy of $D''$. The result is violating for $P$ if and only if $D$ was violating. However, its intersection with $S$ contains fewer simple closed curves.

So, we may assume that $D \cap S$ is a collection of arcs properly embedded in $D$. Say there are $n>0$ of them. Then at the endpoints of these arcs, we see $4n$ points of intersection with $P'$. We also have at most $3$ points of intersection between $\partial D$ and $P$. So, in the $n+1$ discs $D \cut (D \cap S)$, there are at most $4n+3$ points of intersection with $P'$. Hence, one of these discs, $D_1$ say, has at most three points of intersection with $P'$ in its boundary. Necessarily, it has at least two points of intersection with $P'$. We are assuming that $P'$ is essential. Hence, $\partial D_1$ bounds a disc $D_2$ in $\partial M'$. This disc $D_2$ intersects $P'$ either in a single arc or a tripod. Therefore, $D_2\cap S$ is a disc. We may remove $D_1$ from $D$ and replace it by $D_2 \cap S$, and then perform a further small isotopy to give a disc $D_3$ properly embedded in $M$ that is transverse to $P$ and that has fewer components of intersection with $S$. This is a violating disc for $M$ if and only if $D$ was.

So we may assume that $D$ is disjoint from $S$ and hence lies in $M'$. Since $P'$ is essential, $\partial D$ bounds a disc $D'$ in $\partial M'$ intersecting $P'$ in the empty set, an arc or a tripod. Since $\partial D'$ is disjoint from $S$, all of $D'$ must be disjoint from $S$. So, $D'$ lies in $\partial M$. So, $D$ was not a violating disc, and so we have proved that $P$ is essential. 

We now show that $S$ is incompressible. Consider a disc $D$ embedded in the interior of $M$ such that $D \cap S = \partial D$. Then $D$ is properly embedded in $M'$ and is disjoint from the pattern $P'$. Since $P'$ is essential, $\partial D$ bounds a disc in $\partial M'$ disjoint from $P'$. This is a disc in $S$. So $D$ is not a compression disc for $S$.

We now show that $S$ is pattern-incompressible. Let $D$ be a disc embedded in $M$ such that $D \cap S$ is an arc in $\partial D$, and the remainder of $\partial D$ is an arc in $\partial M$ intersecting $P$ at most once. Then $D$ lies in $M'$ and intersects $P'$ at least twice and at most three times. Since $P'$ is essential, $\partial D$ bounds a disc $D'$ in $\partial M'$ intersecting $P'$ in an arc or a tripod.
The intersection between $D'$ and $S$ is a disc separated off by $D \cap S$ that intersects $P$ at most once. So $D$ is not a pattern-compression disc.
\end{proof}

\begin{lemma}[Irreducible and essential pattern pulls back]
\label{Lem:EssentialIrredPullsBack}
Let 
$$(M, P) \xrightarrow{S} (M', P')$$
be a decomposition and suppose that $M'$ is irreducible and $P'$ is essential. Then $M$ is irreducible and $P$ is essential.
\end{lemma}

\begin{proof}
The fact that $P$ is essential is a result of Lemma \ref{Lem:EssentialPullsBack}. We need to show that $M$ is irreducible. 
Consider a 2-sphere $S^2$ properly embedded in $M$. This intersects $S$ in a collection of simple closed curves. An innermost one bounds a disc $D'$ in $M'$ that is disjoint from $P'$. Since $P'$ is essential, $\partial D'$ bounds a disc $D''$ in $S$ and because $M'$ is irreducible, we may isotope $D''$ across to $D'$ and then a little further, to remove this curve of $S^2 \cap S$. So we may assume that $S$ is disjoint from $S^2$, and hence $S^2$ lies in $M'$. As $M'$ is irreducible, $S^2$ bounds a ball in $M'$ and this is also a ball in $M$.
\end{proof}

In many circumstances, it is also the case that when $P$ is essential and $M$ is irreducible, then $P'$ is also essential and $M'$ is also irreducible. However, in order to guarantee this, we require that the decomposing surface $S$ is incompressible and pattern-incompressible. We now investigate how to arrange this.

When we pattern-compress a surface, we would like to say that the new surface is, in some sense, simpler than the
original one. We formalise this as follows, using the following measure of complexity.

\begin{definition}
The \emph{pattern-complexity} $\chi_p(S)$ of a properly embedded surface $S$ is defined to be $-4\chi(S) + |S \cap P|$.
\end{definition}

A useful feature of pattern-complexity is the following immediate result.

\begin{lemma}
If $S$ is a compact connected orientable surface satisfying $\chi_p(S) < 0$, then $S$ is a sphere or disc intersecting the pattern in at most three points.
\end{lemma}

\begin{lemma}[Pattern-compression simplifies]
\label{Lem:CompressionAndPatternComplexity}
Let $S$ be a properly embedded connected orientable surface that is non-separating. Let $S'$ be obtained from $S$ by pattern-compressing or compressing. Then at least one component of $S'$ is also non-separating and has smaller pattern-complexity than $S$.
\end{lemma}

\begin{proof}
The lemma is straightforward and well known in the case of a compression. So suppose that we pattern-compress $S$.
We decrease $-4\chi(S)$ by $4$ and increase $|S \cap P|$ by at most $2$. So $\chi_p$ certainly goes down by at least $2$. Note that at least one component of $S'$ is non-separating. Thus, the lemma is proved in the case where $S'$ is connected. However, we need to consider the possibility that we have  created a separating component $S'_1$ with negative pattern-complexity. Suppose that this occurs. Let $S'_2$ be the other component of $S'$ which is necessarily non-separating. Note that $S'_1$ is a disc and it must intersect $P$ at least twice by the definition of a pattern-compression. If $S'_1$ has three points of intersection with $P$, then $|S'_2 \cap P| < |S \cap P|$ and hence $\chi_p(S'_2) < \chi_p(S)$. If $S'_1$ has two points of intersection with $P$, then both of these must lie in $S$ by the definition of a pattern-compression. Hence, again $\chi_p(S'_2) < \chi_p(S)$.
\end{proof}

\begin{corollary}[Simplest surface pattern-incompressible]
\label{Cor:MinimalComplexityImpliesEssential}
Let $(M,P)$ be a compact orientable 3-manifold with a boundary pattern.
Let $S$ be a properly embedded connected orientable surface that is non-separating and that has smallest pattern-complexity among all such surfaces. Then $S$ is pattern-incompressible and incompressible.
\end{corollary}

\begin{lemma}[Essential pushes forward]
\label{Lem:EssentialPushesForward}
Let $M$ be a compact orientable 3-manifold with essential boundary pattern $P$. Let $S$ be a connected pattern-incompressible incompressible surface properly embedded in $M$, which is not a disc intersecting $P$ in at most $3$ points. Then $M \cut S$ inherits an essential boundary pattern $P'$.
\end{lemma}

\begin{proof}
Let $D$ be a disc properly embedded in $M \cut S$ with $\partial D \cap P'$ at most three points. The curve $\partial D$ may run through parts of $\partial (M \cut S)$ coming from $\partial M$ and parts coming from $S$. Note however the points where it swaps must be points of $\partial D \cap P'$, and that at most one side of any point of $\partial D \cap P'$ lies in $S$. Hence, at most one arc or simple closed curve of $\partial D \cut P'$ lies in $S$.

\medskip
\emph{Case 1.} $\partial D$ is disjoint from $S$.
\medskip

Then $\partial D \subset \partial M$. Since $P$ is essential, $\partial D$ bounds a disc $D'$ in $\partial M$ intersecting $P$ in the empty set, an arc or a tripod.
If $S$ intersects $D'$, then pick a simple closed curve of $S \cap D'$ innermost in $D'$. The disc this bounds cannot be a compression disc for $S$. Hence, $S$ must be a disc. By assumption, $S$ intersects $P$ in at least 4 points. Hence, it intersects some edge of $P \cap D'$ more than once. We can therefore find a pattern-compression disc for $S$, which is contrary to assumption. Hence, $D'$ is disjoint from $S$, and therefore lies in $\partial (M \cut S)$. This verifies that $D$ is not a violating disc for $P'$.

\medskip
\emph{Case 2.} $\partial D$ lies in $S$.
\medskip

By the incompressibility of $S$, $\partial D$ bounds a disc $D'$ in $S$. This is a disc in $\partial (M \cut S)$ disjoint from $P'$. So $D$ is not violating.

\medskip
\emph{Case 3.} $\partial D$ intersects $S$ in an arc.
\medskip

Then $\partial D \cut S$ intersects $P$ at most once. Since $D$ is not a pattern-compression disc for $S$, $D \cap S$ separates off a disc $D_1$ of $S$ intersecting $P$ in at most one point. Then, $D \cup D_1$ is a disc properly embedded in $M$, intersecting $P$ in at most two points. Since $P$ is essential, there is therefore a disc $D_2$ in $\partial M$ with $\partial D_2 = \partial (D \cup D_1)$ and with $D_2 \cap P$ either empty or an arc. Therefore, $D_1 \cup D_2$ is a disc in $\partial (M \cut S)$ intersecting $P'$ in an arc or a tripod. 

This gives that $P'$ is essential.
\end{proof}

We also have the following well known result.

\begin{lemma}[Irreducible pushes forward]
\label{Lem:IrreduciblePushesForward}
Let $M$ be a compact orientable irreducible 3-manifold. Let $S$ be a properly embedded incompressible surface with no 2-sphere components. Then $M \cut S$ is irreducible.
\end{lemma}

We will occasionally want to cut along surfaces in the expectation that they are pattern-incompressible. An example of this is as follows.

\begin{lemma}[Decomposition along square]
\label{Lem:FourPuncDisc}
Let $(M,P)$ be a compact orientable 3-manifold with an essential boundary pattern. Let $S$ be a properly embedded disc that intersects $P$ four times. Then $S$ is pattern-incompressible unless $\partial S$ bounds a disc $D$ in $\partial M$ that has one of the following forms:
\begin{enumerate}
\item $D \cap P$ consists of two arcs; or
\item $D \cap P$ is a graph with two vertices joined by an edge, along with four edges running to the boundary of $D$.
\end{enumerate}
\end{lemma}

\begin{proof} 
Suppose that, on the contrary, $S$ has a pattern-compression disc $D'$. Then $S \cut D'$ has two components $S_1$ and $S_2$, each of which is a disc
intersecting $P$ twice. For $i = 1$ and $2$, let $D'_i$ be $S_i \cup D'$ isotoped a little so that it is disjoint from $S$. Then $D'_i$ is properly embedded in $M$ and intersects $P$ either two or three times. If each of these discs intersects $P$ twice, then using the fact that $P$ is essential, we deduce that $\partial S$ bounds a disc $D$ in $\partial M$ as in (1) above. If both $D_1'$ and $D'_2$ intersect $P$ three times, then we deduce that $\partial S$ bounds a disc $D$ in $\partial M$ as in (1) or (2) above.
\end{proof}

\begin{definition}
\label{Def:FourDisc}
We say that a disc $S$ properly embedded in $M$ intersecting the pattern four times is a \emph{square}. It is \emph{trivial} if $\partial S$ bounds a disc $D$ in $\partial M$ satisfying (1) or (2) of Lemma \ref{Lem:FourPuncDisc}. (See Figure \ref{Fig:Trivial4Discs}.) Otherwise it is \emph{non-trivial}.
\end{definition}

\begin{figure}[h]
\centering
\includegraphics[width=0.9\textwidth]{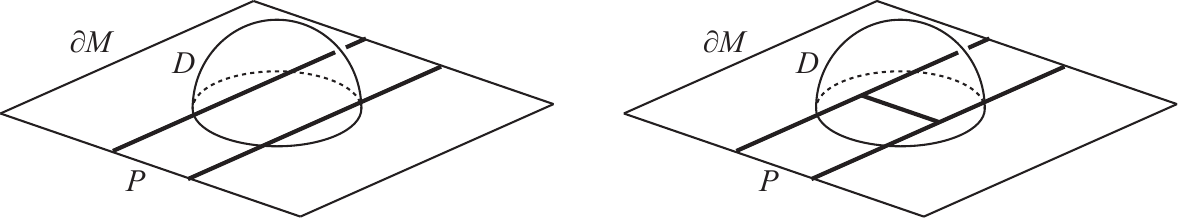}
\caption{Trivial squares} \label{Fig:Trivial4Discs}
\end{figure}

We rephrase and extend Lemma \ref{Lem:FourPuncDisc} as follows.

\begin{lemma}
\label{Lem:PatternComp4Disc}
Let $(M,P)$ be a compact orientable 3-manifold with a boundary pattern. Let $S$ be a non-trivial square. Then a pattern-compression disc for $S$ yields a violating disc for $(M,P)$. This violating disc is one of the components of the surface obtained by pattern-compressing $S$.
\end{lemma}

We now have many of the elements required for the proof of Theorem \ref{Thm:EssentialHierarchy}. 
Let $(M,P)$ be a compact orientable irreducible 3-manifold with an essential boundary pattern. 
Our aim is to build an essential hierarchy for $(M,P)$.
If $M$ is closed, then by assumption, it contains a properly embedded incompressible 
orientable surface with no 2-sphere components. Decomposing along this surface gives an irreducible 3-manifold with non-empty boundary and essential (empty) boundary pattern. So, we may assume that $\partial M$ is non-empty.
If $\partial M$ is a union of spheres, then these bound 3-balls by the irreducibility of $M$, and so in this case there is nothing to prove. Hence, we may assume that $b_1(\partial M) > 0$. Since $b_1(M) \geq b_1(\partial M)/2$, we deduce that $b_1(M)>0$. Therefore, $M$ contains a properly embedded connected orientable non-separating surface $S$. By picking $S$ to have smallest pattern-complexity, we may assume that $S$ is incompressible and pattern-incompressible by Corollary \ref{Cor:MinimalComplexityImpliesEssential}. When we decompose along such a surface, Lemmas \ref{Lem:EssentialPushesForward} and \ref{Lem:IrreduciblePushesForward} imply that the resulting manifold is irreducible and has essential boundary pattern. We may therefore repeat with this manifold. However, we must show that this process can be arranged to terminate. In \cite{Lackenby:Notes}, a proof was given using Kneser's theorem bounding the number of disjoint non-parallel properly embedded incompressible boundary-incompressible surfaces in a compact 3-manifold. In the following section, we give an alternative proof. Although longer than the proof in \cite{Lackenby:Notes}, it will be more useful because it will lead to a good bound on the length of the resulting hierarchy.

\section{Handle structures and normal surfaces} 
\label{Sec:HandleNormal}

Our algorithm deals with partial hierarchies. The way that the surfaces in these partial hierarchies are encoded is as normal surfaces. Normal surface theory takes place in a 3-manifold equipped either with a triangulation or with a handle structure. We mostly use handle structures, as they behave better than triangulations when the 3-manifold is decomposed along a normal surface. Normal surfaces in handle structures have been used in the theory of hierarchies since the work of Haken \cite{Haken:Homeomorphism} and Waldhausen \cite{Waldhausen}. One of their particular uses is as a method for controlling the length of hierarchies, in other words the number of decompositions until the manifold has been cut up into balls.

\begin{definition}
\label{Def:HandleStructure}
In this paper, a \emph{handle structure} on a $3$-manifold or a $2$-manifold
will satisfy the following conditions:
\begin{enumerate}
\item each $i$-handle $D^i \times D^{k-i}$ ($k = 2$ or $3$) intersects the handles of lower index in $\partial D^i \times D^{k-i}$;
\item any two $i$-handles are disjoint;
\item in the case of a 3-manifold, the intersection of any 1-handle $D^1 \times D^2$ with any 2-handle $D^2 \times D^1$ is of the form $D^1 \times \alpha$ in $D^1 \times D^2$ where $\alpha$ is a collection of arcs in $\partial D^2$, and of the form $\beta \times D^1$ in $D^2 \times D^1$ where $\beta$ is a collection of arcs in $\partial D^2$;
\item any 2-handle runs over at least one 1-handle.
\end{enumerate}
\end{definition}

\begin{definition}
Let $\mathcal{H}$ be a handle structure for a 3-manifold and for $i \in \{ 0,1,2,3 \}$, let $\mathcal{H}^i$ denote the
union of the $i$-handles.
The \emph{characteristic surface} for $\mathcal{ H}$ is
$\mathcal{ H}^0 \cap (\mathcal{ H}^1 \cup \mathcal{ H}^2)$. We denote it by $\mathcal{ F}(\mathcal{ H})$, or simply by $\mathcal{F}$ when the context is clear.
It has a handle structure consisting of just 0-handles and 1-handles, where the 0-handles are
$\mathcal{ F}^0(\mathcal{ H})  = \mathcal{ H}^0 \cap \mathcal{ H}^1$ and the 1-handles are $\mathcal{ F}^1(\mathcal{ H}) = 
\mathcal{ H}^0 \cap \mathcal{ H}^2$. We will view the characteristic surface as a subsurface of $\partial \mathcal{ H}^0$.
\end{definition}

We denote the number of 0-handles in the handle structure $\calH$ by $|\calH|$.

\begin{definition}
A \emph{handle structure} for a 3-manifold $M$ with boundary pattern $P$
is a handle structure for $M$ satisfying Definition \ref{Def:HandleStructure} such that $P$ is disjoint from the 2-handles
and, for each 1-handle $H_1 = D^1 \times D^2$, the intersection $H_1 \cap P$ is $D^1 \times X$
where $X$ is a finite set of points in $\partial D^2$.
\end{definition}

We will be decomposing handle structures along normal surfaces. But first we deal with 
the following weaker notion.

\begin{definition}
Let $\calH$ be a handle structure of a 3-manifold $M$ with boundary pattern $P$. We say that a surface $S$ properly embedded in $M$ is \emph{standard} with respect to $\calH$ if the following hold:
\begin{enumerate}
\item it intersects each 0-handle in a collection of disjoint properly embedded discs;
\item it intersects each 1-handle $D^1 \times D^2$ in a surface of the form $D^1 \times C$ for some arcs $C$ properly embedded in $D^2$;
\item it intersects each 2-handle $D^2 \times D^1$ in a surface of the form $D^2 \times X$ for some finite collection of points $X$ in the interior of $D^1$;
\item it is disjoint from the 3-handles;
\item it intersects $P$ transversely.
\end{enumerate}
\end{definition}

We have analogous definitions for 1-manifolds in surfaces.

\begin{definition}
Let $\calH$ be a handle structure of a surface $F$. We say that a 1-manifold $C$ properly embedded in $F$ is \emph{standard} with respect to $\calH$ if the following hold:
\begin{enumerate}
\item it intersects each 0-handle in a union of disjoint properly embedded arcs;
\item it intersects each 1-handle $D^1 \times D^1$ in a union of arcs parallel to the core of the handle, that is, arcs of the form $D^1 \times X$ for some finite collection of points $X$ in the interior of $D^1$;
\item it is disjoint from the 2-handles.
\end{enumerate}
\end{definition}

\begin{definition}
\label{Def:Normal}
Let $\mathcal{ H}$ be a handle structure for a 3-manifold $M$ with boundary pattern $P$.
A curve $C$ in $\partial \mathcal{ H}^0$ is \emph{normal} if it intersects $P$ and $\partial \calF$ transversely, its intersection with $\calF$ is standard and the following all hold:
\begin{enumerate}
\item $C$ does not bound a disc in $\partial \calH^0 - \calF$ that intersects $P$ in the empty set, a single arc or a tripod;
\item $C$ intersects each 1-handle of $\mathcal{ F}$ in at most one arc;
\item $C$ intersects each component of $\partial \mathcal{ H}^0 - (\mathcal{ F} \cup P)$ in at most one arc or simple closed curve;
\item $C$ does not intersect components of $\mathcal{F}^1$ and $\partial \mathcal{ H}^0 - (\mathcal{ F} \cup P)$ that are adjacent;
\item if $C$ intersects two components of $\partial \mathcal{ H}^0 - (\mathcal{ F} \cup P)$ that are incident along an arc of $P \cap \calH^0$, then $C$ also intersects this arc;
\item there is no disc in $\partial \mathcal{H}^0$ that has boundary equal to the union of a sub-arc of $C$ and a sub-arc of $\partial \mathcal{F}^0$ and is otherwise disjoint from $C$ and $\mathcal{F}$, and that intersects $P$ in at most two arcs, each of which runs from $\partial \mathcal{F}^0$ to $C$.
\end{enumerate}
See Figure \ref{Fig:NormalConditions}.
\end{definition}

\begin{figure}[h]
\centering
\includegraphics[width=0.8\textwidth]{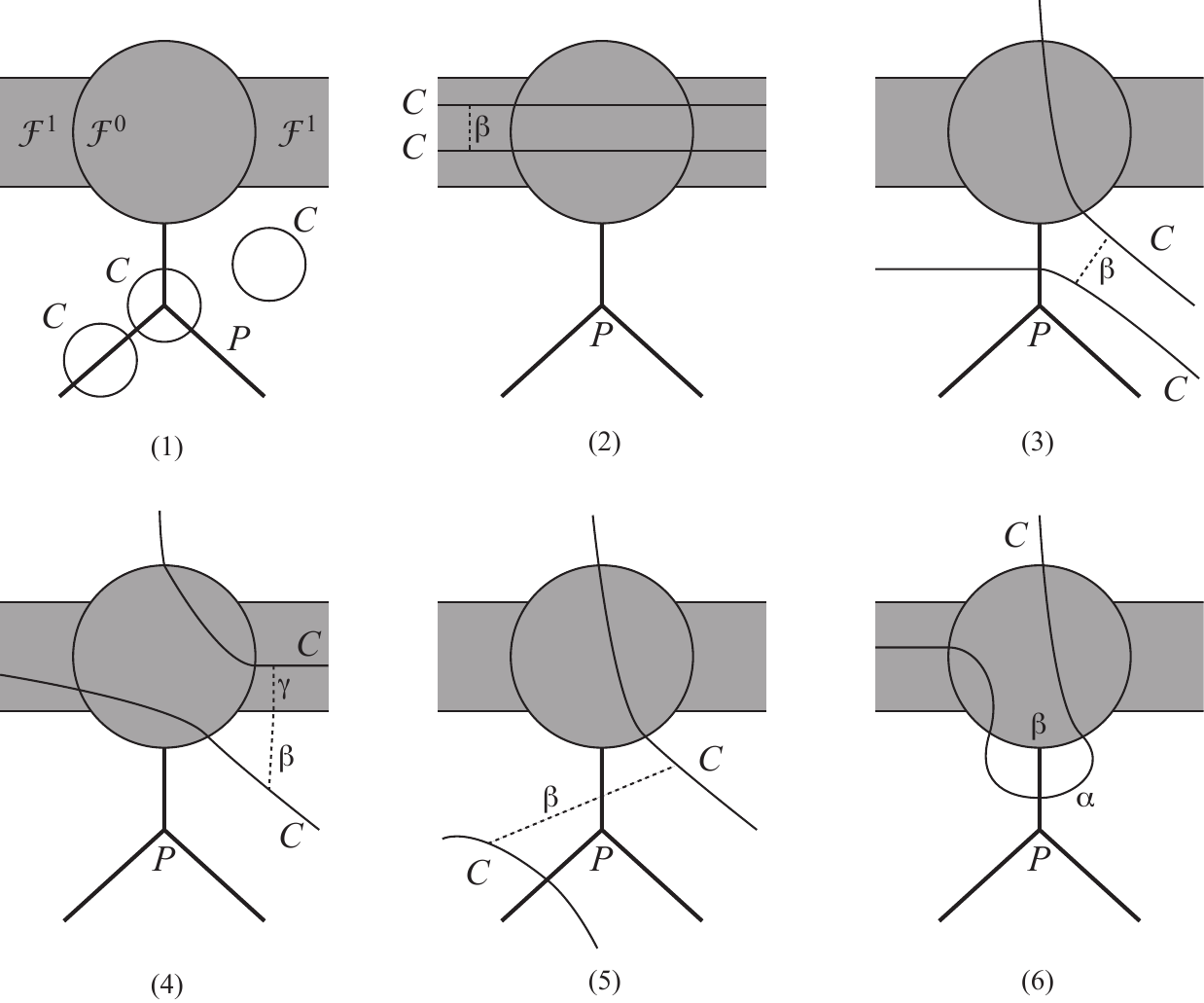}
\caption{Configurations of the curve $C$ in $\partial \calH^0$ that violate normality. In cases (2)-(6), arcs $\alpha$, $\beta$ or $\gamma$ are shown that are part of the boundary of a simplifying disc.} \label{Fig:NormalConditions}
\end{figure}

\begin{definition}
\label{Def:NormalSurface}
Let $\mathcal{ H}$ be a handle structure for a 3-manifold $M$ with boundary pattern $P$.
A surface $S$ properly embedded in $M$ is \emph{normal} with respect to $\mathcal{ H}$ if it is standard and
each component of $S \cap \partial \mathcal{ H}^0$ is normal.
\end{definition}

In the case where the pattern $P$ is empty, Definition \ref{Def:Normal} becomes the usual definition of a normal surface in
a handle structure, for example as in \cite[Definition 3.4.1]{Matveev}. Most of the conditions in Definition \ref{Def:Normal} are natural
generalisations to the case of general pattern. The possible exception is (6), which seems somewhat ad hoc. However,
it will be used to establish that manifolds appearing
in a hierarchy embed in the original manifold in a `simple' way. (See Theorem \ref{Thm:RemainUniformType}.)

There is a way of understanding the definition of normality in terms of simplifying discs, as follows.

\begin{definition}
Let $\calH$ be a handle structure for a 3-manifold $M$ with boundary pattern $P$. Let $S$ be a standard surface properly embedded in $M$. Then a \emph{simplifying disc} for $S$ is 
a disc $D$ lying in a 0-handle $H_0$ of $\calH$ satisfying one of the following:
\begin{enumerate}
\item $\partial D$ is the concatenation of two arcs $\alpha \subset S$ and $\beta \subset \partial H_0$, where $\alpha = D \cap S$, and $\beta = D \cap \partial H_0$ is an arc on the boundary of a 2-handle of $\calH$ that is vertical in its product structure (see Figure \ref{Fig:SimplifyingDisc} and Figure \ref{Fig:NormalConditions} (2));
\item $\partial D$ is disjoint from $P$ and is the concatenation of three arcs $\alpha \subset S$, $\beta \subset \partial H_0 \cap \partial M$ and $\gamma \subset \partial H_0 \cap \calH^2$, where $\alpha = D \cap S$, $\beta = D \cap \partial M$ and $\gamma$ is an arc on the boundary of a 2-handle of $\calH$ that is vertical in its product structure (see Figure \ref{Fig:NormalConditions} (4));
\item $\partial D$ intersects $P$ at most once and is the concatenation of two arcs $\alpha \subset S$ and $\beta \subset \partial H_0$, where $\alpha = D \cap S$, and $\beta = D \cap \partial H_0$ is an arc in $\partial M$ that is not parallel to an arc of $S \cap \partial H_0 \cap \partial M$ intersecting $P$ at most once (Figure \ref{Fig:NormalConditions} (3) and (5));
\item $D$ is a subset of $\partial \calH^0$ and $\partial D$ is the concatenation of two arcs $\alpha = D \cap S$ and $\beta = D \cap \calF^0$, and with the properties that $D$ is otherwise disjoint from $\calF$ and that $D$ intersects $P$ in at most two arcs, each of which runs from $\partial \mathcal{F}^0$ to $\alpha$ (see Figure \ref{Fig:NormalConditions} (6)).
\end{enumerate}
\end{definition}

\begin{figure}[h]
\centering
\includegraphics[width=0.6\textwidth]{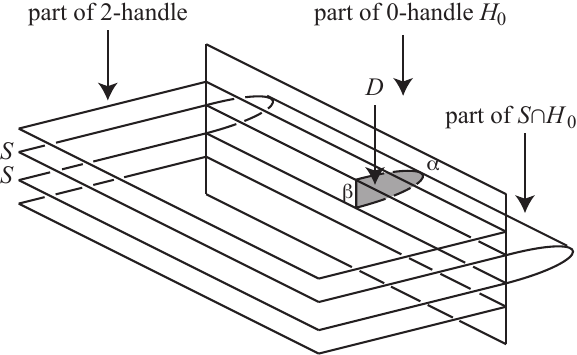}
\caption{A simplifying disc $D$} \label{Fig:SimplifyingDisc}
\end{figure}

The following is immediate from the definition of normality.

\begin{lemma}
\label{Lem:NormalIffNoSimplifying}
Let $\calH$ be a handle structure for a 3-manifold $M$ with boundary pattern $P$. Then a surface $S$ properly embedded in $M$ is normal if and only if the following all hold:
\begin{enumerate}
\item $S$ is standard;
\item $S$ has no simplifying disc;
\item no component of $S$ is a disc lying in a 0-handle that is parallel to a disc in $\partial M \cap \partial \calH^0$ intersecting $P$ in the empty set, an arc or tripod.
\end{enumerate}
\end{lemma}

We will want to place properly embedded surfaces into normal form. In general, we may need to make modifications
to the surface in order to do this. These modifications will not increase the following quantities.

\begin{definition}
\label{Def:Weight}
Let $\calH$ be a handle structure of a compact orientable 3-manifold $M$, and let $S$ be a standard surface properly embedded in $M$. The \emph{weight} of $S$ is the number of components of $S \cap \calH^2$. The \emph{extended weight} of $S$ is $|S \cap \calH^2| + |S \cap \calH^1|$.
\end{definition}

\begin{lemma}[Normalising a surface]
\label{Lem:NormalForm}
Let $M$ be a compact orientable $3$-manifold with boundary pattern $P$, and let $\calH$ be a handle structure for $(M,P)$.
Suppose that for any simple closed curve $C$ lying in the boundary of a 0-handle of $\calH$, that is disjoint from $\calF$, and that intersects $P$
at most 3 times, $C$ bounds a disc in $\partial \calH^0 - \calF$ intersecting $P$ in the empty set, an arc or a tripod.
Let $S$ be a compact orientable surface properly embedded in $M$ 
transverse to $P$ that is non-trivial in $H_2(M, \partial M)$. Then, after applying a sequence of isotopies preserving $P$, compressions and pattern-compressions, some component of the resulting surface is normal and either non-separating or a violating disc. Moreover, if $S$ is standard in $\mathcal{H}$, then the weight and extended weight of the resulting normal surface are at most the weight and extended weight of $S$.
\end{lemma}

\begin{proof}
We may readily ensure that $S$ is standard. In particular, if $S$ intersects some 0-handle $H_0$ in anything other than discs, then we may modify it, keeping $S \cap \partial H_0$ unchanged but ensuring that each of these curves bounds a disc in $H_0$. This process is achieved by a sequence of isotopies and compressions.

We may further assume that $S$ has no component that is a boundary-parallel disc, as we could remove this component and the surface would still be homologically non-trivial. 

Suppose that $S$ is not normal. Then by Lemma \ref{Lem:NormalIffNoSimplifying}, $S$ has a simplifying disc. The first and second types of simplifying disc can be used to pattern-isotope $S$ to reduce its weight without increasing $|S \cap \mathcal{H}^1|$. The third type of simplifying disc is a potential pattern-compression disc for $S$. It satisfies (1) and (2) in Definition \ref{Def:PatternCompDisc}, but possibly not (3).
We can pattern-compress $S$ along this disc. If $S$ was a violating disc, then some component of the resulting surface is again a violating disc, and we work with this instead. If $S$ was non-separating, but the pattern-compression produces a violating disc, then we work with this violating disc instead. However, if the pattern-compression creates a component that is a boundary parallel disc, but not a violating disc, then discard it. The resulting surface has no greater weight, extended weight or $\chi_p$ than $S$. Moreover, either the extended weight has strictly decreased or $\chi_p$ has strictly decreased. On the other hand, if the pattern-compression does not create a violating disc or a boundary-parallel disc, then some component of the resulting surface is non-separating, has smaller $\chi_p$ and no greater weight or extended weight. The fourth type of simplifying disc specifies an isotopy of $S$, which reduces the number of components of $S \cap \calH^1$ and that leaves the weight of $S$ unchanged. At each stage, we are either reducing $\chi_p$ or leaving it unchanged and reducing extended weight. So this process terminates. The resulting surface is normal, connected, and orientable. Moreover, it is either non-separating or a violating disc.
\end{proof}

\section{Decomposing a handle structure along a normal surface}
\label{Sec:DecomposingHS}

We now wish to show that when $(M,P)$ is decomposed along a normal surface, the resulting manifold is, in some sense, simpler than the original one. The notion of complexity that goes down is defined in terms of handle structures. Its definition needs to take account of the boundary pattern, as follows.

\begin{definition}
\label{Def:Index}
Let $\mathcal{H}$ be a handle structure for $(M,P)$, and let $\mathcal{F}(\mathcal{H})$ be its characteristic surface.
For a 0-handle $H$ of $\mathcal{H}$, its \emph{index} $I(H)$ is defined to be the sum of $\chi_p(F)$, over all components $F$ of $H \cap \mathcal{F}(\mathcal{H})$.
\end{definition}

Note that $\chi_p(F)$ can be computed from the 0-handles $D$ of $F$, as follows. 
Define the \emph{index} of $D$ to be 
$$I(D) = -4 + 2|D \cap \mathcal{F}^1| + |D \cap P|.$$
Then $\chi_p(F)$ is $\sum_D I(D)$, as $D$ ranges over each 0-handle of $F$.

\begin{lemma}[Behaviour of index under decomposition]
\label{Lem:DecompIndex}
Let $\mathcal{H}'$ be obtained from $\mathcal{ H}$ by decomposing along 
a standard surface $S$. Let $D$ be a 0-handle of $\mathcal{ F}$. Then the sum, over all components
$D'$ of $D \cap \mathcal{ H}'$, of $I(D')$ is equal to $I(D)$. Hence, for any component $F$ of $\mathcal{F}(\mathcal{H})$,
the sum, over all components $F'$ of $F \cap \mathcal{ H}'$, of $\chi_p(F')$ is equal to $\chi_p(F)$.
\end{lemma}

\begin{proof}
Let $D$ be a 0-handle of $\calF$. Its intersection with $S$ is a collection of properly embedded arcs. Say that there are $n$ of them. (See Figure \ref{Fig:DecomposeF}.)
Then decomposing $D$ along these arcs gives $n+1$ discs and so this changes the index by $-4n$. The endpoint of each arc either lies in $\mathcal{F}^1(\mathcal{H})$ or in $\partial M$. In the former case, this component of $\mathcal{F}^1(\mathcal{H})$ is decomposed along the extension of the arc into $\mathcal{F}^1(\mathcal{H})$. Hence this gives an increase of $2$ to the index. If the endpoint of the arc lies in $\partial M$, then we get two new points of intersection between $D \cap \mathcal{H}'$ and the boundary pattern. Hence, again this increases the index by $2$. So each arc leads to an increase in index by $4$. Therefore, in total, the index is unchanged.
\end{proof}

\begin{figure}[h]
\centering
\includegraphics[width=0.8\textwidth]{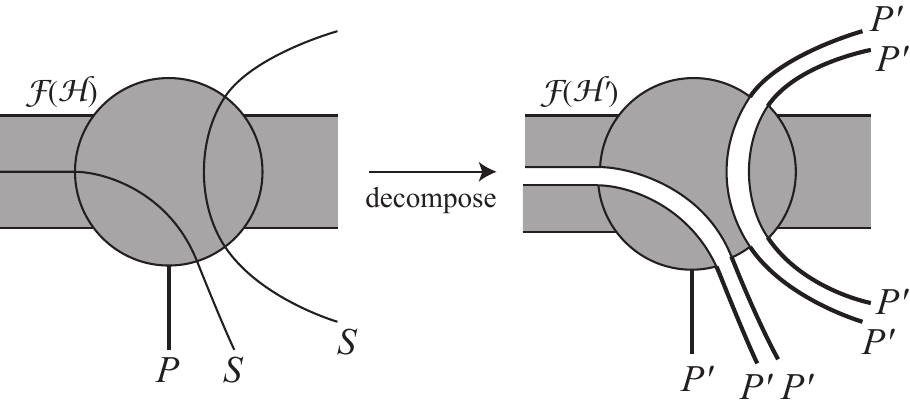}
\caption{Decomposition of $\mathcal{F}(\mathcal{H})$ along $S$} \label{Fig:DecomposeF}
\end{figure}

We will mostly want to deal with handle structures with the following property.

\begin{definition}
A handle structure $\mathcal{H}$ for $(M,P)$ is \emph{admissible} if the following both hold:
\begin{enumerate}
\item each 0-handle of $\mathcal{F}$ either has positive index or has zero index and is disjoint from $P$;
\item for any simple closed curve $C$ lying in the boundary of a 0-handle of $\calH$, that is disjoint from $\calF$, and that intersects $P$
at most 3 times, $C$ bounds a disc in $\partial \calH^0 - \calF$ intersecting $P$ in the empty set, an arc or a tripod.
\end{enumerate}
\end{definition}

\begin{figure}[h]
\centering
\includegraphics[width=0.8\textwidth]{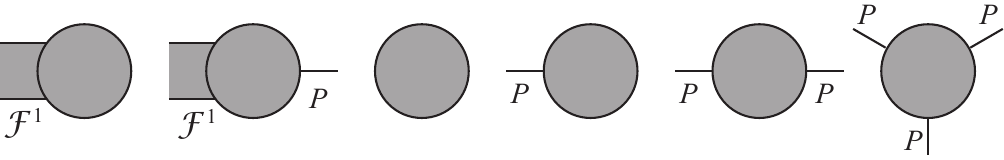}
\caption{The possible 0-handles of $\mathcal{F}$ with negative index} \label{Fig:NegIndex}
\end{figure}

\begin{figure}[h]
\centering
\includegraphics[width=0.6\textwidth]{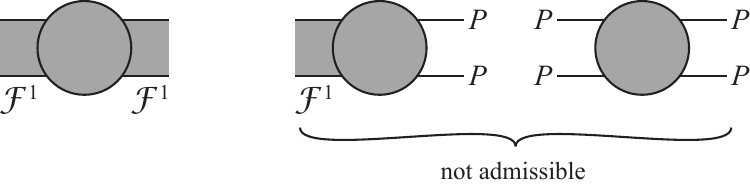}
\caption{The possible 0-handles of $\mathcal{F}$ with zero index} \label{Fig:ZeroIndex}
\end{figure}

\begin{proposition}[{\sc Create admissible handle structure}]
\label{Prop:CreateAdmissibleHS}
Let $(M,P)$ be a compact orientable irreducible 3-manifold with a boundary pattern, and let $\mathcal{H}$ be a handle structure for $(M,P)$ with no 3-handles. Then there is an algorithm that either creates an admissible handle structure for a manifold $(M',P')$ that is obtained from $(M,P)$ by decomposing along disjoint non-trivial squares or provides a violating disc for $(M,P)$.
\end{proposition}

\begin{proof}
The procedure that we will follow will not increase the index of any 0-handle of $\mathcal{F}$. It will strictly decrease the number of 1-handles of $\mathcal{H}$ at each stage, and so it is guaranteed to terminate. Suppose that there is a 0-handle $D$ of $\mathcal{F}$ that either has negative index or has zero index and intersects $P$. Then $D$ has at most one arc of intersection with $\mathcal{F}^1$. (See Figures \ref{Fig:NegIndex} and \ref{Fig:ZeroIndex}.)

Suppose that it has exactly one such arc. Then $D$ lies in a 1-handle $H_1$ of $\mathcal{H}$ and the arc lies in a 2-handle $H_2$ of $\mathcal{H}$, and we may remove $H_1 \cup H_2$ from the handle structure. There may be one or two arcs of $P$ running along $H_1$. If so, then we need to replace each component of $\mathcal{F}^1 \cap H_2$ with one or two arcs of $P$.

Suppose now that $D$ is disjoint from $\mathcal{F}^1$. Then it is properly embedded in $M$. If $D$ has negative index, it intersects $P$ in at most $3$ points. We can then determine whether $\partial D$ bounds a disc $D'$ in $\partial M$ intersecting $P$ in either the empty set, an arc or a tripod. If so, $D \cup D'$ bounds a 3-ball in $M$, by irreducibility. We remove this submanifold bounded by $D \cup D'$, and we insert into $D$ the boundary pattern that was in $D'$. On the other hand, if $\partial D$ does not bound such a disc $D'$, then $D$ is a violating disc. It might not be normal, but in this case, we can find another properly embedded disc in $\calH^0$, disjoint from $\calF$, that also forms a violating disc and intersecting $P$ fewer times. Hence, we may assume that $D$ is normal, and the algorithm terminates with this disc. Suppose now that
$D$ has zero index, and hence is a square. If it is a trivial square, it bounds a disc $D'$ in $\partial M$ intersecting $P$ as in (1) or (2) of Lemma \ref{Lem:FourPuncDisc}. We can then remove the submanifold bounded by $D \cup D'$, and insert into $D$ the boundary pattern that was in $D'$. If $D$ is a non-trivial square, then we decompose along it.

Now suppose that there is a simple closed curve $C$ lying in the boundary of a 0-handle of $\calH$, that is disjoint from $\calF$, and that intersects $P$
at most 3 times, but that $C$ does not bound a disc in $\partial \calH^0 - \calF$ intersecting $P$ in the empty set, an arc or a tripod. There is
always such a curve $C$ when $\partial \calH^0 \cut (\calF \cup P)$ is not a union of discs and spheres. In that case, we pick $C$ to be boundary parallel
in a component of $\partial \calH^0 \cut (\calF \cup P)$ that is not a disc or sphere. We can determine
whether $C$ bounds a disc $D'$ in $\partial M$ intersecting $P$ in the empty set, an arc or a tripod. If there is no such disc $D'$ in $\partial M$,
then a disc properly embedded in $\calH^0$ bounded by $C$ is a violating disc. On the other hand, if there is such a disc $D'$,
then we can remove the submanifold bounded by $D \cup D'$ and insert into $D$ the boundary pattern that was in $D'$. 

We have performed a sequence of decompositions along squares. We need to explain why we can take these squares all to be disjoint. We also need to explain why a violating disc that is produced in the above process actually forms a violating disc for $(M,P)$. We note that when we decompose along a square, the copies of this disc in the resulting 3-manifold lie in 0-handles and are disjoint from handles of higher index. In particular, they are disjoint from the later decomposing squares. Moreover, they are disjoint from any later violating discs that are produced. Thus, the squares are all disjoint from each other, and any violating disc that is produced is actually a violating disc for $(M,P)$.

\end{proof}

\begin{remark}
\label{Rem:WhatIfMReducible}
In the above algorithm, it is assumed that $M$ is irreducible. If a compact orientable reducible 3-manifold $M$ is provided, then a version of the
algorithm still partially works as follows. At various points, discs $D$ and $D'$ are found with $\partial D' = \partial D$. When $M$ is reducible,
we do not automatically know that $D \cup D'$ bounds a 3-ball. However, we can easily check that it bounds a rational homology 3-ball, and if not then we can
correctly declare that $M$ is reducible. On the other hand, if the sphere does bound a rational homology 3-ball, then we can replace it with a 3-ball,
and thereby remove from $M$ a connected summand that is a rational homology 3-sphere. In fact, at this stage, the algorithm then removes this ball, and then continues. Hence, with this modification, we obtain either a correct declaration that $M$ is reducible, or a violating disc for $(M,P)$, or an admissible handle structure on a 3-manifold $M'$ with boundary pattern $P'$ that is obtained from $(M,P)$ by decomposing along disjoint non-trivial squares and removing closed rational homology 3-sphere connected summands.
\end{remark}

One of the reasons why we work with admissible handle structures is that they behave well when they are decomposed along a normal surface, as follows.

\begin{lemma}[Decomposing an admissible handle structure]
\label{Lem:AdmissibleDecomposition}
Let $(M,P)$ be a compact orientable 3-manifold with a boundary pattern, and let $\mathcal{H}$ be an admissible handle structure for $(M,P)$. Let $S$ be a normal surface properly embedded in $M$, and let $\mathcal{H}'$ be the handle structure on the manifold $(M', P')$ obtained by cutting along $S$. Then each 0-handle of $\mathcal{F}(\mathcal{H}')$ has non-negative index, and each component of $\partial (\calH')^0 - (\calF(\calH') \cup P')$ is a disc or sphere.
\end{lemma}

\begin{proof}
Suppose that $D'$ is a 0-handle of $\mathcal{F}(\mathcal{H}')$ with negative index. The possibilities for $D'$ are shown in Figure \ref{Fig:NegIndex}. 

Suppose first that $D'$ intersects $\mathcal{F}^1(\mathcal{H}')$ in a single arc and is disjoint from $P'$. If $\partial D'$ is disjoint from $S$, then $D'$ forms a 0-handle of $\calF(\calH)$ that violates the assumption that $\calH$ is admissible. Therefore, $\partial D'$ must intersect $S$ and this intersection must be an arc, because $\partial D'$ is disjoint from $P'$. But this implies that $S$ violated (2) in the definition of normality. 

Suppose now that $D'$ intersects $\mathcal{F}^1(\mathcal{H}')$ in a single arc and intersects $P'$ in a point. Again, $\partial D'$ cannot be disjoint from $S$. It must therefore intersect $S$ in an arc. But this implies that $S$ violated (4) in the definition of normality.

Suppose that $D'$ misses $\mathcal{F}^1(\mathcal{H}')$. Now $D'$ is a subset of a 0-handle $D$ of $\mathcal{F}^0(\mathcal{H})$, and its boundary alternates between arcs of $S \cap D$ and arcs lying in $\partial D$. Since $D'$ intersects the pattern at most three times, there can be at most three such arcs. But in fact, there must be an even number of them. There must be at least one arc of $S \cap D$ lying in $\partial D'$, as otherwise $D' = D$, contradicting the assumption that $\mathcal{H}$ is admissible. Hence, $D' \cap S$ is an outermost arc in $D$. The remainder of $\partial D'$ is an arc $A$ in $\partial D$ intersecting $P$ at most once. The arcs of $S \cap (\partial \mathcal{H}^0 \cut (\mathcal{F} \cup P))$ incident to $\partial D'$ lie in the same component of $\partial \mathcal{H}^0 \cut (\mathcal{F} \cup P)$ or in adjacent components of $\partial \mathcal{H}^0 \cut (\mathcal{F} \cup P)$. Hence, by (3) and (5) in the definition of normality, they are either equal or they meet at a point of $P$. So the endpoints of $D' \cap S$ are joined by an arc of $S \cap (\partial \calH^0 \cut \calF)$. The union of this latter arc and $A$ is a simple closed curve $C$ that intersects $P$ in at most two points. So by admissibility, $C$ bounds a disc in $\partial \calH^0$ intersecting $P$ in an arc or the empty set. But then (6) in the definition of normality is contradicted.

Finally, suppose that some component of $\partial (\calH')^0 - (\calF(\calH') \cup P')$ is not a disc or sphere. Then it contains a simple closed curve $C$ that does not bound a disc in $\partial (\calH')^0 - (\calF(\calH') \cup P')$. Because $C$ is disjoint from $P'$, it lies in $S$ or $\partial \calH^0 - (\calF(\calH) \cup P)$. But if it lies in $S$, then it lies in a normal disc of $S$, and so it does bound a disc in $\partial (\calH')^0 - (\calF(\calH') \cup P')$. On the other hand, if $C$ lies in a component of $\partial \calH^0 - (\calF(\calH) \cup P)$, then this is a disc because $\calH$ is admissible, and so (1) in the definition of normality is contradicted.
\end{proof}

In an admissible handle structure, we permit some 0-handles of $\mathcal{F}$ to have zero index. These give rise to the following types of handle of $\mathcal{H}$.

\begin{definition}
Let $\mathcal{H}$ be a handle structure for $(M,P)$. A handle $H$ of $\mathcal{H}$ is an \emph{interior parallelity handle} if one of the following holds:
\begin{enumerate} 
\item it is a 0-handle disjoint from $P$, each 0-handle of $\mathcal{F} \cap H$ has index zero and $\mathcal{F} \cap H$ is connected;
\item it is a 1-handle disjoint from $P$ and each 0-handle of $\mathcal{F} \cap H$ has index zero;
\item it is a 2-handle.
\end{enumerate}
Each interior parallelity handle can be given a product structure $D^2 \times I$ so that the intersection with any other 0-handle, 1-handle or 2-handle is of the form $\alpha \times I$ where $\alpha$ is a union of arcs in $\partial D^2$. These product structures are viewed as forming an $I$-bundle structure and this $I$-bundle structure can be chosen compatibly over all the parallelity handles. The union of the parallelity handles is the \emph{interior parallelity bundle} for $\mathcal{H}$.
\end{definition}

We will define more general versions of this structure in Section \ref{Sec:CertifyingEssentialHierarchy}, which will give rise to `parallelity bundles'. These important subsets of a handle structure have had numerous applications, for example \cite{Lackenby:Composite}. We recall the following terminology.

\begin{definition} 
Let $\mathcal{B}$ be an $I$-bundle over a surface $F$. Then its \emph{vertical boundary} $\partial_v \mathcal{B}$ is the $I$-bundle over $\partial F$, and its \emph{horizontal boundary} $\partial_h \mathcal{B}$ is the $(\partial I)$-bundle over $F$.
\end{definition}

When we decompose a handle structure $\mathcal{H}$ along a normal surface, we would like the resulting handle structure $\mathcal{H}'$ to be simpler than $\mathcal{H}$, in some sense. The measure of complexity that we will use is as follows.

\begin{definition}
Define the \emph{q-complexity} $c_q(H)$ of a 0-handle $H$ of $\calH$ to be
$$\left ( \sum_F \max \{I(F), 0\} \right )^2,$$ where the sum is taken over all components $F$ of $H \cap \calF(\mathcal{H})$. Define the 
\emph{q-complexity} $c_q(\calH)$ of $\calH$ to be the sum of $c_q(H)$  over all 0-handles $H$ of $\calH$. Here, `q' stands for `quadratic'.
\end{definition}

\begin{proposition}[Decomposition and q-complexity]
\label{Prop:ComplexityDecreases}
Let $\mathcal{H}$ be an admissible handle structure for $(M,P)$. Let $S$ be a compact orientable connected non-separating properly embedded normal surface. 
Let $\mathcal{H}'$ be the handle structure obtained from $\mathcal{H}$ by cutting along $S$. Then $c_q(\mathcal{H}') \leq c_q(\mathcal{H})$. Moreover, if this is an equality, then $S$ is homologous in $H_2(M, \partial M)$ to a vertical surface in the interior parallelity bundle for $\mathcal{H}$.
\end{proposition}

\begin{proof}
Consider any 0-handle $H$ of $\calH$, and any component $F$ of $H \cap \calF({\calH})$. Since $\calH$ is admissible, $F$ has non-negative index. By Lemma \ref{Lem:DecompIndex}, the sum, over all components $F'$ of $F \cap \mathcal{ H}'$, of $I(F')$ is equal to $I(F)$. By Lemma \ref{Lem:AdmissibleDecomposition}, each component $F'$ has non-negative index. Hence, the sum, over all 0-handles $H'$ of $\calH'$ in $H$, of $c_q(H')$ is at most $c_q(H)$,
with equality if and only if at most one 0-handle of $H \cap \mathcal{H}'$ has positive q-complexity and the remaining 0-handles have zero q-complexity. We deduce that $c_q(\mathcal{H}') \leq c_q(\mathcal{H})$. 

Suppose that this is an equality. We will show that $S$ is homologous to a surface in $\calB$. Consider a component $S'$ of $S \cut \calB$. We will
define a transverse orientation on $S'$ and will show that $S'$ with this transverse orientation is homologous to a 2-chain in $\partial M \cup \calB$. Since $S'$ is connected, this transverse orientation is either the same as or opposite to the one it inherits from $S$. Hence, $S'$ is homologous in $H_2(M \cut \calB, \partial (M \cut \calB))$ to a 2-chain in $\calB$. Since this is true for each component of $S \cut \calB$, we deduce that $S$ is homologous to a 2-chain in $\calB$.

We first define the transverse orientation on the components of $S' \cap \mathcal{H}^0$. Each such component lies within a 0-handle $H$ of $\mathcal{H}$ not lying in $\calB$. This 0-handle cannot be disjoint from the 1-handles, since it would then form a 3-ball component of $M$, which cannot contain a properly embedded non-separating surface. It therefore has at least one 0-handle of $\mathcal{F}(\mathcal{H})$. It cannot be the case that all 0-handles of $H \cap \mathcal{F}(\mathcal{H})$ have zero index since $H$ would then be part of the interior parallelity bundle, contrary to assumption. So $H \cap \mathcal{F}(\mathcal{H})$ has positive index. This is the same index as a unique 0-handle $H'$ of $H \cap \mathcal{H}'$.
The remaining 0-handles of $H \cap \mathcal{H}'$ have zero index. We orient this component of $S' \cap \mathcal{H}^0$ away from $H'$. This transverse orientation extends over the components of $S' \cap \mathcal{H}^1$ as follows. Each such component lies within a 1-handle $H_1$ that has two components of intersection with $\mathcal{F}^0(\mathcal{H})$. Since $\mathcal{H}$ is admissible and $H_1$ does not lie in the interior parallelity bundle, these two components of $\mathcal{F}^0(\mathcal{H})$ have positive index. They are divided up by $S'$ into components of $\mathcal{F}^0(\mathcal{H}')$. In each of these two components of $\mathcal{F}^0(\mathcal{H})$, exactly one of these components of $\mathcal{F}^0(\mathcal{H}')$ has positive index. The adjacent discs of $S' \cap \mathcal{H}^0$ point away from this. Hence, we may extend the transverse orientation consistently over the components of $S' \cap \mathcal{H}^1$. Note that $S'$ is actually disjoint from the 2-handles, since these are part of the interior parallelity bundle.

We now show that $S'$ with this transverse orientation is homologous to a 2-chain in $\calB$. Each component of $S' \cap \mathcal{H}^i$ is a disc. It lies within a handle $H$ of $\mathcal{H}$. If we were to cut $H$ along this component of $S' \cap \mathcal{H}^i$, we would obtain two 3-balls. We say that the 3-ball component into which this component points is \emph{associated} with it. The union of these balls forms a 3-chain, the boundary of which is the union of $S'$ and a 2-chain in $\partial M \cup \calB$. Thus, $S'$ is equal in $H_2(M \cut \calB, \partial (M \cut \calB))$ to a 2-chain in $\calB$, and so
 $S$ is equal in $H_2(M, \partial M)$ to a 2-chain in $\calB$ with boundary in $\partial M$. By replacing this 2-chain with a 2-chain that is equal in $H_2(M, \partial M)$, we may assume that it is a surface properly embedded in $\calB$ disjoint from $\partial_v \calB$. Furthermore, we may assume that it is incompressible and also admits no boundary compression discs disjoint from $\partial_v \calB$. Each component of such a surface is either vertical or horizontal in $\calB$. But an orientable horizontal surface is homologically trivial. Hence, we have shown that $S$ is homologous to a vertical surface in $\calB$.
\end{proof}

\begin{lemma}[Cutting along a vertical surface]
\label{Lem:CutAlongVertical}
Let $(M,P)$ be a compact orientable 3-manifold with a boundary pattern. Let $\mathcal{H}$ be an admissible handle structure for $(M,P)$ and let $\mathcal{B}$ be its interior parallelity bundle. Let $S$ be an orientable vertical surface in $\mathcal{B}$ that is properly embedded and non-separating in $M$. Then $S$ may be placed in normal form and when $\mathcal{H}$ is cut along $S$, the resulting handle structure $\mathcal{H}'$ on the resulting 3-manifold $M'$ with boundary pattern $P'$ satisfies $c_q(\mathcal{H}') = c_q(\mathcal{H})$. No 0-handles of $\calF(\calH')$ have negative index. Each component of $\partial (\calH')^0 \cut (\calF(\calH') \cup P')$ is a disc. Moreover, the following quantities do not increase and at least one strictly decreases: the total number of components of $\mathcal{F}^1$ in the 0-handles of $\mathcal{H}$ with positive index, and the number of interior parallelity 0-handles.
\end{lemma}

\begin{proof}
The interior parallelity bundle is an $I$-bundle over a surface $F$. We give $F$ a handle structure as follows. Each $i$-handle of $\mathcal{B}$ vertically projects onto an $i$-handle of $F$. In addition, we create a 0-handle of $F$ for each component of $\partial_v \mathcal{B} \cap \mathcal{F}^0$ and a 1-handle of $F$ for each component of $\partial_v \mathcal{B} \cap \mathcal{F}^1$. We call the latter handles \emph{boundary handles}, since they form a regular neighbourhood of $\partial F$. They are added to ensure that we really do have a handle structure for $F$. Without them, we would have 1-handles or 2-handles not completely attached to lower index handles.

The vertical surface $S$ projects to a simple closed curve $C$ in $F$. We isotope $C$ so that it is normal in the handle structure of $F$. A corresponding isotopy of $S$ makes it normal. When $C$ runs over any boundary 0-handles of $F$, then the corresponding normal discs of $S$ lie in 0-handles of $\mathcal{H}$ that are not parallelity 0-handles and hence that have positive index. When we decompose along these discs, then this reduces the number of components of $\mathcal{F}^1$ in these handles. When $C$ runs through 0-handles in the interior of $F$, the corresponding normal discs of $S$ decompose interior parallelity 0-handles of $\mathcal{H}$. This reduces the number of such 0-handles.

By Proposition \ref{Prop:ComplexityDecreases}, $c_q(\mathcal{H}') \leq c_q(\mathcal{H})$. In fact, this is an equality because the effect on $\mathcal{F}$ of the decomposition is to remove some lines of index zero 0-handles and 1-handles and replace them by two parallel arcs of pattern. Hence for any 0-handle $H$ of $\mathcal{H}$, its index is unchanged by the decomposition.
\end{proof}

The following is easily verified.

\begin{lemma}
\label{Lem:ComplexitiesDecreaseInAdmissibleAlgorithm}
Let $\calH$ be a handle structure for a compact orientable 3-manifold $M$ with boundary pattern $P$. Suppose that $\calH$ has no 3-handles and  that no 0-handles of $\calF(\calH)$ have negative index.
Then the algorithm {\sc Create admissible handle structure} in Proposition \ref{Prop:CreateAdmissibleHS} does not increase any of the following quantities: the q-complexity of $\mathcal{H}$, and the total number of components of $\mathcal{F}^1$ in the 0-handles of $\mathcal{H}$ with positive q-complexity. Furthermore, if each component of $\partial \calH^0 \cut (\calF(\calH) \cup P)$ is a disc, then the number of interior parallelity 0-handles also does not increase.
\end{lemma}

We are now in a position to complete the proof of Theorem \ref{Thm:EssentialHierarchy}. 

\begin{proof} 
We have proved one direction of the theorem in Lemmas \ref{Lem:EssentialPullsBack} and \ref{Lem:EssentialIrredPullsBack}. Therefore,
let $(M,P)$ be a compact connected orientable irreducible 3-manifold with an essential boundary pattern. If $M$ is closed, then we are assuming that it contains a properly embedded orientable incompressible surface with no 2-sphere components. We must show that $(M,P)$ admits an essential hierarchy. 

If $M$ is closed, we can decompose along the hypothesised surface, and the resulting manifold is irreducible by Lemma \ref{Lem:IrreduciblePushesForward} and has non-empty boundary and essential (empty) pattern. So we may assume that $\partial M$ is non-empty.

Let $\mathcal{H}$ be a handle structure for $(M,P)$ with no 3-handles. By Proposition \ref{Prop:CreateAdmissibleHS}, there is a decomposition along non-trivial squares, after which the manifold (which we continue to denote by $(M,P)$) has an admissible handle structure. By Lemma \ref{Lem:PatternComp4Disc}, these squares are pattern-incompressible and they are automatically incompressible because they are discs. Hence, by Lemma \ref{Lem:EssentialPushesForward}, the resulting pattern is essential and by Lemma \ref{Lem:IrreduciblePushesForward}, the manifold is irreducible.

If $\partial M$ has any sphere components, these bound 3-balls. Thus, we may assume that $\partial M$ is not a union of spheres. Therefore, $b_1(M) > 0$, and so $M$ contains a connected properly embedded non-separating orientable surface. Let $S$ be such a surface with smallest pattern-complexity. By Corollary \ref{Cor:MinimalComplexityImpliesEssential}, $S$ is pattern-incompressible and incompressible. By Lemma \ref{Lem:NormalForm}, it may be placed into normal form. 

By Lemma \ref{Lem:EssentialPushesForward}, the manifold $M'$ obtained by decomposing along $S$ inherits an essential boundary pattern and by Lemma \ref{Lem:IrreduciblePushesForward}, it is irreducible. By Proposition \ref{Prop:ComplexityDecreases}, the handle structure $\mathcal{H}'$ that it inherits has q-complexity at most that of $\mathcal{H}$. Moreover, its q-complexity is strictly less unless $S$ is homologous to a vertical surface in the interior parallelity bundle. 

In the case where $S$ is homologous to a vertical surface in the interior parallelity bundle, we instead use Lemma \ref{Lem:CutAlongVertical} to decompose along such a surface. This leaves $c_q(\mathcal{H})$ unchanged. But it decreases the quantity $m(\mathcal{H})$ which is the sum of the total number of components of $\mathcal{F}^1$ in the 0-handles with positive q-complexity, and the number of interior parallelity 0-handles.

Thus, if we consider the ordered pair $(c_q(\mathcal{H}), m(\mathcal{H}))$
and we compare such pairs using lexicographical ordering, then we have reduced this pair. We now apply Proposition \ref{Prop:CreateAdmissibleHS} to turn $\mathcal{H}'$ into an admissible handle structure. By Lemma \ref{Lem:ComplexitiesDecreaseInAdmissibleAlgorithm}, this does not increase $c_q(\mathcal{H}')$ or $m(\mathcal{H}')$. Thus, we can repeat with this new handle structure. Since the complexity pairs are being reduced at each stage, this process must terminate, and so creates an essential hierarchy for $(M,P)$.
\end{proof}

\section{The algorithm to decide whether a pattern is essential}
\label{Sec:AlgorithmCompressible}

In this section, we present our algorithm required by Theorem \ref{Thm:AlgorithmEssentialPattern}. Thus, we are given a triangulation for a compact orientable irreducible 3-manifold $M$ in which $P$ is simplicial, and we must decide whether $P$ is essential. 

We may assume that $\partial M$ is non-empty, as otherwise, the pattern is trivially essential. The first step is to build a handle structure $\mathcal{H}$ for $(M,P)$. Since $\partial M$ is non-empty, we modify $\calH$ so that it contains no 3-handles. For if it contains a 3-handle, then there is some 2-handle with $\partial M$ on one side and a 3-handle on the other, and we may then remove this 2-handle and the incident 3-handle. Thus, we now assume that $\calH$ has no 3-handles.

In our algorithm, we will store several lists:
\begin{enumerate}
\item handle structures $\mathcal{ H}_1, \dots, \mathcal{ H}_\ell$ for 3-manifolds $M_1, \dots, M_\ell$ with boundary patterns $P_1 ,\dots, P_\ell$;
\item normal surfaces $S_1, \dots, S_{\ell-1}$, where each $S_i$ is a normal surface in $\mathcal{ H}_i$ and where $(M_{i+1}, P_{i+1}) = 
(M_i, P_i) \cut S_i$ with $\mathcal{ H}_{i+1}$ its induced handle structure.
\end{enumerate}

The algorithm is as follows:
\begin{enumerate}
\item Set $\mathcal{ H}_1$ to be the initial handle structure $\mathcal{ H}$ and set $\ell$ to be 0.
\item Apply the algorithm {\sc Create admissible handle structure} in Proposition \ref{Prop:CreateAdmissibleHS} to $\mathcal{H}_{\ell + 1}$. If this produces a violating disc $D$, go to Step (12).
\item If Proposition \ref{Prop:CreateAdmissibleHS} creates an admissible handle structure, then add any decomposing squares that have been used to the hierarchy.
Increase $\ell$ by $1$ (if no squares were used) or $2$ (if squares were used).
\item Using linear algebra compute $b_1(M_\ell)$. 
\item If $b_1(M_\ell) = 0$, then go to Step (9).
\item By reaching this step, $b_1(M_\ell)$ must be positive. If there is a vertical orientable connected surface in the interior parallelity bundle for $M_\ell$ that is non-separating in $M_\ell$, let $S_\ell$ be such a surface, isotoped so that it is normal as in Lemma \ref{Lem:CutAlongVertical}. Otherwise, let $S_\ell$ be any connected, non-separating, properly embedded, oriented  normal surface $S_\ell$ in $M_\ell$. (See Theorem \ref{Thm:BuildSurface} below for a simple and fast construction of such a surface.)
\item Is some component of $S_\ell$ a violating disc for $(M_\ell, P_\ell)$? If yes, then set $D$ to be this component and go to Step (12).
\item Set $(M_{\ell+1}, P_{\ell+1})$ to be $(M_\ell, P_\ell) \cut S_\ell$ and set $\mathcal{ H}_{\ell+1}$ to be its induced handle structure. Go to Step (2).
\item By reaching this step, $b_1(M_\ell)$ must be $0$, and so $\partial M_\ell$ must be a union of spheres. Examine whether the intersection of $P_\ell$ with each such sphere is connected. If not, then we obtain a violating disc. Then examine all isotopy classes of simple closed loops in $\partial M_\ell$ that intersect $P_\ell$ transversely in at most 3 points. (Equivalently, consider whether the intersection between $P_\ell$ and each component of $\partial M_\ell$ is connected. If it is, examine all closed loops of length at most 3 in the dual graph.) For each such simple closed curve, determine whether it bounds a violating disc.
\item If no violating disc is found, then $(M_\ell, P_\ell)$ is essential, and therefore by Theorem \ref{Thm:EssentialHierarchy}, 
so is $(M,P)$. In this situation, the algorithm terminates with a declaration that $P$ is essential and it outputs the hierarchy that has been constructed.
\item By reaching this step, $\partial M_\ell$ is a union of spheres and the algorithm has found a closed loop in $\partial M_\ell$ that bounds a violating disc $D$.
\item Is $D$ a violating disc for the initial manifold $(M,P)$? If so, then the algorithm terminates by declaring that $P$ is inessential and it outputs the violating disc $D$.
\item By reaching this step, $D$ must run over at least one surface in the hierarchy.
Let $S_j$ be the last surface in the hierarchy that intersects $D$. 
\item Is $D$ a compression disc or pattern-compression disc for $S_j$? If not, then $\partial D \cap S_j$ must in fact be a boundary-parallel arc in $S_j$, separating off a disc $D'$ that intersects the pattern $P_j$ at most once. Moreover, $D \cup D'$ must be a violating disc for $(M_j, P_j)$. In this case, replace $D$ by $D \cup D'$ and go to Step (12).
\item By reaching this step, $D$ forms a compression disc or pattern-compression disc for $S_j$ in $M_j$. Compress or pattern-compress $S_j$ along this disc to form a properly embedded surface. Perform isotopies and possibly compressions and pattern-compressions to this surface, and possibly discard components, to obtain either a compact orientable connected properly embedded normal surface $\overline S_j$ that is non-separating in $M_{j}$, or a violating disc for $P_j$. Replace
$S_j$ by $\overline S_j$. Discard all later surfaces in the hierarchy. Set $\ell$ to be equal to $j$. Go to Step (7).
\end{enumerate}

\begin{figure}[h]
\centering
\includegraphics[width=1.0\textwidth]{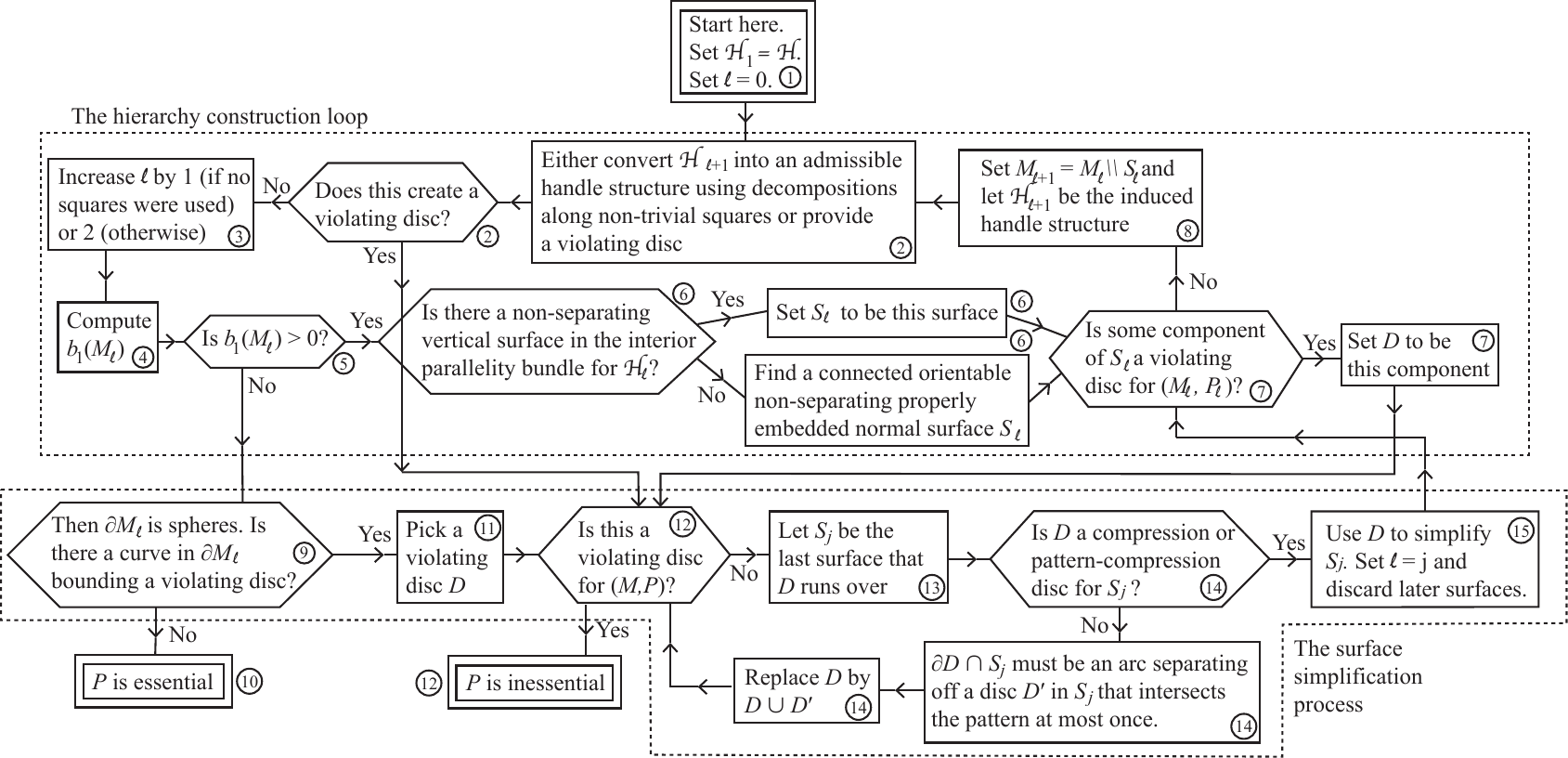}
\caption{The algorithm to determine whether a pattern is essential} \label{Fig:UnknotAlgorithm}
\end{figure}

\begin{theorem}[The algorithm works]
\label{Thm:UnknotAlgorithm1}
This algorithm terminates with the correct declaration either that $P$ is essential or that $P$ is inessential. Moreover when $P$ is essential, the algorithm outputs an essential hierarchy for $(M,P)$. When $P$ is inessential, it outputs a violating disc.
\end{theorem}

\begin{proof}
It is clear that if the algorithm terminates, then its output is correct. For if it terminates by declaring that $P$ is inessential, then it does so because it has found a violating disc. If it terminates by declaring that $P$ is essential, then it does so because it has found an essential hierarchy for $(M, P)$.

We must show that the algorithm terminates. We first note that it cannot remain in the hierarchy construction loop (Steps (2) to (8)) forever. Suppose that, on the contrary, it does. Each time we go through Step (6), we have an admissible handle structure $\mathcal{H}_\ell$. If there is no non-separating vertical surface in the interior parallelity bundle, then any connected normal orientable non-separating surface is chosen or a violating disc is found. If a violating disc is found, then we leave the hierarchy construction loop. If a connected normal orientable non-separating surface is chosen, then
by Proposition \ref{Prop:ComplexityDecreases}, cutting along this surface strictly reduces the q-complexity $c_q(\mathcal{H}_\ell)$. When there is a vertical surface, then by Lemma \ref{Lem:CutAlongVertical}, cutting along it leaves $c_q(\mathcal{H}_\ell)$ unchanged, but it reduces $m(\mathcal{H}_\ell)$. Then when the handle structure is made admissible in Step (2) using Proposition \ref{Prop:CreateAdmissibleHS}, none of these quantities increases. Hence, the ordered pair $(c_q(\mathcal{H}_\ell), m(\mathcal{H}_\ell))$ decreases at each stage. But this contradicts the fact that these pairs are well-ordered.

So, if the algorithm does not terminate, it must leave the hierarchy construction loop infinitely often. Index these occurrences by the natural numbers $i$. At the $i$th stage, let
$$(S_1^i, \dots, S_{\ell(i)}^i)$$ 
be the partial hierarchy that is currently stored. When the algorithm leaves the hierarchy construction loop and does not immediately then terminate, a violating disc $D$ is found. Furthermore, $D$ runs over at least one surface in the partial hierarchy, unless the algorithm immediately terminates. Let $S_j^i$ be the last surface that it runs over. Suppose first that $D$ is not a compression disc or pattern-compression disc for $S_j^i$. Then $D \cap S_j^i$ is either a simple closed curve that bounds a disc in $S_j^i$ or an arc in $S_j^i$ that separates off a disc in $S_j^i$ intersecting $P_j$ at most once. If $D \cap S_j^i$ is a simple closed curve bounding a disc $D'$ in $S_j$, then $D \cup D'$ is a sphere. By the irreducibility of $M$, this bounds a ball in $M$. The boundary of this ball is disjoint from the surfaces $S_1^i, \dots, S_{j-1}^i$. Hence, each of these surfaces must be disjoint from the ball, by our assumption that each surface is either connected and non-separating or a disjoint union of non-trivial squares. Thus, $D \cup D'$ bounds a ball in $M_j$. None of the later surfaces $S_{j+1}^i, \dots, S_{\ell(i)}^i$ can enter this ball, again by our assumption that each is either connected and non-separating or a disjoint union of non-trivial squares. Therefore, $D'$ lies in $M_{\ell(i)}$ and its interior is disjoint from $P_{\ell(i)}$. This contradicts the assumption that $D$ is a violating disc for $(M_{\ell(i)}, P_{\ell(i)})$. Suppose now that $D \cap S_j^i$ is an arc in $S_j^i$ that separates off a disc $D'$ in $S_j^i$ intersecting $P_j$ at most once. Then $D \cup D'$ is a disc properly embedded in $M_j$ intersecting the pattern at most twice. Suppose that this is not a violating disc for $(M_{j}, P_{j})$. Then $\partial (D \cup D')$ bounds a disc $D''$ in $\partial M_{j}$ that intersects $P_j$ either in a single arc or the empty set. The union of these three discs bounds a ball that none of the earlier surfaces or later surfaces can enter. Hence, $D' \cup D''$ contains an arc or tripod of $P_{\ell(i)}$. This contradicts the fact that $D$ is a violating disc for $(M_{\ell(i)}, P_{\ell(i)})$. Hence, the algorithm does produce a violating disc for $(M_j, P_j)$, as claimed in Step (14). 

Hence, the algorithm eventually produces a compression or pattern-compression disc for some $S_j^i$ and the later surfaces are discarded. This reduces $\chi_p(S_j^i)$. 
Now, $\chi_p(S_1^i)$ can decrease only finitely many times. Hence, there is some point after which $S_1^i$ remains unchanged. Let $S_1$ be this final surface. After this point, the manifold obtained by cutting along $S_1$ remains unchanged. We therefore consider the surfaces $S_2^i$ after this point. Again, these can change only finitely many times, and so there is some final surface $S_2$. Repeating this process, we obtain an infinite sequence of surfaces $S_1, S_2, \dots$. We must show that this cannot occur. Let $\calH_j$ be the handle structure obtained from $\calH$ by cutting along $S_1, \dots, S_j$.

Now we note that if some surface in one of these hierarchies is a disc that intersects the pattern four times, and at some stage we decrease its pattern-complexity, then this creates a violating disc. This is then used to simplify an earlier surface in the hierarchy. Thus, none of the surfaces $S_j$ in our infinite list is obtained from a square by pattern-compressing.

Thus, the surfaces $S_1, S_2, \dots$ and handle structures $\mathcal{H}_1, \mathcal{H}_2, \dots$ have the following property. Suppose that some $\mathcal{H}_j$ is admissible. If there is no non-separating vertical surface in the interior parallelity bundle, then $S_j$ is some connected normal orientable non-separating surface. On the other hand, if there is such a vertical surface, then $S_j$ is as in Lemma \ref{Lem:CutAlongVertical}. Then the next surfaces after $S_{j}$ are possible squares. The result is an admissible handle structure, and so on. As argued in the proof of Theorem \ref{Thm:EssentialHierarchy}, this is impossible, because the ordered pair $(c_q(\mathcal{H}_j), m(\mathcal{H}_j))$ decreases as we pass from one admissible handle structure to the next one.
\end{proof}

One reason why the above algorithm works so well is that, in a compact orientable 3-manifold $M$ with $b_1(M)>0$, it is very easy to construct
a connected, non-separating, properly embedded, oriented surface. We make this precise in the following theorem. This seems to be well known, but
it is so central to the efficiency of our algorithm that we give a proof.

\begin{theorem}[{\sc Build non-separating surface}]
\label{Thm:BuildSurface}
There is an algorithm that takes, as its input, a triangulation $\calT$ or an admissible handle structure $\calH$ for a compact orientable 3-manifold $M$ with boundary pattern $P$, and with $b_1(M) > 0$. In the case where $\partial M$ is non-empty, we also assume that $\calH$ has no 3-handles. The algorithm builds a connected, non-separating, properly embedded, oriented, normal surface $S$ with weight at most $c^{|\calT|}$ or extended weight at most $c^{|\calH|}$. Here, $c$ is a constant that is universal in the case of a triangulation. In the case of a handle structure, $c$ depends on the number $N$, which is the maximum, over each 0-handle of $\calH$, of the number of components of intersection between the 0-handle and each of the following: the 1-handles, the 2-handles, the vertices of $P$ and the edges of $P$. The running time is bounded above by a polynomial function of $|\calT|$ or $|\calH|$, where the polynomial is universal for triangulations and depends on $N$ in the case of handle structures.
\end{theorem}

\begin{proof}
Let us initially assume that $M$ is given via a triangulation $\calT$, since this case is more straightforward.

We are assuming that $H^1(M)$ is non-zero. Any non-zero element of this group can be represented by a simplicial 1-cocycle, and indeed such a 1-cocycle is unique if we require that its evaluation on each edge in some given maximal tree in the 1-skeleton is zero. So, one can set up a system of linear equations, with a variable for each oriented edge not in this tree, and an equation specifying that the total evaluation around each face is zero. Then any non-zero real solution to these equations represents a non-trivial element of $H^1(M; \mathbb{R})$. In fact, one can ensure that the coefficients of this element are rational, by imposing enough extra constraints specifying that some edges have zero evaluation, until the solution space is 1-dimensional, and then requiring that some specific coefficient is 1 (where this coefficient is chosen so that it is not identically zero in the 1-dimensional solution space). Then we can multiply through by  a suitable integer to obtain a non-zero solution with integer coefficients, which is a 1-cocycle representing a non-trivial element of $H^1(M)$. 

The size of the co-ordinates of this solution are bounded above by $3^{6|\calT|}$ for the following reason. The system of equations has at most $n = E - V + 1$ variables, where $E$ and $V$ are the number of edges and vertices of $\calT$. This is because a maximal tree has $V - 1$ edges in it. Note that $E - V + 1 \leq E \leq 6|\calT|$. There is an equation for each face, as well as an equation that requires that one specific coefficient is $1$. We also have extra conditions that require that some edges evaluate to zero. These equations are combined into a single matrix equation $Av = w$, where $A$ represents the above equations, $v$ is the solution, and $w$ is the vector with all zeroes except the final entry, which is $1$. We are assuming that the kernel of $A$ is zero, and hence the rank of $A$ is equal to the number of variables. There is a square submatrix $B$ of $A$ with the same number of columns of $A$ having the same rank. Thus $B$ is invertible. We have $Bv = w'$, where $w'$ consists of some of the entries of $w$, and so all entries are zero apart from one. Then $v = B^{-1}w'$. Note that $B^{-1} = (\mathrm{det}(B))^{-1} \mathrm{adj}(B)$, where $\mathrm{det}(B)$ is the determinant of $B$ and $\mathrm{adj}(B)$ is the adjugate matrix. We can obtain an integral solution by multiplying through by $\mathrm{det}(B)$. Each co-ordinate of this integral solution is an entry of the adjugate matrix, which is the determinant of a minor of $B$. The size of this is at most the product of the $l^2$ norms of the rows, which is at most $3^n \leq 3^{6|\calT|}$, as required. Note that the procedure used to produce this solution runs in polynomial time, since it involves computing $(n-1)^2$ determinants.

From this, one can build a properly embedded surface representing an element of $H_2(M, \partial M)$ dual to this 1-cocycle $c$, as follows. First, for each edge $e$ of the triangulation, place a collection of $|c(e)|$ points along the interior of the edge. These will be the intersection points between $e$ and the surface. Pick a circular ordering on the edges of each face. Since the evaluation of $c$ around the face is zero, we may join the points in the edges in pairs by a collection of disjoint normal arcs, so that for each arc, one endpoint lies on an edge with positive evaluation and one endpoint lies on an edge with negative evaluation. These arcs form the intersection between the surface and the face. In the boundary of each tetrahedron, these arcs combine to form a collection of normal curves, each with length three or four, and we may fill these in with normal triangles and squares. The result is a compact oriented properly embedded surface dual to $c$. Any component of this surface can be used as the required surface $S$. Such a component can be found using the algorithm of Agol, Hass and Thurston \cite[Corollary 17]{AgolHassThurston}.

We now consider the case where $M$ comes equipped with a handle structure $\calH$. When $\partial M$ is non-empty, we are assuming that $\calH$ has no 3-handles. This determines a cell structure on $M$, as follows. Each handle becomes a 3-cell. The 2-cells arise from the components of intersection between handles, as well as the components of intersection between a handle and $\partial M \cut P$. Some of these surfaces might not be discs, in which case we decompose them along arcs forming 1-cells, until they become discs. We will explain how to do this below. The remaining 1-cells arise from the components of intersection between any three handles, as well as the components of intersection between pairs of handles and $\partial M \cut P$, and also the components of intersection between a handle and an edge of $P$. Finally, the 0-cells arise as follows: the components of intersection between four handles; the components of intersection between any three handles and $\partial M \cut P$; the components of intersection between any two handles and $P$; the vertices of $P$. One arrangement when non-disc components of intersection between handles arises is when a 1-handle of $\calH$ is disjoint from $\calH^2$ and $P$. In that case, we insert a single vertical 1-cell running along the 1-handle. The other scenario where non-disc components arise is in the boundary of a 0-handle, since the intersection between the 0-handle and $\partial M \cut P$ or $\calH^3$ may have components other than discs. In that case, we insert properly embedded 1-cells into the planar surface, but without adding any new 0-cells. 

We now pick a maximal tree in the 1-skeleton of this cell structure. Our approach will depend on whether $M$ is closed or has non-empty boundary. When $M$ is closed, we pick a maximal tree in the 1-skeleton of each 3-handle. The union of these trees forms a forest in $M$. In each 1-handle disjoint from the 2-handles, we also add the new vertical 1-cell to this forest. We then extend this to a maximal tree in all of $M$. The effect of this is that, for each 1-cocycle $c$ that is zero on the edges in the maximal tree, the evaluation of each 1-cell in each 3-handle is zero. This is automatic for 1-cells in the maximal tree. But for any other 1-cell in a 3-handle, the evaluation of $c$ on that edge is equal to the evaluation of the loop that runs along the edge, then through the maximal tree back to the start of the edge. This loop lies entirely in the 3-handle, and so is homotopically trivial, and hence its evaluation is zero, as claimed. When $M$ has non-empty boundary, we pick a maximal tree in the 1-skeleton of each 2-handle and again any new vertical edges in a 1-handle are included in this forest. We also add each component of $P \cap \calH^1$ to this forest. Furthermore, for each disc component of $\calH^1 \cap \partial M$, we pick one of its arcs of intersection with $\calH^0$ and we declare that all but one of the 1-cells in that arc lies in the forest.
We again extend this to a maximal tree in $M$. This has the effect that the evaluation of each 1-cell in each 2-handle is actually zero.

As above, such cocycles $c$ are solutions to a system of linear equations. The variables correspond to the oriented 1-cells not in the maximal tree, and there is an equation for each 2-cell. The argument above also gives that there is a non-zero integral cocycle with the property that its value on each 1-cell has modulus that is bounded above by an exponential function of $|\calH|$, where the exponent only depends on the uniform type of $\calH$.

We now construct a surface $S$ that is dual to this 1-cocycle $c$. For each 1-cell $e$, we insert $|c(e)|$ points into its interior. These will be the points of intersection between $S$ and $e$. Each 2-cell is a polygon, and we have specified the intersection between $S$ and the boundary of this polygon. Since the total evaluation of the oriented edges in this boundary is zero, we can insert transversely oriented arcs into the interior of the polygon joining these points. Note that the 1-cells that are vertical in the 1-handles have zero evaluation, for the following reason. Each such 1-cell is either in the boundary of a 3-handle, or in the intersection between a 2-handle and $\partial M$, or is a component of $P \cap \calH^1$, or is an inserted 1-cell in a 1-handle disjoint from the 2-handles. For each such 1-cell, its evaluation is zero. Each 2-cell in the vertical boundary of a 1-handle has four sides, and we have just shown that the two vertical sides have zero evaluation. Hence, the remaining two sides have the same evaluation. So the arcs that we insert into such a 2-cell can be isotoped to be vertical. We can also ensure that the arcs in the two polygons at the ends of this 1-handle correspond via the product structure. In the boundary of each handle, these arcs patch together to form closed curves. We can insert discs into the handle bounded by these curves. The result is a transversely oriented, properly embedded surface dual to $c$. It is standard because it is disjoint from the 3-handles and respects the product structure of the other handles of $\calH$. For 1-handles, we have already arranged this. For 2-handles in a closed 3-manifold $M$, this follows from the fact that the surface misses the 3-handles and intersects each 1-cell in points of the same sign. When $M$ has non-empty boundary, the surface misses the 2-handles. Note also that all the points of intersection between the surface and any arc component of $\calH^0 \cap \calH^1 \cap \partial M$ have the same sign.

We need to arrange that this surface is in fact normal. We observe that when $M$ is closed, we have in fact already arranged this. When $M$ is closed, the only normality condition that is not automatic is (2) of Definition \ref{Def:Normal}. However, for each disc that we have inserted into a 0-handle, it runs over each component of $\calF^1$ at most once. For if it ran over a component of $\calF^1$ more than once, then there would be two arcs of intersection in that component of $\calF^1$ with opposite transverse orientations. However, all the arcs of intersection between $S$ and any component of $\calF^1$ are compatibly oriented.

We now deal with the case where $M$ has non-empty boundary. Each component of $S \cap \calH^0$ intersects each 1-cell in $\partial \calH^0$ at most once, since $S$ intersects each such 1-cell in points of the same sign. (This latter property is convenient at this stage, but we will not maintain it in the modifications that we make below.) So this results in finitely many types of disc in $S \cap \calH^0$. For each such disc type $D$, we define its \emph{complexity} to be $(|D \cap P| + 3|D \cap \calF^0| -2)^2$. Note that by the admissibility of $\calH$, $|D \cap P| + 3|D \cap \calF^0| -2 > 0$. The complexity of $D$ is bounded above by a constant depending only on the type of $\calH$. The \emph{complexity} of $S$ is equal to the sum of the complexities of its disc types. Note that in this sum, we only consider one disc of $S \cap \calH^0$ of each type in each 0-handle. Hence, the complexity is bounded above by a constant times $|\calH^0|$, where the constant only depends on the uniform type of $\calH$.

We now examine the normality conditions (1)-(6) from Definition \ref{Def:Normal}.  Condition (1) is automatically satisfied. In our setting, $S$ is disjoint from the 2-handles, and so (2) and (4) are satisfied. If a curve of $S \cap \partial \calH^0$ intersects a component of $\partial \calH^0 - (\calF \cup P)$ in more than one arc, contradicting (3), then we can pattern compress such a disc. We can pattern compress all discs of this type. For each such disc, this creates two discs $D_1$ and $D_2$. Discs of these types may already exist in $S$, but this does not affect our argument. 
The points of intersection between $D$ and $P$ are shared out between $D_1$ and $D_2$. The same is true for $D \cap \calF^0$. Hence, the total complexity of $S$ decreases. Note that when a pattern-compression is performed, then the class that $S$ represents in $H_2(M, \partial M)$ is unchanged. A similar argument works when (5) in the definition of normality is violated. We pattern compress the disc $D$ to discs $D_1$ and $D_2$. This introduces two points of intersection with $P$, but the complexity of these discs still decreases. Finally note that condition (6) must be satisfied, since if it were not, then we obtain two points of intersection between $S$ and a component of $\calH^1 \cap \calH^0 \cap \partial M$ that are oriented towards or away from each other, contrary to our construction. Thus we have constructed a normal surface. Any non-separating component of this is our required normal surface.
\end{proof}

\begin{remark} While this gives a simple construction of some surface $S$, mostly using only elementary linear algebra, it is not necessary to follow this construction in step (6) of the above algorithm to determine whether a boundary pattern is essential. Indeed, any connected, non-separating, properly embedded, oriented, normal surface can be used.
\end{remark}

\begin{remark} When Theorem \ref{Thm:BuildSurface} is applied to a handle structure $\calH$, both the upper bound on the extended weight of the surface and the running time of the algorithm depend on the quantity $N$. In most applications, $N$ is bounded above by a universal constant. For example, this is true when $\calH$ is dual to a triangulation of $M$ in which $P$ is simplicial. More generally, we will see in Section \ref{Sec:UniformType} that typically $\calH$ can be arranged to be of `uniform type' as one proceeds down a hierarchy. This term is defined in Definition \ref{Def:UniformType}, but it implies that $N$ is bounded above by a universal constant.
\end{remark}

\begin{remark}
\label{Rem:ReducibleInput}
An assumption in Theorem \ref{Thm:AlgorithmEssentialPattern} is that $M$ is irreducible.  
It is worth considering how the algorithm behaves if $M$ is a compact orientable, possibly reducible, 3-manifold. 
As explained in Remark \ref{Rem:WhatIfMReducible}, we can easily adjust the algorithm {\sc Create admissible handle structure} either to declare that $M$ is reducible or to provide a violating disc for $(M,P)$, or to give an admissible handle structure on a 3-manifold $M'$ with boundary pattern $P'$ that is obtained from $(M,P)$ by decomposing along disjoint non-trivial squares and removing closed rational homology 3-sphere connected summands. Furthermore, in Step (9), we discover that the manifold $M_\ell$ has boundary a union of spheres. However, these need not bound 3-ball components of $M_\ell$. Hence in Step (10) when no violating disc for $P_\ell$ is found, we cannot immediately use Theorem \ref{Thm:EssentialHierarchy} to deduce that $P$ is essential. However, we can instead use Lemma \ref{Lem:EssentialPullsBack} to deduce that $P$ is essential. Thus, Step (10) still completes with the correct declaration that $P$ is essential in $M$.
\end{remark}

\section{An example}
\label{Sec:Example}

We give an example of this algorithm in practice. We use it to show that the tripus 3-manifold has incompressible boundary. This manifold is shown in (i) of Figure \ref{Fig:TripusHierarchy}. There, a 3-ball (represented by the outer circle) is shown, containing a properly embedded graph. The tripus manifold is the exterior of this graph.
It is a famous example of a 3-manifold admitting a hyperbolic structure with totally geodesic boundary. Totally geodesic boundary components of a hyperbolic 3-manifold are always incompressible. However, the application of our algorithm gives a straightforward alternative proof of this.

\begin{figure}[h]
\centering
\includegraphics[width=4in]{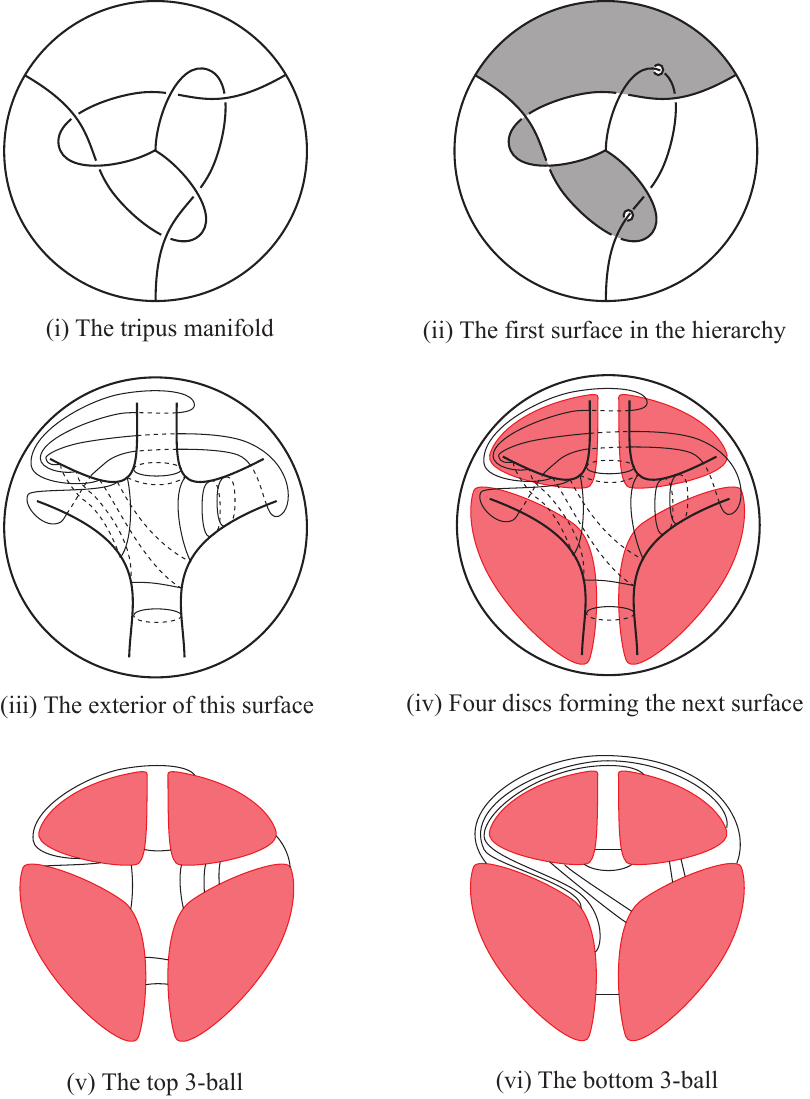}
\caption{An essential hierarchy for the tripus manifold} \label{Fig:TripusHierarchy}
\end{figure}

The first decomposing surface is a thrice-punctured sphere as shown in (ii) of Figure \ref{Fig:TripusHierarchy}. The exterior of this surface, with its boundary pattern, is shown in (iii). The next surface is the four discs shown in red in (iv). The exterior of these discs is two 3-balls. Their boundary patterns are shown in (v) and (vi). It is easily checked that they are essential. Hence, the empty boundary pattern for the tripus manifold is essential. In other words, the boundary is incompressible. Furthermore, the manifold is irreducible.

A skeptical reader may object that we have not quite followed the algorithm from Section \ref{Sec:AlgorithmCompressible}, in several ways. We have not specified an initial handle structure $\mathcal{H}$ for the manifold. Of course we could have done so, and we could have realised the first surface as normal, and we could have decomposed $\mathcal{H}$ along it, thereby forming a new handle structure $\mathcal{H}_2$, and so on. Although a computer implementation would have used handle structures, we have chosen to avoid them in this example, for ease of exposition.

We deviated from the algorithm in another way, because the second surface is a union of discs which is separating, whereas the surfaces in Step (6) are required to be connected and non-separating. However, we could instead have chosen three of the four discs, and decomposed along each of these in turn. These are connected and non-separating. The resulting manifold would have been a single 3-ball. Its boundary pattern would have been essential, by an application of Lemma \ref{Lem:EssentialPullsBack} and because decomposition along the fourth disc gives an essential pattern. We chose to decompose along four discs rather than three, again for ease of visualisation and exposition.

\begin{figure}[h]
\centering
\includegraphics[width=3.5in]{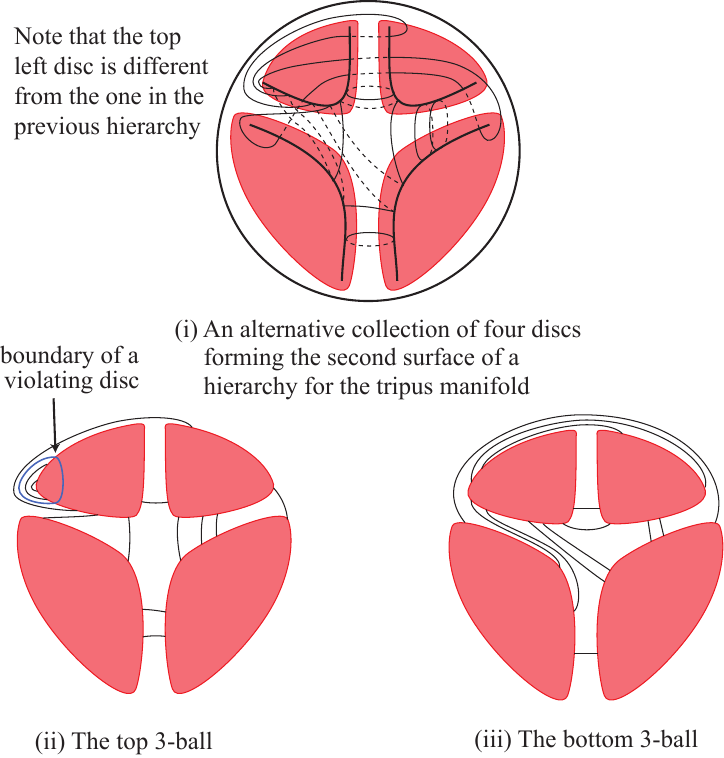}
\caption{An alternative decomposing surface. A curve bounding a violating disc is shown.} \label{Fig:TripusHierarchyAlternative}
\end{figure}

A reader may also object that we have been very fortuitous that we did not need to leave the hierarchy construction loop. So suppose instead that we had initially used a slightly different set of decomposing surfaces. Suppose that we used the same first surface, but that the second surface was a slightly different collection of discs $S_2$, as shown in (i) of Figure \ref{Fig:TripusHierarchyAlternative}. Then when we decompose along these discs, we get two 3-balls, with boundary patterns as shown in (ii) and (iii). The pattern in (ii) is inessential; a curve bounding a violating disc $D$ is shown. This would have been discovered in step (9) of the algorithm. It runs over one of the discs in $S_2$. When we pattern-compress $S_2$ along $D$, the result is five discs, one of which is boundary-parallel. Removing this from the collection, we obtain the four discs described in Figure \ref{Fig:TripusHierarchy} (iv). Hence, the algorithm would again terminate with the correct declaration that the boundary of the tripus manifold is incompressible.

\section{The algorithm to determine incompressibility}
\label{Sec:AlgorithmIncompressible}

In this section, we complete the proof of Theorem \ref{Thm:AlgorithmIncompressible}.

\begin{lemma}
\label{Lem:CompressibilityPreserved}
Let $S$ be a compact orientable surface properly embedded in a compact orientable 3-manifold $M$. Let $D$ be a compression disc for $\partial M$ that is disjoint from $S$. Then $S$ is incompressible in $M$ if and only if it is incompressible in $M \cut D$. Also, $S$ is boundary-incompressible in $M$ if and only if it is boundary-incompressible in $M \cut D$.
\end{lemma}

\begin{proof}
One direction is clear. A compression or boundary-compression disc for $S$ in $M \cut D$ can be slid off the copies of $D$ in $\partial (M \cut D)$. It then becomes
a compression or boundary-compression disc for $S$ in $M$.

Suppose that $D'$ is a compression or boundary-compression disc for $S$ in $M$. 
The intersection between this and $D$ is a collection of simple closed curves and properly embedded arcs.
Consider a simple closed curve that is innermost in $D$. It bounds a disc in $D$. We may surger $D'$ along this disc, and remove the resulting 2-sphere component. The result is a new compression or boundary-compression disc for $S$ that intersects $D$ a fewer number of times. Hence, we may assume that $D \cap D'$ contains no simple closed curves. When $D'$ is a compression disc, this implies that $D \cap D' = \emptyset$, since if there were an arc of intersection, its boundary would lie in $S$, whereas $D$ is assumed to be disjoint from $S$. So suppose that $D'$ is a boundary-compression disc for $S$. Now consider an arc of $D' \cap D$ that is outermost in $D$. It separates off a sub-disc in $D$ with interior disjoint from $D'$. We may surger $D'$ along this disc, giving two discs $D'_1$ and $D'_2$. One of these is a boundary-compression disc for $S$. Hence, again by minimising $|D \cap D'|$, we may assume that $D \cap D' = \emptyset$. But then $D'$ lies in $M \cut D$, and it forms a compression or boundary-compression disc for $S$ in $M \cut D$.
\end{proof}

\begin{theorem}[{\sc Determining incompressibility and boundary-incompressibility}] 
\label{Thm:DeterminingBoundaryIncompressible}
There is an algorithm that takes, as its input a compact orientable irreducible 3-manifold $M$ and a compact orientable properly embedded surface $S$, 
and determines whether $S$ is incompressible and boundary-incompressible. Furthermore, when $S$ is compressible or boundary-compressible and no component of $S$ lies in a 3-ball in $M$, the algorithm outputs a compression disc or a boundary-compression disc.
\end{theorem}

\begin{proof}
We may remove the sphere and disc components of $S$ without affecting whether $S$ is compressible or boundary-compressible. So suppose that $S$ has no sphere or disc components.
Give $M$ the empty boundary pattern, and let $P$ be the pattern that $M \cut S$ inherits. Run the algorithm
provided by Theorem \ref{Thm:AlgorithmEssentialPattern} on $(M \cut S, P)$. As in Remark \ref{Rem:ReducibleInput},
we allow the input manifold $M \cut S$ to be possibly reducible. If this declares that $M \cut S$ is reducible,
then any reducing sphere bounds a 3-ball in $M$, and we deduce that a component of $S$ lies in this 3-ball.
Since we have removed all the sphere components of $S$, $S$ must have been compressible, and we can terminate
the algorithm with a declaration to that effect.

So let us assume that the algorithm from Theorem \ref{Thm:AlgorithmEssentialPattern} does not declare that $M \cut S$
is reducible. It therefore establishes whether $P$ is essential. Furthermore, when it is inessential, the algorithm provides a violating disc.

First suppose that $P$ is essential. Then $S$ is incompressible and boundary-incompressible, by Lemma \ref{Lem:EssentialPullsBack}. Hence, in this case, our algorithm terminates with the declaration that $S$ is incompressible and boundary-incompressible.

Now suppose that $P$ has a violating disc $D$. Then $\partial D$ intersects $P$ an even number of times, since $P$ is where the copies of $S$ in $\partial (M \cut S)$ meet the parts of the boundary coming from $\partial M$. So, $\partial D$ intersects $P$ zero times or twice. If $\partial D$ intersects $P$ twice, then $D$ forms a potential boundary-compression disc for $S$. If this is an actual boundary-compression disc, then the algorithm terminates and provides this disc. On the other hand, if $D \cap S$ separates off a subdisc $D'$ from $S$, then $D \cup D'$ yields a violating disc for $P$ that is disjoint from $P$. So we may suppose that $\partial D$ is disjoint from $P$. It then lies in a copy of $S$ or in $\partial M \cut S$. In the former case, $D$ is a compression disc for $S$. Suppose that $\partial D$ lies in $\partial M \cut S$. If it forms a compression disc for $\partial M$, then we cut $M$ along $D$ and by Lemma \ref{Lem:CompressibilityPreserved}, $S$ is compressible or boundary-compressible in $M \cut D$ if and only if it was beforehand. Thus, we can work with this new manifold, which has boundary with higher Euler characteristic. Suppose that $D$ is not a compression disc for $\partial M$. Then $\partial D$ bounds a disc $D'$ in $\partial M$. The boundary of $S$ must intersect this disc, since $D$ is a violating disc. Pick an innermost curve of $\partial S \cap D'$ in $D'$. This bounds a disc in $D'$ that forms a compression disc for $S$, since we are assuming that $S$ has no disc components.
\end{proof}

\begin{theorem}[{\sc Determining incompressibility}] 
There is an algorithm that takes, as its input, a compact orientable irreducible 3-manifold $M$ and a compact orientable properly embedded surface $S$, 
and determines whether $S$ is incompressible. Furthermore, when $S$ is compressible and no component of $S$ lies in a 3-ball in $M$, the algorithm outputs a compression disc.
\end{theorem}

\begin{proof}
Again we may assume that $S$ has no sphere and disc components.
Give $M$ the empty boundary pattern, and let $P$ be the pattern that $M \cut S$ inherits. Now replace each curve of $P$ by two parallel curves, forming a new pattern $P'$. The algorithm given in Theorem \ref{Thm:AlgorithmEssentialPattern} is then applied to $(M \cut S, P')$. If it declares that $M \cut S$ is reducible, then $S$ was compressible. If $P'$ is essential, then $S$ is incompressible and the algorithm terminates with this declaration. So suppose that $P'$ has a violating disc $D$. The boundary of $M \cut S$ consists of parts coming from $S$ and parts coming from $\partial M$, and between these lie annuli lying between parallel curves of $P'$. Hence, $\partial D$ intersects $P'$ an even number of times. Moreover, if it intersects $P'$ twice, then one arc of $\partial D \cut P'$ is boundary parallel in an annular component of $\partial (M \cut S) \cut P'$. In that case, we may isotope $\partial D$ to form a violating disc disjoint from $P'$. So we may assume that $D$ is disjoint from $P'$. It either forms a compression disc for $S$ or it is disjoint from $S$. In the former case, the algorithm terminates and provides this disc $D$. In the latter case, it either forms a compression disc for $\partial M$, in which case we compress along it and continue, or it can be used to find a compression disc for $S$ lying in $\partial M$.
\end{proof}

\section{The number of steps}
\label{Sec:NumberSteps}

There are various steps in the algorithm 
from Theorem \ref{Thm:AlgorithmEssentialPattern}
that are not described with enough precision to be able to estimate their running time. So, we will initially content ourselves with an analysis of the number of times that the algorithm goes through any of the steps in Section \ref{Sec:AlgorithmCompressible}. We do this in terms of two quantities:
\begin{enumerate}
\item the maximal length $L$ of any of the hierarchies; thus the variable $\ell$ will be at most $L$ throughout;
\item the maximal pattern-complexity $g$ of any decomposing surface.
\end{enumerate}
At present, we will not attempt to bound these two quantities. But we can bound the number of steps in the algorithm in terms of $L$ and $g$.

When the algorithm is in the hierarchy construction loop, namely Steps (2)-(8), it continues until a violating disc is found or until it reaches a manifold $M_{\ell+1}$ with $b_1(M_{\ell+1}) = 0$. The number of times it can iterate this loop is at most $L$. By this stage, we have produced a list of surfaces $S_1, \dots, S_\ell$ with pattern-complexity $\chi_p(S_1), \dots, \chi_p(S_\ell)$. It may be the case that the final surface $S_\ell$ has negative pattern-complexity, in which case it is a violating disc (because it is assumed to be connected and non-separating), and we do not include it in this list. The algorithm then proceeds through the surface simplification process, namely Steps (9)-(15). If it does not terminate here, then we pattern-compress or compress some $S_j$ and thereby reduce its pattern-complexity by at least one. We discard the later surfaces and return to Step (7) to create new ones.

In order to find an upper bound on the number of times we can leave the hierarchy construction loop, we define the {\sl $(g,L)$-complexity} of the hierarchy $S_1, \dots, S_\ell$ to be 
$$\sum_{i=1}^\ell (\chi_p(S_i)+1) (g+2)^{L-i}.$$
In other words, we extend the list $\chi_p(S_1) + 1, \dots, \chi_p(S_\ell) + 1$ to one of length $L$ by adding zeroes, and then we view this $L$-tuple of integers as the digits of a number in base $g+2$. This number is the $(g,L)$-complexity. Importantly, it goes down by at least one between successive occasions where we leave the hierarchy construction loop. This is where we compress or pattern-compress a surface $S_j$ and we discard the later surfaces beyond $S_j$. When new ones are created to replace them, all we have is an upper bound of $g$ on their pattern-complexity. Thus, the $L$-tuple of pattern-complexities might change as follows:
\begin{align*}
& (\chi_p(S_1)+1, \dots, \chi_p(S_j)+1, \chi_p(S_{j+1})+1, \dots, \chi_p(S_L)+1) \\
\longrightarrow \quad &(\chi_p(S_1)+1, \dots, \chi_p(S_j), g+1, \dots, g+1).
\end{align*}
But if we view these $L$-tuples as the digits of a number in base $g+2$, we see that this number has gone down by at least 1.
Since the $(g,L)$-complexity is never more than $(g+2)^L$, we deduce that the number of times we leave the hierarchy construction loop is at most $(g+2)^L$. Hence, the number of times that we perform any step of the algorithm is at most $(L+1) (g+2)^L$. The extra factor $L+1$ is because we may go round the hierarchy construction loop at most $L$ times before leaving it. We summarise this as follows.

\begin{proposition}[Number of iterations]
\label{Prop:UpperBoundRunningTime}
The number of times that the above algorithm can iterate through any step is at most $(L+1) (g+2)^L$, where $L$ is the maximal length of any of the hierarchies and $g$ is the maximum pattern-complexity of any of the surfaces.
\end{proposition}

Given the formula $(L+1) (g+2)^L$ in the above proposition, $L$ is the crucial quantity that one needs to control in order to provide a good upper bound for the running time for the algorithm. In fact, in the sequel \cite{Lackenby:LinkHyperbolicity} to this paper, we will show that, given any non-split link $K$ in the 3-sphere, when the above algorithm is applied to the exterior of $K$, then $g$ can be arranged to be at most a quadratic function of the crossing number of $K$.

\section{A refined algorithm using interior parallelity bundles}
\label{Sec:AlgorithmParallelity}

We saw in Proposition \ref{Prop:UpperBoundRunningTime} that the quantity $L$, which provides an upper bound to the length of any of the partial hierarchies that we construct, is key to controlling the running time of the algorithm. With our current algorithm, it is not so straightforward to give good estimates on $L$. This is because the proof at the end of Section \ref{Sec:DecomposingHS} that the hierarchies terminate uses the ordered pairs $(c_q(\mathcal{H}), m(\mathcal{H}))$. Although lexicographical ordering on these pairs is enough to establish termination, it does not establish how quickly this happens. Now under many circumstances $c_q(\mathcal{H})$ strictly decreases. The quantity $m(\mathcal{H})$ only becomes relevant when $\mathcal{H}$ has a non-separating vertical surface in its interior parallelity bundle. Therefore, in this section, we show how we may remove every component of the interior parallelity bundle that is not an $I$-bundle over a disc. This will lead to a somewhat more complicated algorithm. 
But it will mean that we can provide a linear bound on $L$ in terms of the initial q-complexity of $\mathcal{H}$. 

Again, we may modify our handle structure $\calH$ so that it contains no 3-handles, and so we make this assumption throughout.

The first change that we make to the algorithm is as follows.

\medskip
\noindent \emph{First modification to the algorithm}: After Step (2), once the handle structure has been made admissible, we then do the following:
\begin{enumerate}
\item[$(2')$] Replace each component of the interior parallelity bundle that is an $I$-bundle over a disc by a single 2-handle.
\item[$(2'')$] Cut along the vertical boundary components of the remaining $I$-bundles and further decompose these using squares into 3-balls with essential boundary patterns. In the case where a component of the interior parallelity bundle is an entire component of the 3-manifold, decompose it along an essential vertical annulus, and then along squares into 3-balls with essential boundary patterns. Finally, discard all these 3-balls.
\end{enumerate}

This ensures that the resulting 3-manifold has interior parallelity bundle that consists solely of 2-handles and hence cannot contain a non-separating vertical surface. Therefore, by Proposition \ref{Prop:ComplexityDecreases}, the q-complexity $c_q(\calH_i)$ of the handle structure $\calH_i$ strictly decreases when we decompose along a connected orientable non-separating normal surface $S_i$.

However, this algorithm has an extra complication. In the previous algorithm described in Section \ref{Sec:AlgorithmCompressible}, if a non-trivial square was found to be pattern-compressible, this led to a violating disc. This in turn was used to simplify an earlier surface in the hierarchy. Hence, pattern-compressed squares never survive in the hierarchy. But what about the annuli that are used in step $(2'')$? What if these are compressed or pattern-compressed? We now explain how to deal with this situation.

\begin{proposition}[{\sc (Pattern-)compressible vertical boundary}]
\label{Prop:CompressAnnuli}
Let $M$ be a compact orientable irreducible 3-manifold other than a 3-ball, and let $P$ be a boundary pattern for $M$. Let $\mathcal{H}$ be an admissible handle structure for $(M,P)$ with no 3-handles. Let $B$ be a component of the interior parallelity bundle $\calB$ that is not an $I$-bundle over a disc, and let $A$ be a vertical boundary component of $B$. Suppose that $A$ has a pattern-compression or compression disc $E$. Then there is an algorithm that, given $\calH$, $A$ and $E$, either creates a handle structure for $(M,P)$ with smaller q-complexity or produces a violating disc for $(M,P)$.
\end{proposition}

\begin{proof}
Suppose first that $E$ is a pattern-compression disc. Pattern-compress $A$ along $E$ to form a properly embedded disc $D$. If this is a violating disc, then the algorithm produces this disc and stops. So, suppose that $D$ is not violating. Its boundary therefore bounds a disc $D'$ in $\partial M$ that intersects $P$ either in the empty set or a single arc. Since $M$ is irreducible, $D \cup D'$ bounds a 3-ball in $M$. If the arc $E \cap \partial M$ is disjoint from $D'$, then we deduce that $A$ is parallel to an annulus $A'$ in $\partial M$ that intersects $P$ either in the empty set or a single core curve. On the other hand, if $E \cap \partial M$ lies in $D'$, then we deduce that $A$ is compressible. Indeed one of the discs of $D' \cut \partial A$ forms a compression disc for $A$ when pushed a little into the interior of $M$. We will later examine the case where $A$ is compressible, and so we now suppose that $A$ is parallel to the annulus $A'$.

Consider first the case where the solid torus between $A$ and $A'$ contains $B$. Then we remove this solid torus from the handle structure, thereby removing this component of $\mathcal{B}$ and possibly others. If $A'$ contained a core curve of $P$, then we transfer this across to $A$. This has the effect of replacing some components of $\mathcal{F}^0(\mathcal{H})$ and $\mathcal{F}^1(\mathcal{H})$ by an arc of $P$, which does not increase the q-complexity of $\mathcal{H}$. Moreover, the solid torus cannot be an entire component of $B$, and so it must contain a 0-handle of $\mathcal{H}$ of positive index. So removing this solid torus strictly reduced $c_q(\mathcal{H})$.

So, suppose that the interior of $B$ is disjoint from this solid torus. If $B$ is an $I$-bundle over a disc, then this implies that $M$ is a 3-ball which is contrary to one of the hypotheses of the proposition. 
Suppose that $B$ is an $I$-bundle over a surface $F$ with a single boundary component other than a disc. This then contains a non-separating essential arc. The $I$-bundle over this arc, together with two meridian discs for the solid torus, forms a violating disc for $(M,P)$. Suppose now that $B$ is an $I$-bundle over a surface with more than one boundary component. Pick an arc in the base surface with exactly one endpoint on the boundary component that is the image of $A$. Let $A''$ be the vertical boundary component incident to the other endpoint of this arc. The $I$-bundle over this arc, together with a meridian disc for the solid torus, gives a pattern-compression disc for $A''$. If we pattern-compress along this disc and we create a violating disc, then the algorithm outputs this. If it does not create a violating disc, then again there are two cases. It might lead to a compression disc for $A''$, which we will consider later. If not, then $A''$ is parallel to an annulus in $\partial M$ intersecting $P$ either in the empty set or a single core curve. But now the region between $A''$ and this annulus contains $B$. Hence, we may remove this solid torus, thereby removing $B$ from the handle structure.

We now consider the case where $A$ has a compression disc $E$. This compresses $A$ to two discs $D_1$ and $D_2$. If one of $\partial D_1$ or $\partial D_2$ does not bound a disc in $\partial M$, then $D_1$ or $D_2$ is violating, and the algorithm outputs this disc. So we may suppose that $D_1$ and $D_2$ are parallel to discs $D'_1$ and $D'_2$ in $\partial M$. It may be that $D'_1$ and $D'_2$ are nested, in which case suppose that $D'_1 \subset D_2'$. It may also be the case that $D'_1$ contains boundary components of $\partial_h B$ other than $\partial D_1'$. But by changing our choice of $A$ if necessary, we may assume that this is not the case.
If one of $D'_1$ or $D'_2$ has non-empty intersection with the pattern, it forms a violating disc. So, suppose that $D'_1$ and $D'_2$ are disjoint from the pattern.

The first case we consider is where $D'_1$ and $D'_2$ are disjoint. Then $D'_1 \cup A \cup D'_2$ bounds a ball. This ball does not contain $B$ because $\partial_h B$ has no boundary components in $D'_1$, and $B$ is not an $I$-bundle over a disc. We may modify the handle structure in the ball by declaring that it is a single interior parallelity 0-handle. This reduces the q-complexity of the handle structure.

Suppose now that $D'_1$ and $D'_2$ are nested. Then $D'_1$ lies within $D'_2$, by our earlier assumption. The fact that $B$ is not an $I$-bundle over a disc implies that the interior of $D'_1$ is disjoint from $B$. Hence, $B$ lies in the region between $A$ and $D'_2 \cut D'_1$. We remove this region from the handle structure. This reduces the q-complexity of the handle structure.
\end{proof}

\begin{remark}
\label{Rem:DontNeedFullHS}
In order to apply {\sc (Pattern-)compressible vertical boundary}, we do not need the full handle structure $\calH$. Instead, we could have worked with the handle structure obtained by cutting along the vertical boundary of some components of the interior parallelity bundle that are not $I$-bundles over discs, and then discarding those $I$-bundles, but keeping the information of each base-surface type and its vertical attachment data, and
the identification of the surviving exterior annuli with their original vertical boundary
components. Then $E$ would be supplied as a disc in this cut-open manifold.
\end{remark}

Thus, the algorithm in Section \ref{Sec:AlgorithmCompressible} is also modified as follows. 

\medskip
\noindent \emph{Further modification to the algorithm}:
In Step (13), we again consider a violating disc $D$. This forms a compression or pattern-compression disc for some surface $S_j$ in the hierarchy. If $S_j$ is a square, then this creates a new violating disc and we proceed to Step (12). If $S_j$ is the vertical boundary of the interior parallelity bundle for $\mathcal{H}_j$, we apply the algorithm in Proposition \ref{Prop:CompressAnnuli}. This either produces a violating disc or it simplifies $\mathcal{H}_j$. If it produces a violating disc, we go to Step (12). If $\calH_j$ is simplified, then we set $\ell = j-1$ and go to Step (2). In all other cases, we proceed to Step (14).

\begin{proposition}[The refined algorithm works]
\label{Prop:RefinedAlgorithmTerminates}
Let $M$ be a compact orientable irreducible 3-manifold, and let $P$ be a boundary pattern.
The above algorithm terminates with the correct declaration either that $P$ is essential or that $P$ is inessential. 
Moreover, the length of any of the partial hierarchies produced by the above algorithm is at most $4c_q(\mathcal{H}) + 3$.
\end{proposition}

\begin{proof}
As in the proof of Theorem \ref{Thm:UnknotAlgorithm1}, it is clear that if the algorithm terminates, then its output is correct. 

We start by bounding the length of the partial hierarchies that the algorithm produces. Therefore consider one such partial hierarchy $S_1, \dots, S_\ell$, and let $\calH_i$ be the handle structure for the manifold $(M_i, P_i)$ containing $S_i$. We will show that $c_q(\calH_{i+4}) < c_q(\calH_i)$ for each $i$. This will establish that $\ell \leq 4 c_q(\calH) + 3$.

Certainly, $c_q(\calH_{i+1}) \leq c_q(\calH_i)$ by Proposition \ref{Prop:ComplexityDecreases} and Lemma \ref{Lem:ComplexitiesDecreaseInAdmissibleAlgorithm}. We must show that at least every 4 steps, this quantity strictly decreases. To establish this, we examine carefully how the algorithm works.

The first thing that the algorithm does is to make the initial handle structure admissible, using  {\sc Create admissible handle structure} in Proposition \ref{Prop:CreateAdmissibleHS}. It then modifies the handle structure to ensure that the interior parallelity bundle is just $I$-bundles over discs, or it finds a violating disc.

The algorithm then proceeds by iterating the hierarchy construction loop. At each stage through the loop, it cuts along a connected orientable non-separating normal surface, then possibly cuts along non-trivial squares to make the handle structure admissible, then cuts along vertical boundary components of the interior parallelity bundle and then vertical squares. After this stage, the interior parallelity bundle again consists of $I$-bundles over discs.

If at any stage, the algorithm leaves the hierarchy construction loop and does not immediately terminate, it finds a compression disc or pattern-compression disc for some $S_j$. When this $S_j$ is a connected non-separating normal surface, a component of the resulting surface is non-separating and it is normalised, and we work instead with that surface. On the other hand, when $S_j$ is a square or the vertical boundary components of the interior parallelity bundle, then $S_j$ is either used to produce a further violating disc or it is used to simplify the handle structure; in both cases, it is discarded. Hence, in the list of surfaces $S_1, \dots, S_\ell$, either our given $S_i$ is connected, normal and non-separating, or this is true for some $S_j$ with $i \leq j \leq i + 3$. Therefore, $c_q(\calH_{i+4}) < c_q(\calH_i)$, as claimed.

We now show that the algorithm terminates.
At the $i$th stage, let
$(S_1^i, \dots, S_{\ell(i)}^i)$
be the partial hierarchy that is currently stored. 
As in the proof of Theorem \ref{Thm:UnknotAlgorithm1}, the surface $S_1^i$ must eventually stabilise at some surface $S_1$. This is because each time that $S_1^i$ is modified, either $\chi_p(S_1^i)$ decreases (when a pattern-compression or compression is performed) or the q-complexity of the handle structure containing $S_1^i$ is reduced (as in {\sc (Pattern-)compressible vertical boundary}). We then consider the next surfaces $S_2^i$. Again they eventually stabilise at some surface $S_2$. This is for the same reason as above, plus the observation that when $S_2^i$ is used to create a violating disc, then either this leads to the algorithm terminating or it is used to reduce the pattern-complexity of a previous surface. However, the latter case does not arise because we are assuming that $S_1^i$ does not change any further. Repeating in this way, we create an infinite sequence of surfaces $S_1, S_2, \dots$. However, this is impossible because we have just established an upper bound on the length of any partial hierarchy created by the algorithm.
\end{proof}

\begin{remark} 
\label{Rem:MaybeReducible}
Again a minor variant of the above algorithm runs when a compact orientable, but possibly reducible, 3-manifold $M$ is provided as an input. In that case, it 
terminates with the correct declaration of one of the following: $P$ is essential, $P$ is inessential or $M$ is reducible.
\end{remark}

\section{Handle structures of uniform type}
\label{Sec:UniformType}

Our algorithm starts with a handle structure $\mathcal{H}$. We then decompose this handle structure along a sequence of surfaces. Suppose that each of the 0-handles of the initial handle structure $\mathcal{H}$ falls into one of finitely many `types'. For example, if $\mathcal{H}$ had been dual to an ideal triangulation, then for each 0-handle $H$, its intersection with $\mathcal{F}$ is a thickened version of the complete graph on 4 vertices. If $\mathcal{H}$ had been dual to a triangulation, then for each 0-handle $H$, $H \cap \mathcal{F}$ is a thickened version of a subgraph of the complete graph on 4 vertices. We want to maintain this local finiteness as we proceed down the hierarchy. In this section, we show how to achieve this.

\begin{definition}
Let $H$ and $H'$ be 0-handles in handle structures $\mathcal{H}$ and $\mathcal{H}'$ for $(M,P)$ and $(M', P')$. We say that they have the same \emph{type} if there is a homeomorphism $H \rightarrow H'$ sending $H \cap \mathcal{F}^0(\mathcal{H})$ homeomorphically to $H' \cap \mathcal{F}^0(\mathcal{H}')$, sending homeomorphically $H \cap \mathcal{F}^1(\mathcal{H})$ to $H' \cap \mathcal{F}^1(\mathcal{H}')$, and sending $H \cap P$ homeomorphically  to $H' \cap P'$.
\end{definition}

\begin{definition}
\label{Def:UniformType}
We say that a handle structure $\mathcal{H}$ has \emph{uniform type} if there is a finite list of handle types such that any 0-handle of $\mathcal{H}$ is one of these types.
\end{definition}

Of course, this definition only really makes sense if the list of handle types is given in advance, and then the requirement that $\mathcal{H}$ is of uniform type forces its 0-handles to conform to this list of types. But this slightly imprecise wording is useful, because it will be convenient occasionally to enlarge our finite list of types. An example is given below in Theorem \ref{Thm:RemainUniformType}.

It will be important that we only use handle structures of uniform type. This is for many reasons. For example, if we wish to write down a handle structure, it is useful that each of its 0-handles falls into finitely many types. Moreover, when the handle structure is admissible, we can list the different types of normal disc that can arise within a handle. 


The following result ensures that the handle structures in our partial hierarchies have uniform type.

\begin{theorem}[Maintaining uniform type]
\label{Thm:RemainUniformType}
Let $\mathcal{H}$ be a handle structure of uniform type, with no 3-handles, 
for a compact orientable irreducible 3-manifold $M$ with boundary pattern $P$. Let $\mathcal{H}'$ be 
the handle structure of a 3-manifold $M'$ with boundary pattern $P'$
obtained from $\mathcal{H}$ by a finite sequence of the following moves:
\begin{enumerate}
\item decomposition along a connected normal surface, which is not a boundary-parallel disc, and applied to an admissible handle structure;
\item application of {\sc Create admissible handle structure} in Proposition \ref{Prop:CreateAdmissibleHS};
\item decomposition along vertical boundary components of the interior parallelity bundle;
\item replacing a component of the interior parallelity bundle that is a copy of $D^2 \times I$ by a single 2-handle;
\item application of {\sc (Pattern-)compressible vertical boundary} in Proposition \ref{Prop:CompressAnnuli}.
\end{enumerate}
Then $\mathcal{H}'$ is of uniform type, although possibly with a larger set of handle types. Moreover, when $\mathcal{H}'$ is admissible and its interior parallelity bundle consists just of 2-handles, then for each 0-handle $H$ of $\mathcal{H}$, there are only finitely many possibilities for the following subsets:
\begin{enumerate}
\item the intersection between $H$ and the 0-handles of $\mathcal{H}'$;
\item $H \cap \mathcal{F}^0(\mathcal{H}')$;
\item $H \cap \mathcal{F}^1(\mathcal{H}')$;
\item $H \cap P'$,
\end{enumerate}
up to the following operations:
\begin{enumerate}
\item isotopy of $H$ preserving $H \cap \mathcal{H}^1$, $H \cap \mathcal{H}^2$ and $H \cap P$;
\item removing a small collar neighbourhood of a subset of $\partial M'$;
\item re-embedding $\calH' \cap H$ in $H$ in a way that leaves $\calF^0(\calH')$, $\calF^1(\calH')$
and $P' \cap \partial \calF(\calH')$ unchanged.
\end{enumerate}
\end{theorem}

\begin{remark}
Note that we are only claiming fairly limited control over the location of the 2-handles of $\calH'$. We are specifying their attaching locus via $H \cap \mathcal{F}^0(\mathcal{H}')$ and $H \cap \mathcal{F}^1(\mathcal{H}')$, but we do not specify where the remainder of the 2-handles lie.

The re-embedding of $\calH' \cap H$ might require some further explanation. Suppose that we had applied a sequence of modifications to $\calH$, giving the handle structure $\calH'$, and then suppose that we made some further modifications to give another handle structure $\calH''$. Suppose that $\calH' \cap H$ and $\calH'' \cap H$ differ solely by a Dehn twist about a disc properly embedded in $M'$ and lying in $\calH' \cap H$. Then $\calH' \cap H$ and $\calH'' \cap H$ would differ by operation (3) above, and so, for the purposes of this proposition, would be viewed as equivalent.
\end{remark}

\begin{proof}
We define a temporary notion of complexity for a 0-handle $H$ of a handle structure, called \emph{t-complexity}. It will be a list of four integers. The first element of the 4-tuple is the sum of the indices of the components of $H \cap \mathcal{F}$ with positive index. The second element is $|H \cap \mathcal{F}^1|$. The third element is $|H \cap \mathcal{F}^0|$. 
The final element is the number of vertices of $P$ in $H$. We compare two such 4-tuples in the usual way, using lexicographical ordering. The set of 4-tuples is then well-ordered.

We will prove by induction on t-complexity that a 0-handle type $H$ gives rise to only finitely many handle types of $\mathcal{H}' \cap H$ under a sequence of the operations in (1)-(5) above. Moreover, we will show that the union of the 0-handles of $\mathcal{H}' \cap H$, other than those that are interior parallelity handles or  are removed in {\sc Create admissible handle structure}, embed in $H$ in one of only finitely many ways, up to the equivalence described in (1)-(3) at the end of the proposition. This will also establish the second part of the proposition, when $\mathcal{H}'$ is admissible and its interior parallelity bundle consists of 2-handles.

The induction starts with $H$ having t-complexity $(0,0,0,0)$. Then $H$ is an isolated 0-handle and none of the procedures in (1)-(5) affect it.

We now prove the inductive step. We consider when one of procedures (1)-(5) is applied to a 0-handle $H$. We may assume that the result is not a single 0-handle $H'$ that is the same type as $H$ and where $H'$ is obtained from $H$ by removing a small collar on a subset of $\partial H$, as in this case there is nothing to prove.

In the procedure {\sc Create admissible handle structure}, there are three types of operation that are performed. One type is where we collapse a 1-handle and 2-handle pair, where the 1-handle has at most two arcs of $P$ running along it. The effect on $H \cap \mathcal{F}^0$ is possibly to remove some components of $H \cap \mathcal{F}^0$ with non-positive index. The effect on  $H \cap \mathcal{F}^1$ is to remove some components and replace them by at most two parallel arcs of $P \cap H$. This does not increase the index of any component of $H \cap \mathcal{F}$. If $H$ is modified, then it reduces at least one of $|H \cap \mathcal{F}^1|$ and $|H \cap \mathcal{F}^0|$, and it does not increase either of them. There are only finitely many such modifications that can be applied to $H$. Hence, the induction is proved in this case.

The second type of operation in {\sc Create admissible handle structure} is the removal of isolated 0-handles of $\mathcal{F}$ with non-positive index. Each such operation does not affect the index of the other components of $\mathcal{F} \cap H$. It leaves $|H \cap \mathcal{F}^1|$ unchanged and it reduces $|H \cap \mathcal{F}^0|$. Hence, again the inductive step is proved.

The third type of operation in {\sc Create admissible handle structure} is when there is a simple closed curve $C$ in $\partial H \cut \calF$ that intersects the pattern at most three times. It is assumed that $C$ does not bound a disc in $\partial H \cut \calF$ but that it does bound a disc in $\partial M$ intersecting $P$ in the empty set, an arc or a tripod. When $\partial H \cut (\calF \cup P)$ is not a union of discs, then $C$ is chosen to be boundary parallel in a component of $\partial H \cut (\calF \cup P)$ that is not a disc. Hence, in that case, there are only finitely many possibilities for $C$. On the other hand, when $\partial H \cut (\calF \cup P)$ is a union of discs, then again there are only finitely many possibilities for $C$. Now $C$ bounds a properly embedded disc $D$ in $H$, and the operation that is performed is to remove the 3-ball separated off by $D$, and possibly to insert into $D$ an arc or tripod. The effect of this on $H$ is to remove some components of $\calF$ and to possibly insert some new pattern. This does not increase the index of the components of $H \cap \calF$, it does not increase $|H \cap \calF^1|$, and it strictly reduces $|H \cap \calF^0|$. Hence again the inductive step is proved.

The effect when we decompose along vertical boundary components of the interior parallelity bundle is similar to the case of {\sc Create admissible handle structure}. The effect is to replace some components of $H \cap \mathcal{F}^1$ by two parallel arcs of $H \cap P$. This leaves index unchanged, but reduces $|H \cap \mathcal{F}^1|$. It also removes some components of $H \cap \mathcal{F}^0$ that are disjoint from $P$ and that intersect $\mathcal{F}^1$ in two arcs, and replaces them by two arcs of $P$. Again, this leaves index unchanged, and reduces $|H \cap \mathcal{F}^0|$.

In {\sc (Pattern-)compressible vertical boundary}, submanifolds of the manifold are removed. The effect on $H$ is either to remove it entirely (in which case there is nothing to prove) or possibly to remove some components of $\mathcal{F}^1$ and replace them with an arc of $P$, and possibly to remove some components of $\mathcal{F}^0$ that intersect $\mathcal{F}^1$ in two arcs and are disjoint from $P$, and replace them with an arc of $P$. As argued above, this does not increase the index of any component of $\mathcal{F}$ or  $|H \cap \mathcal{F}^0|$, and it strictly reduces $|H \cap \mathcal{F}^1|$.

When a $D^2 \times I$ component of the interior parallelity bundle is removed and replaced by a 2-handle, the effect is to remove some index zero components of $\calF^0$ that are disjoint from $P$ and, for each such component, to combine its two incident components of $\calF^1$ into a single component of $\calF^1$. This reduces the number of components of $\calF^0$ and $\calF^1$ and it leaves the index of each component of $\calF$ unchanged.

We now examine what happens when $H$ is decomposed along a normal surface. We are assuming that the handle structure is admissible in this case. Hence, every 0-handle of $H \cap \mathcal{F}^0$ has non-negative index. So, $H$ has non-negative index. Furthermore, if it has zero index, then it is an interior parallelity handle. In that case, any normal disc in the handle is horizontal or vertical in the bundle. In the horizontal case, the handle is decomposed into multiple copies of itself and so these remain of uniform type. In the vertical case, the resulting handles are of the form $D^2 \times I$, where each handle in $\calF$ is of the form $\alpha \times I$ for some arc $\alpha$ in $\partial D^2$, and the intersection with $P$ is $\beta \times \partial I$ for some union of arcs $\beta$ in $\partial D^2$. Such handles are not admissible, and when operations (2)-(5) are applied to them, they remain in this form or are removed. So we may include these handle types in the class of uniform ones.

So, we may assume that the index of $H$ is positive. The normal surface intersects $H$ in discs. We can decompose along these one at a time. So we focus on one such disc $D$. Note that there are only finitely many normal disc types in $H$. This disc $D$ divides $H$ into two 0-handles $H_0$ and $H_1$, both of which have non-negative index, by Proposition \ref{Lem:AdmissibleDecomposition}. The sum of these indices is $I(H)$. Hence, if they are both positive, then $H_0$ and $H_1$ have smaller index, and so we are done by induction in this case. So suppose that one 0-handle $H_0$ has zero index and the other handle $H_1$ has positive index. Since $D$ is normal, each component of $\mathcal{F}^1 \cap H$ is either divided into two by $D$ or lies wholly in one of the handles $H_0$ or $H_1$. If one component of $\mathcal{F}^1 \cap H$ ends up wholly in $H_0$, then $H_1$ has smaller t-complexity than $H$, and we are again done by induction.

If $\partial D$ lies wholly in $\mathcal{F} \cap H$, then the component of $H_0 \cap \mathcal{F}$ incident to $\partial D$ has zero index and so must be an annulus disjoint from the pattern. Hence, in this case, $H_1$ is a copy of $H$ and, given $H$, there are only finitely many possible types for $H_0$. 

Suppose that $\partial D$ is disjoint from $\mathcal{F} \cap H$. It is therefore properly embedded in $M$. Since we are assuming that the decomposing surface is connected and is not a boundary-parallel disc, $D$ must separate components of $\mathcal{F} \cap H$. Hence, in this case, $H_1$ has smaller index than $H$, which is contrary to assumption. 

So suppose that $\partial D$ intersects $\mathcal{F} \cap H$ in arcs. These must be boundary-parallel in $\mathcal{F}$.

Let $D'$ be the component of $\partial H \cut \partial D$ lying in $H_0$. This intersects $\mathcal{F}$ in $n$ components, say, each of which is a disc intersecting $P$ twice.
Let $\Gamma$ be the graph $H_0 \cap P$. This is a graph where each vertex has degree $1$ or $3$. The vertices with degree $1$ occur at the points of $P \cap \mathcal{F}$ in $H_0$, and also at the points $\partial D \cap P$. Let $v_1$ and $v_3$ be the number of vertices with degree $1$ and $3$ respectively. 

\begin{figure}[h]
\centering
\includegraphics[width=0.8\textwidth]{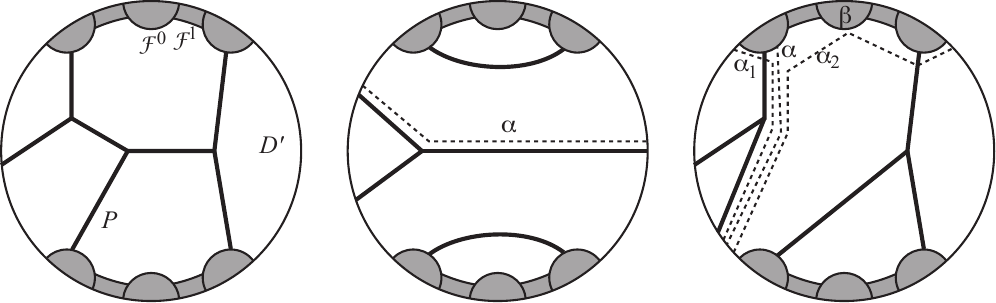}
\caption{Left: the disc $D'$ and its intersection with $P$ and $\mathcal{F}$. Middle: a tree component of $\Gamma$ disjoint from $\mathcal{F}$. Right: A tree component of $\Gamma$ intersecting $\mathcal{F}$ at a single point} \label{Fig:DecompNormalDisc}
\end{figure}

We claim that no component $\Gamma'$ of $\Gamma$ is a tree disjoint from $\mathcal{F}$ (see Figure \ref{Fig:DecompNormalDisc} middle).
A regular neighbourhood of $\Gamma'$ in $D'$ is a disc, and the part of the boundary of this disc not lying in $\partial D'$ is a collection of arcs. There must be at least one such arc $\alpha$ that is not outermost in $D'$, in the sense that it does not separate off a sub-disc of $D'$ disjoint from the other arcs and from $\mathcal{F}$. At its endpoints are two arcs of $\partial D \cut (P \cup \mathcal{F})$ lying in the same component of $\partial H \cut (P \cup \mathcal{F})$. By the normality of $D$, these are therefore the same arc of $\partial D \cut (P \cup \mathcal{F})$. Hence, $\alpha$ was outermost which is a contradiction. 

We now show that no component  of $\Gamma$ is a tree intersecting $\mathcal{F}$ in a single point (see Figure \ref{Fig:DecompNormalDisc} right).
Suppose that $\Gamma'$ is such a component. Let $\beta$ be the arc of $\partial \mathcal{F} \cut P$ incident to $\Gamma'$ and that is disjoint from $\partial D$.  Let $\alpha$ be an arc in $D'$ disjoint from $\Gamma$ running from $\beta$ to $\partial D$ that runs close to $\Gamma'$. We may extend $\alpha$ along $\partial \mathcal{F}$ in two ways to form arcs $\alpha_1$ and $\alpha_2$ properly embedded in $D'$, each intersecting $P$ once. Let $\alpha_1$ be the arc that intersects $\Gamma'$. The endpoints of $\alpha_1$ lie in components of $\partial H \cut (\mathcal{F} \cup P)$ that are incident along an arc of $\partial H \cap P$. So, by (5) in the definition of normality, $\partial D$ intersects this arc. By (3), $\partial D$ must run straight between these two endpoints of $\alpha_1$, intersecting $P$ and nothing else. We deduce that $\Gamma'$ is a single edge. A similar argument applied to $\alpha_2$ gives that it is also parallel to a sub-arc of $\partial D$. So, $D' \cap P$ consists of two arcs. These must emanate from the same 0-handle of $\calF$, as otherwise $\partial D$ violates (4) in the definition of normality. So, we deduce that $D$ contradicts (6) in the definition of normality.

We deduce that each component $\Gamma'$ of $\Gamma$ satisfies $2\chi(\Gamma') - |\mathcal{F} \cap \Gamma'| \leq 0$. So, $v_1 - v_3 - |\mathcal{F} \cap \Gamma| = 2 \chi(\Gamma) - |\mathcal{F} \cap \Gamma| \leq 0$. Now the number of vertices of the boundary pattern in $H_1$ is equal to the number of vertices of $P$ in $H$, plus $|\Gamma \cap \partial D|$, minus $v_3$. But $|\Gamma \cap \partial D| = v_1 - |\mathcal{F} \cap \Gamma|$. We deduce that the number of vertices of the boundary pattern in $H_1$ is at most the number of vertices of $P$ in $H$. 

To summarise, we have shown that in each of the operations (1)-(5), the t-complexity of $H$ strictly goes down, except possibly in the case of (1). When the t-complexity does not go down, $H$ is decomposed along the normal disc $D$, and there is a unique new 0-handle $H_1$ with positive q-complexity. 
The remaining 0-handles are interior parallelity 0-handles, or non-admissible 0-handles which are completely removed in later modifications. 
Returning to the 0-handle $H_1$, we observe that
$H_1 \cap \mathcal{F}$ is an exact copy of $H \cap \mathcal{F}$, and the new pattern $P'$ intersects $H_1 \cap \mathcal{F}$ in the same location that $P$ intersected $H \cap \mathcal{F}$. More precisely, there is a homeomorphism $H \rightarrow H_1$ taking $\mathcal{F}^0(H)$ homeomorphically to $\mathcal{F}^0(H_1)$, taking $\mathcal{F}^1(H)$ homeomorphically to $\mathcal{F}^1(H_1)$, and taking $P \cap \mathcal{F}(H)$ homeomorphically to $P' \cap \mathcal{F}(H_1)$. This homeomorphism need not take $P \cap H$ homeomorphically to $P' \cap H_1$, but we have shown that the number of vertices of $P' \cap H_1$ is at most the number of vertices of $P \cap H$. Moreover, no simple closed curve components of the pattern have been introduced to $H_1$. Hence, after a sequence of steps (1)-(5), the t-complexity must go down at some stage, or the handles run through only finitely many types. In either case, the argument is complete.
\end{proof}

\begin{remark}
The above proposition is crucial to the efficient certification that a boundary pattern $P$ for a compact orientable 3-manifold $M$ is essential, as required by Theorem \ref{Thm:EssentialNPcoNP}. When $M$ is irreducible and $\partial M$ is non-empty, this is achieved by giving an essential hierarchy for $(M, P)$. Now $(M,P)$ is given to us by means of a triangulation $\calT$,
and so Theorem \ref{Thm:RemainUniformType} guarantees that the resulting handle structures for all the manifolds in the hierarchy are of uniform type. This is important
at many points. For example, in Proposition \ref{Prop:ExpBoundExtendedWeight} below, a constant $k$ depends on the type of the handles. Hence, for all the handle structures that we consider, $k$ is in fact bounded above by a universal constant independent of $M$ and the hierarchy.

Another crucial conclusion from Theorem \ref{Thm:RemainUniformType} is that when the handle structure $\calH'$ for a manifold in our hierarchy is admissible and its interior parallelity bundle consists of 2-handles, then the intersection between the 0-handles of $\calH'$ and any tetrahedron of $\calT$ is one of only finitely many types, where again this list of types is universal. In particular, the number of 0-handles of $\calH'$ is bounded above by a universal constant times the number of tetrahedra of $\calT$. This is important because many quantities depend on $|\calH'|$, for example again in Proposition \ref{Prop:ExpBoundExtendedWeight}. It is also central to the efficiency of our algorithm. For example, when a manifold in the hierarchy is created and all its boundary components are 2-spheres, then we need to examine whether the boundary pattern in these 2-spheres is essential. We can do so efficiently, since the handle structure on this manifold has uniform type and its number of handles is bounded above by a universal constant times $|\calT|$.

Theorem \ref{Thm:RemainUniformType} is an analogue of results known for sutured manifold hierarchies (see \cite[Proposition 6.3]{Lackenby:Exceptional} and \cite[Theorem 6.1]{Lackenby:EfficientCertification}).
\end{remark}

\section{Certifying an inessential pattern}
\label{Sec:CertifyingInessential}

In this section, we establish one part of Theorem \ref{Thm:EssentialNPcoNP}, by showing that 
{\sc Essential boundary pattern} is in co-NP. 

We start by recalling some terminology from normal surface theory \cite{Matveev}.
We fix a triangulation $\mathcal{T}$ of the compact orientable 3-manifold $M$.

There are $7|\mathcal{T}|$ types of
normal triangles and squares in $\mathcal{T}$. We assign a variable $x_i$ to each type. When $S$ is a normal surface,
these variables count the number of triangles and squares of each type in $S$, giving a list of $7|\mathcal{T}|$ non-negative integers for
$S$ called its \emph{normal surface vector} and denoted $(S)$. 

A normal surface is said to be the \emph{normal sum} of two properly embedded normal surfaces $S_1$ and $S_2$ if its normal vector is
$(S_1) + (S_2)$. We denote the surface by $S_1+S_2$.  A normal surface is said to be \emph{fundamental} if it cannot be written as the normal sum of two non-empty normal surfaces.

In the case where the given 3-manifold is irreducible, our certificate for having inessential pattern will
be based on the following result, which we will prove in this section.

\begin{theorem}[Fundamental violating disc]
\label{Thm:ViolatingFundamental}
Let $M$ be a compact orientable irreducible 3-manifold with a boundary pattern $P$. Let $\mathcal{T}$
be a triangulation for $M$ in which $P$ is simplicial. If $P$ is inessential, then it has a violating
disc that is normal and fundamental.
\end{theorem}

The following result provides a violating disc in an important special case.

\begin{theorem}[Fundamental compression disc]
\label{Thm:CompressionDiscFund}
Let $M$ be a compact orientable irreducible 3-manifold with pattern $P$. Let $\mathcal{T}$ be
a triangulation for $M$ in which $P$ is simplicial. If $\partial M \cut P$ is 
compressible, then it has a compression disc that is normal and fundamental.
\end{theorem}

This result was proved in \cite[Theorem 4.1.13]{Matveev}. Although that theorem does not
explicitly assert that the compression disc is fundamental, this is established in the proof.
So, in our proof of Theorem \ref{Thm:ViolatingFundamental}, we may assume that $\partial M \cut P$
is incompressible. In the terminology of Matveev \cite{Matveev}, this means that $(M,P)$ is 
\emph{boundary-irreducible}. 

Let $S$ be a normal surface
that is the normal sum of two surfaces $S_1$ and $S_2$. We may pattern-isotope one of the surfaces, preserving
its normality throughout, so that afterwards $S_1$ and $S_2$ intersect transversely in a collection of simple closed curves and properly embedded
arcs. Then each component of $S_1 \cut S_2$ and $S_2 \cut S_1$ is a \emph{patch}. It is 
a \emph{clean disc patch} if it is a disc that is disjoint from the pattern and is either disjoint from $\partial M$
or intersects $\partial M$ in a single arc.

A summation $S = S_1 + S_2$ is in \emph{reduced form} if there do not exist normal surfaces $S'_1$ and $S'_2$
where each $S'_i$ is pattern-isotopic to $S_i$, and $S = S'_1 + S'_2$, with $|S'_1 \cap S'_2| < |S_1 \cap S_2|$.

Let $S$ be a compact surface properly embedded in $M$. A \emph{clean boundary-compression disc} for $S$
is a disc embedded in $M$ such that $\partial D$ is the concatenation of two arcs $\alpha$ and $\beta$, where
\begin{enumerate}
\item $\alpha = D \cap S$ does not separate off a disc from $S$ that is disjoint from $P$;
\item $\beta = D \cap \partial M$ is disjoint from $P$.
\end{enumerate}
When $S$ has no clean boundary-compression disc, it is said to be \emph{cleanly boundary-incompressible}.

Note that a violating disc that intersects $P$ as few times as possible is cleanly boundary-incompressible.

The following is \cite[Lemma 4.1.8]{Matveev}.

\begin{theorem}
\label{Thm:NoCleanDiscPatches}
Let $M$ be a compact orientable irreducible 3-manifold with pattern $P$. Let $\mathcal{T}$ be
a triangulation for $M$ in which $P$ is simplicial. Suppose that $\partial M \cut P$ is incompressible.
Let $S$ be a compact orientable normal surface properly embedded in $M$, that is incompressible
and cleanly boundary-incompressible and that has no component that is a disc parallel to a
disc in $\partial M$ intersecting $P$ in an arc or the empty set. Suppose that $S$ is a sum of
two normal surfaces $S_1$ and $S_2$ in general position and that this summation is reduced.
Suppose also that $S$ has minimal weight up to isotopy preserving the pattern. Then 
no patch of the summation is a clean disc patch.
\end{theorem}

We need an analogous version of this result for violating discs, and with a strengthened conclusion about
disc patches. When a normal surface $S$ is a sum of surfaces $S_1$ and $S_2$, we say that a patch is 
a \emph{nearly clean disc patch} if it is a disc intersecting $\partial S$ in at most one arc and
intersecting the pattern $P$ in at most one point. 

\begin{theorem}[Patches in a violating disc]
\label{Thm:NoNearlyCleanDiscPatches}
Let $M$ be a compact orientable irreducible 3-manifold with pattern $P$. Let $\mathcal{T}$ be
a triangulation for $M$ in which $P$ is simplicial. Suppose that $\partial M \cut P$ is incompressible.
Let $S$ be a violating disc for $P$ that is normal and that intersects $P$ as few times as possible,
and subject to this condition, suppose that $S$ has least possible weight.
Suppose that $S$ is a sum of
two normal surfaces $S_1$ and $S_2$ in general position and that this summation is reduced.
Then no patch of the summation is a nearly clean disc patch.
\end{theorem}

We now sketch the main steps of the proof. This follows very closely the proof of Theorem \ref{Thm:NoCleanDiscPatches},
and so we do not give full details.

We need further terminology about the sum $S$ of two normal surfaces $S_1$ and $S_2$. 
We pattern isotope $S_1$, keeping it in normal form throughout,
so that $S_1$ and $S_2$ intersect transversely in a collection of arcs and closed curves. 
It is well known (see for example \cite[Section 3.2.2]{Matveev}) that $S$ is obtained from $S_1 \cup S_2$ by resolving the intersections $S_1 \cap S_2$
as follows. A regular neighbourhood of $S_1 \cap S_2$ is removed from $S_1$ and $S_2$. This
regular neighbourhood is a collection of discs, annuli and M\"obius bands. At each component of $S_1 \cap S_2$,
a pair of discs or one or two annuli are inserted, lying in a regular neighbourhood of that component of $S_1 \cap S_2$. These
are known as \emph{trace} discs and annuli.
Each of these discs and annuli has the structure of an $I$-bundle, where the $(\partial I)$-bundle is the
intersection with $S_1 \cup S_2$. The zero-sections of these $I$-bundles are the \emph{trace curves}
in $S$. When we cut $S$ along these trace curves, the result is (after a small isotopy) the patches of the sum.

At any component $C$ of $S_1 \cap S_2$, this procedure of cutting out a regular neighbourhood of $C$
from $S_1$ and $S_2$, and then attaching discs or annuli is called a \emph{switch}. 
There are two possible ways to perform a switch. Only one choice of
switch at each component of $S_1 \cap S_2$ leads to the normal surface $S$. These are called
\emph{regular switches}. The other type of switch is an \emph{irregular switch}. If a switch is made at each
component of $S_1 \cap S_2$ and at least one of these switches is irregular, then the resulting surface $S'$
admits a disc $D$ in some face of $\mathcal{T}$, such that $D \cap S'$ is an arc in $\partial D$, and the
remainder of $\partial D$ lies in an edge of $\mathcal{T}$. This is called a \emph{return disc}.
When $D$ is disjoint from $\partial M$,
this leads to an isotopy that reduces the weight of $S'$. When $D$ intersects $\partial M$ but not $P$,
then it is a potential clean boundary-compression disc for $S'$. If it is not a clean boundary-compression disc
then, assuming $\partial M \cut P$ is incompressible and $M$ is irreducible, there is an isotopy of $S'$ leaving $P$ fixed,
that reduces its weight. On the other hand, when $D$ intersects $P$, then a small
isotopy of $D$ pushing it off the face makes it a potential clean boundary-compression disc.
There are two possible directions in which to push $D$ off the face and if neither direction gives
a clean boundary-compression disc, then, again assuming $\partial M \cut P$ is incompressible  and $M$ is irreducible, 
$S'$ must have been a disc parallel to a disc in $\partial M$ intersecting $P$ in a single arc.

If a component of $S_1 \cap S_2$ is two-sided in both $S_1$ and $S_2$, then it contributes two
trace curves in $S$. We say that they are \emph{twins}. A component of $S_1 \cap S_2$
is one-sided in $S_1$ if and only if it is one-sided in $S_2$. In that case, it gives a single
trace curve in $S$, which is its own twin.

Let $C$ be a component of $S_1 \cap S_2$, and let $E$ and $E'$ be two distinct patches incident to $C$.
We say that $E$ and $E'$ are \emph{opposite} at $C$ if they lie in the same surface $S_i$;
otherwise they are \emph{adjacent}. It may happen that the same patch $E$ approaches $C$ from
opposite sides. In that case, $E$ is \emph{self-opposite} at $C$.

Consider a nearly clean disc patch $E$ that is not self-opposite along the component of $S_1 \cap S_2$ in its
boundary. Let $C$ be the arc or simple closed curve $\partial E \cut \partial S$.
Suppose that its twin curve $C'$ separates off a nearly clean disc $E'$ from $S$ that is a disc
intersecting $\partial S$ in at most one arc and intersecting $P$ in at most one point. We say that $E'$ is 
an \emph{adjacent} or \emph{opposite companion} of $E$ if the patch of $E'$ containing $C'$ is adjacent 
or opposite to $E$.

We now summarise the key steps in the argument for Theorem \ref{Thm:NoNearlyCleanDiscPatches}.
Let $S$ be a normal violating disc as in the statement of the theorem, and suppose that $S$ is
the normal sum of $S_1$ and $S_2$, as in the statement of the theorem.

Since $S$ intersects $P$ as few times as possible among all violating discs, it is pattern-incompressible.
It is automatically incompressible.

 The argument of Lemma 4.1.2 of \cite{Matveev} gives that if $S$ has a self-opposite
nearly clean disc patch, then $M$ contains an embedded real projective plane. This implies that $M$
is reducible or $\mathbb{RP}^3$, neither of which is the case. 

So we now suppose that $S$ has no self-opposite nearly clean disc patch. Let $E$ be any nearly clean
disc patch, and let $C$ be the trace curve in its boundary. Suppose that the twin of $C$ does not separate
off a nearly clean disc in $S$. Then $S$ is compressible or pattern-compressible, which we have ruled out. This is the analogue
of Lemma 4.1.3 in \cite{Matveev}.

Now suppose that $S$ intersects $P$ at most twice. By Theorem \ref{Thm:NoCleanDiscPatches},
there are no clean disc patches. 
Every nearly clean disc patch admits an adjacent companion disc
that intersects the pattern once, simply because any arc in $S$ separates
off some disc that intersects the pattern at most once. If that disc
was disjoint from the pattern, then this would lead to a clean disc patch,
which has been ruled out.

Suppose that no nearly clean disc patch has an adjacent companion disc that is also a patch.
Then an analogue of Lemma 4.1.4 in \cite{Matveev} gives that 
the surface $S$ can
be written as $S' + T$ where $T$ is a clean annulus or clean M\"obius band,
and where $S' \cap T$ consists of essential arcs in $T$.

Suppose that $T$ intersects $S'$ in a single arc. If $T$ is an annulus,
then $S$ is non-separating, and so the surface obtained by performing an irregular
switch along $S' \cap T$ is also non-separating, and so it forms a violating disc.
If $T$ is a M\"obius band and $S$ is separating, then $T$ is non-trivial in $H_2(M, \partial M; \mathbb{Z}/2\mathbb{Z})$
and so $S'$ is non-separating. On the other
hand, if $S$ is non-separating, then the surface obtained by performing an irregular
switch is also non-separating. In all cases, we have found a violating disc that intersects
the pattern at most twice and that either
has small weighter weight than $S$ or has a return. In the latter case, this can be used to find a violating
disc that intersects the pattern at most twice and with smaller weight than $S$. This is a contradiction.

So we may assume that $S' \cap T$ is at least two arcs.
If we perform an irregular switch along one
of the trace arcs incident to a nearly clean disc patch and a regular switch along all the other trace arcs, 
the result is the disjoint union of a disc $S''$ intersecting the pattern 
twice and an annulus or M\"obius band $T''$. Suppose first that $S''$ is a violating disc.
Since we used an irregular switch,
this surface admits a return disc $D$. If $D$ intersects the annulus or M\"obius band $T''$, then 
$T''$ has positive weight and so the weight of $S''$ is strictly less than the weight of $S$.
On the other hand, if $D$ intersects $S''$, then $S''$ admits a pattern-isotopy that reduces its
weight. Both of these situations are impossible under the assumption that $S$ has smallest weight.

So we may suppose that, for each nearly clean disc patch, when an irregular switch is performed along its incident
trace arc and regular switches are performed at the remaining trace arcs, this does not create
a violating disc. But then $S'$ is pattern-isotopic to $S$. This violates the assumption that $S$ had least weight. 

Now consider the case where a nearly clean disc patch $E$ has an adjacent companion disc $E'$ that
is also a patch. Say that $E$ lies in $S_1$, and $E'$ lies in $S_2$. Then we may form $S'_1$ from $S_1$
by removing $E$ and attaching $E'$, and we create $S'_2$ by removing $E'$ and attaching $E$.
If $E \cup E'$ is not a violating disc, then $S_1'$ and $S'_2$ are pattern-isotopic to $S_1$ and $S_2$ respectively. Furthermore,
$S$ is also $S'_1 + S'_2$, but $|S'_1 \cap S'_2| < |S_1 \cap S_2|$. This is ruled out by the assumption
that the sum $S_1 + S_2$ is reduced.

On the other hand, suppose that $E \cup E'$ is a violating disc. This is a component $S'$
of the surface obtained from $S_1$ and $S_2$ by performing an irregular switch along $E \cap E'$
and performing regular switches along the other curves of $S_1 \cap S_2$. Let $T$ be the union
of the other components of this surface. Since we performed an irregular switch, $S' \cup T$ admits
a return disc $D$. If $D$ intersects $T$, then $T$ had positive weight,
and so the weight of $S'$ is less than the weight of $S$. If $D$ intersects $S'$ then a pattern-isotopy reduces the weight of $S'$,
and so the resulting violating disc has smaller weight than that of $S$. Furthermore, $|S' \cap P| \leq |S \cap P|$.
This contradicts our assumption that $|S \cap P|$ is minimal among violating discs
or the assumption that the weight of $S$ was minimal. 

Thus, we have ruled out the situation where $S$ intersects the pattern at most twice. We therefore
suppose that $S$ intersects the pattern three times. Since we are assuming that $S$ is a violating disc intersecting the pattern as few times as possible,
there can be no violating discs intersecting the pattern at most twice.
The argument above deals with the situation where every nearly clean disc patch admits an adjacent companion disc
that intersects the pattern at most once.

So if there is a nearly clean disc patch $E$, then it has an opposite
companion disc $E'$ intersecting the pattern at most once. The argument that we present here is a minor adaption of that in Lemma 4.1.5 in
\cite{Matveev}. If $E'$ does not contain $E$, then their union is a sphere or disc that intersects the pattern
at most twice. A small perturbation makes $E \cup E'$ disjoint from $S$, and we observe that components
of $S$ lie on each side of $E \cup E'$. Hence, by the connectedness of $S$, $E \cup E'$ is non-separating.
This contradicts the assumption that $M$ is irreducible or the assumption that $P$ has no violating
disc intersecting the pattern at most twice.

Now suppose that $E'$ contains $E$. Then one may form an irregular switch along $E' \cap E$, and obtain
the disjoint union of a surface $S'$ and a surface $T$ that is an annulus, M\"obius band, Klein bottle or torus. The surface $S'$ is obtained
from $S$ by removing $E'$ and replacing it with $E$. Hence, by the assumption that $M$ is irreducible and the fact
that there is no violating disc intersecting the pattern at most twice, $S'$ is pattern-isotopic to $S$. Since $S' \cup T$ is obtained by performing an irregular
switch, this again contradicts the assumption that $S$ had least weight.

Thus, in all cases, we reach a contradiction. This confirms the conclusion that the sum $S_1 + S_2$ has
no nearly clean disc patch, completing the outline of the proof of Theorem \ref{Thm:NoNearlyCleanDiscPatches}.

\begin{proof}[Proof of Theorem \ref{Thm:ViolatingFundamental}]
In the case where $\partial M \cut P$ is compressible, then Theorem \ref{Thm:CompressionDiscFund} provides
a fundamental compression disc. So we may assume that $\partial M \cut P$ is
incompressible. 

Let $S$ be a violating disc that intersects the pattern as few times as possible. Then $S$ is
cleanly boundary-incompressible and it is automatically incompressible.
So, by \cite[Corollary 3.3.25]{Matveev}, $S$ can be isotoped,
preserving the pattern, into normal form without increasing its weight.
We may assume that $S$ has least weight among all normal violating discs that
intersect the pattern as few times as possible.

Suppose that $S$ is a normal sum $S_1 + S_2$. We may assume that
this summation is in reduced form. The trace curves in $S$ are a system of
arcs and simple closed curves. There can be no simple closed curves, because an
innermost such curve would bound a clean disc patch, contradicting Theorem \ref{Thm:NoCleanDiscPatches}.
So the trace curves are just arcs. Consider an outermost such arc, separating
off a disc patch. By Theorem \ref{Thm:NoNearlyCleanDiscPatches}, this patch
must contain at least two points of intersection with $P$. Consider the complement
of this disc patch and a trace arc that is outermost in it. Again, the disc in $S$
that it separates off must contain at least two points of intersection with $P$.
We therefore obtain at least four points of intersection between $S$ and the pattern,
contradicting the assumption that $S$ is a violating disc.

So $S$ must have been fundamental.
\end{proof}

We now complete one half of the proof of Theorem \ref{Thm:EssentialNPcoNP}. We are given 
a triangulation $\calT$ of a compact orientable 3-manifold $M$, with boundary pattern
$P$ that is simplicial. When $P$ is inessential, we need to provide a certificate
that confirms this and that can be verified in polynomial time as a function of $t= |\calT|$,
the number of tetrahedra. When $M$ is irreducible, Theorem \ref{Thm:ViolatingFundamental}
shows that there is a violating disc $D$ that is fundamental. This is the required certificate.
The algorithm of Agol, Hass and Thurston \cite{AgolHassThurston} can be used to verify that $D$ is indeed
a disc intersecting the pattern at most 3 times. Also, by \cite[Theorem 9.4]{Lackenby:EfficientCertification},
one can check whether $\partial D$ bounds a disc in $\partial M$. If there is no such disc,
then $D$ is violating. So suppose that $\partial D$ does bound a disc $D'$ in $\partial M$.
It may be the case that $\partial D$ bounds two discs, in which case we consider them both in 
turn. Again using \cite[Theorem 9.4]{Lackenby:EfficientCertification}, one can check whether 
each component of $D' \cut P$ is a disc intersecting $\partial D'$ in one arc or curve and
with at most one vertex of $P$ in its boundary. Since $D$ is violating, at least one
component of $D' \cut P$ is not of this form, and the algorithm provided by \cite[Theorem 9.4]{Lackenby:EfficientCertification}
will confirm this.

We need to deal with the general case where $M$ is reducible. Our certificate for establishing that $P$ is inessential is:
\begin{enumerate}
\item a vector $(S)$ of a (possibly empty) normal surface $S$ in $\mathcal{T}$ with weight at most $2^{740t^2}$; this will in fact be a maximal collection of disjoint embedded normal spheres;
\item a triangulation $\mathcal{T}'$ for a 3-manifold $M'$ with at most $1000t$ tetrahedra; this will be $M \cut S$ with a 3-ball attached to each spherical boundary component arising from $S$;
\item a certificate that $M'$ has inessential boundary pattern.
\end{enumerate}

We now explain why this certificate exists.
It is proved in \cite[Lemma 4]{King} that there is a maximal collection of disjoint non-parallel normal 2-spheres $S$ in $\mathcal{T}$, none of which bounds a ball in $M$, 
with weight at most $2^{740t^2}$.
A triangulation for $M \cut S$ with at most $1000t$ tetrahedra is produced using \cite[Theorem 11.4]{Lackenby:EfficientCertification}. Then a coned 3-ball
is attached to each spherical boundary component, forming the triangulation $\mathcal{T}'$ for $M'$. 
The copy of $P$ in $M'$ is essential if and only if $P$ was essential in $M$, because
any violating disc for $P$ in $M'$ can be chosen to avoid the attached 3-balls, and so forms a violating disc in $M$, and conversely, any violating disc in $M$ can be chosen to avoid $S$ and so becomes a violating disc in $M'$. Since we are assuming that $P$ is inessential, there is a certificate for $(M',P)$ that establishes this.

The verification of the certificate is straightforward. One checks that $S$ is a collection of 2-spheres, using the
Agol-Hass-Thurston algorithm. Then $\calT'$ is produced using the algorithm in \cite[Theorem 11.4]{Lackenby:EfficientCertification}.
Finally the certificate showing that $P$ is inessential in $M'$ is checked in the way described above.

We close this section with the observation that the argument above for Theorem \ref{Thm:NoNearlyCleanDiscPatches}
also gives the following result.

\begin{theorem}[No nearly clean disc patches]
\label{Thm:DiscPatchEssentialPattern}
Let $M$ be a compact orientable irreducible 3-manifold with essential pattern $P$. Let $\mathcal{T}$ be
a triangulation for $M$ in which $P$ is simplicial. 
Let $S$ be a compact orientable normal surface properly embedded in $M$, that is incompressible
and pattern-incompressible and that has no component that is a disc parallel to a
disc in $\partial M$ intersecting $P$ in an arc or the empty set. Suppose that $S$ is a sum of
two normal surfaces $S_1$ and $S_2$ in general position and that this summation is reduced.
Suppose also that $S$ has minimal weight up to isotopy preserving the pattern. Then 
no patch of the summation is a nearly clean disc patch.
\end{theorem}

\section{Fundamental decomposing surfaces}
\label{Sec:Fundamental}

Our goal in this section is to control the extended weight of our decomposing surfaces.

\begin{proposition}[Extended weight bound]
\label{Prop:ExpBoundExtendedWeight}
Let $M$ be a compact connected orientable irreducible 3-manifold, other than a 3-ball, with non-empty boundary and
an essential boundary pattern $P$.
Let $\mathcal{H}$ be an admissible handle structure of $(M,P)$ of uniform type. Then there is a constant $k$
(depending only on the handle types of $\mathcal{H}$) with the following property. There is 
a compact orientable connected non-separating normal surface $S$ properly embedded in $M$
with smallest pattern complexity among all such surfaces and with extended weight at most
$k^{|\mathcal H|}$.
\end{proposition}

Our proof of this fact relies heavily on work of Tollefson and Wang \cite{TollefsonWang}. A fairly minor complication
is that they worked
with triangulations rather than handle structures. Starting from $\mathcal{H}$, it is straightforward
to build a triangulation $\mathcal{T}$ of $M$ in which $P$ is simplicial, satisfying the
following conditions:
\begin{enumerate}[(i)]
\item the number $|\mathcal{T}|$ of tetrahedra is at most $c_1|\mathcal{H}|$, where $c_1$ is a constant
depending only on the handle types of $\mathcal{H}$;
\item any incompressible, pattern-incompressible normal surface $S$ in $\mathcal{T}$
can be isotoped, leaving $P$ invariant, to form a standard surface in $\mathcal{H}$ with extended weight at most $c_2 \, w(S)$,
where $w(S)$ is the weight of $S$ in $\mathcal{T}$ and $c_2$ is a constant
depending only on the handle types of $\mathcal{H}$. 
\end{enumerate}
Hence, we now work with a triangulation $\mathcal{T}$ satisfying these conditions.

We now recall the projective normal solution space associated with $\mathcal{T}$. 
As discussed in Section \ref{Sec:CertifyingInessential}, we assign a variable $x_i$ to each type of normal triangle and square
in $\calT$. When $S$ is a normal surface, these variables are non-negative integers forming the coordinates of
its normal surface vector $(S)$. Such vectors satisfy a system of linear constraints called the \emph{matching
equations}. There are three such equations for each face of $\mathcal{T}$ with tetrahedra on both sides. For each such face,
the triangles and squares in the tetrahedron on one side of the face must patch together correctly with the
triangles and squares in the tetrahedron on the other side. The triangles and squares on one side intersect the
face in normal arcs, which come in three different types. The matching equations assert that for each such normal
arc type, the number of triangles and squares on one side intersecting the face in arcs of that form is equal to the
number of triangles and squares on the other side also intersecting that face in arcs of that form.
The \emph{projective normal solution space} is the subset of $\mathbb{R}^{7|\mathcal{T}|}$ subject to the matching equations, the conditions
$x_i \geq 0$ for each $i$, and the equation $\sum x_i  = 1$. It is a convex compact polyhedron.
Each normal surface $S$ determines a point in the projective normal solution space by scaling all its co-ordinates $x_i$
so that $\sum x_i = 1$. The unique closed face of minimal dimension in the projective normal solution space containing this point
is called the face that \emph{carries} $S$ and is denoted $C_S$.

We will work with the following analogue of Thurston complexity in the presence of boundary pattern.

\begin{definition}
Let $S$ be a compact orientable surface properly embedded in $(M,P)$. If $S$ is connected, its \emph{reduced pattern-complexity}
$\chi_p^{red}(S)$ is defined to be $\max \{ 0, \chi_p(S) \}$. If $S$ is disconnected, then its reduced pattern complexity $\chi_p^{red}(S)$ 
is the sum of $\chi_p^{red}(S')$ over each component $S'$ of $S$.
\end{definition}

\begin{definition}
A compact oriented incompressible pattern-incompressible surface $S$ properly embedded in $(M,P)$ is said to be \emph{p-taut} if the following all hold:
\begin{enumerate}
\item $[S,\partial S]$ is non-trivial in $H_2(M, \partial M)$;
\item $S$ minimises reduced pattern-complexity in its homology class;
\item there is no homologically trivial union of components of $S$.
\end{enumerate}
We say that $S$ is \emph{lw-p-taut} if it is p-taut and it has minimal weight among all p-taut surfaces representing $[S,\partial S]$.
\end{definition}

This is an analogue of a definition of Tollefson and Wang \cite[Section 1]{TollefsonWang}. Instead of reduced pattern-complexity, their surfaces minimised $\chi_-$,
which is defined to be the sum, over all components $S'$ of $S$, of $\max \{ 0, -\chi(S') \}$. Their surfaces were known as lw-taut.

The following is an analogue of Theorem 3.3 of \cite{TollefsonWang}, but with their lw-taut surfaces replaced by lw-p-taut surfaces.

\begin{theorem}
\label{Thm:TollefsonWangPattern}
Let $M$ be a compact orientable irreducible 3-manifold with $b_1(M) >0$ and let $P$ be an essential boundary pattern. 
Let $\mathcal{T}$ be a triangulation of $M$ in which $P$ is simplicial. 
Let $S$ be a compact oriented incompressible pattern-incompressible properly embedded normal surface that is lw-p-taut. Let $C_S$ be the face of the projective normal solution space that carries $S$.
Then every surface $G$ carried by $C_S$ inherits an orientation and so represents
an element $[G] \in H_2(M, \partial M)$. With this orientation, $G$ is lw-p-taut. Moreover, for any two
surfaces $G$ and $H$ carried by $C_S$, the normal sum $G+H$ satisfies $[G+H] = [G] + [H]$.
\end{theorem}

Rather than give the full proof here, we sketch a proof of the following lemma, which gives the central idea.
In fact, we will only need this lemma rather than the full strength of Theorem \ref{Thm:TollefsonWangPattern} 
in our proof of Proposition \ref{Prop:ExpBoundExtendedWeight}.

\begin{lemma}
\label{Lem:WeakTollefsonWang}
Let $M$ be a compact orientable irreducible 3-manifold with $b_1(M) >0$ and let $P$ be an essential boundary pattern. 
Let $\mathcal{T}$ be a triangulation of $M$ in which $P$ is simplicial. 
Let $S$ be a compact oriented incompressible pattern-incompressible properly embedded normal surface that is lw-p-taut.
Suppose that $S$ is the normal sum $G + H$, where the sum is reduced. Then the orientation that the patches inherit from
$S$ glue together to give orientations on $G$ and $H$. Hence, the homology classes satisfy $[S] = [G] + [H]$. Furthermore,
$G$ and $H$ are lw-p-taut and $\chi_p^{red}(S) = \chi_p^{red}(G) + \chi_p^{red}(H)$.
\end{lemma}

Since $S$ is oriented, each patch is oriented. The crucial observation in both \cite{TollefsonWang}
and here is that the orientations on the patches glue together to form consistent orientations on $G$ and $H$. Hence, $G$ and $H$ represent
homology classes $[G]$ and $[H]$ in $H_2(M, \partial M)$ and because normal summation respects the orientation, then $[S] = [G] + [H]$.
Suppose that the patches do not give consistent orientation on $G$ and $H$. Then at some component $C$ of $G \cap H$, the regular switch does not respect the
orientations. That means that at the corresponding trace curves in $S$, the orientations on $S$ both point towards $C$ or both point away from $C$.
Therefore an irregular switch at $C$ respects the orientation on $S$. Perform this irregular switch, and at all the other components of $G \cap H$, perform
regular switches, giving an oriented surface $S'$. This is homologous to $S$ and it has the same pattern complexity. In fact, $\chi_p^{red}(S')$ equals
$\chi_p^{red}(S)$. This is a consequence of Theorem \ref{Thm:DiscPatchEssentialPattern}, which states that the summation for $S$ has no nearly clean disc patches. Hence, when we perform the switches giving $S'$, this does not create any components with negative pattern complexity. Now $S'$ has a return. This leads either to a weight reducing isotopy or to a pattern-compression, both of which contradict the hypothesis that $S$ was lw-p-taut. (Note that if a pattern-compression to $S'$ does not reduce $\chi_p^{red}$, then the component that is pattern-compressed is homologically trivial and can be removed, which again implies that $S$ was not lw-p-taut.)
Finally, $G$ (and $H$) is lw-p-taut, since if there was a homologous surface $G'$ with $\chi_p^{red}(G') < \chi_p^{red}(G)$ or with $\chi_p^{red}(G') = \chi_p^{red}(G)$ but with strictly smaller weight than $G$, then one could form the oriented sum of $G'$ and $H$ and thereby disprove that $S$ was lw-p-taut. Similarly, $G$ and $H$ cannot have any non-empty homologically trivial union of components.

In fact, to prove Theorem 3.3 of \cite{TollefsonWang}, Tollefson and Wang consider a strengthening of the above lemma.
In \cite[Lemma 3.2]{TollefsonWang}, they consider the situation where some positive multiple $mS$, with $m > 0$,
is a normal sum $G + H$, and they deduce a similar conclusion to that above. In our setting, we have an analogous result
which readily leads to Theorem \ref{Thm:TollefsonWangPattern}. This is a routine modification of the argument of
Tollefson and Wang and is omitted.

Hence, we have the following result.

\begin{theorem}[Fundamental surface]
\label{Thm:FundamentalTaut}
Let $M$ be a compact orientable irreducible 3-manifold with $b_1(M)>0$ and let $P$ be an essential boundary pattern. Let $\mathcal{T}$ be a triangulation of $M$ in which $P$ is simplicial. Then there is a fundamental normal surface $S$ properly embedded in $M$ satisfying the following:
\begin{enumerate}
\item $S$ is a compact connected orientable surface that is non-separating in $M$;
\item $S$ has smallest reduced pattern-complexity among all surfaces satisfying (1).
\end{enumerate}
\end{theorem}

\begin{proof}
Pick a compact orientable surface $S$ satisfying (1) and (2) above. Such a surface exists because $b_1(M)>0$.
We may assume that it is incompressible and pattern-incompressible, by Corollary \ref{Cor:MinimalComplexityImpliesEssential}. Hence it can be isotoped, via an isotopy preserving the pattern, to a normal surface, also called $S$. We may further assume that $S$ has least weight among all such surfaces. So it is lw-p-taut. Let $C_{S}$ be the face that carries $S$. Suppose that $S$ is not fundamental. Then $S$ can be written as a non-trivial normal sum $G+ H$ of normal surfaces $G$ and $H$. We may assume that $G$ is connected, since if $G_1$ is any component of $G$, then $S = G_1 + ((G-G_1) + H)$. By Lemma \ref{Lem:WeakTollefsonWang}, $G$ and $H$ inherit orientations which make them lw-p-taut. In particular, $G$ is non-separating in $M$. 
Since $\chi_p^{red}(S) = \chi_p^{red}(G) + \chi_p^{red}(H) \geq \chi_p^{red}(G)$. But $G$ has smaller weight than $S$, contrary to our assumption.
\end{proof}

\begin{proof}[Proof of Proposition \ref{Prop:ExpBoundExtendedWeight}]
We are given a handle structure $\mathcal{H}$ of $(M,P)$ of uniform type. From this, construct a triangulation $\mathcal{T}$ of $M$
in which $P$ is simplicial and satisfying (i) and (ii) above. By Theorem \ref{Thm:FundamentalTaut}, there is a fundamental normal surface $S$ properly embedded in $M$
satisfying (1) and (2) in the theorem. By a result of Hass, Lagarias and Pippenger \cite[Lemma 6.1]{HassLagariasPippenger}, the weight of $S$ in $\mathcal{T}$
is at most $|\mathcal{T}|^2 2^{7|\mathcal{T}|+7}$. Hence, $S$ is isotopic, leaving $P$ invariant, to a standard surface $S$ in $\mathcal{H}$
with extended weight at most $c_2 |\mathcal{T}|^2 2^{7|\mathcal{T}|+7}$. Then by Lemma \ref{Lem:NormalForm}, $S$ can be isotoped, leaving $P$ invariant, 
to be normal
in $\mathcal{H}$ and without increasing its extended weight. Since, $|\mathcal{T}| \leq c_1 |\mathcal{H}|$, we obtain our required constant $k$ and exponential bound
on the extended weight of $S$.
\end{proof}

\section{Certifying an essential hierarchy}
\label{Sec:CertifyingEssentialHierarchy}

In this section, we complete the proof of Theorem \ref{Thm:EssentialNPcoNP}, by showing that 
{\sc Essential boundary pattern} is in NP. Thus, we are given a triangulation $\calT$ of a compact orientable 3-manifold
$M$ with essential pattern $P$ that is simplicial, and we must provide a certificate, verifiable in polynomial time, that
establishes that $P$ is indeed essential. We may assume that $\partial M$ is non-empty, as otherwise there is nothing to prove.
Using the techniques at the end of Section \ref{Sec:CertifyingInessential},
it is not hard to reduce to the case where $M$ is irreducible. We will describe this reduction in more detail 
at the end of this section, but for the moment, suppose that $M$ is irreducible. Then, by Theorem \ref{Thm:EssentialHierarchy}, 
we know that $(M,P)$ admits an essential hierarchy. Moreover, the existence of such a hierarchy
guarantees that $P$ is essential. So, our goal is to write down such a hierarchy and prove that it is essential. We know by 
Proposition \ref{Prop:RefinedAlgorithmTerminates} that there is such a hierarchy with length that is
linearly bounded in terms of the number of tetrahedra in $\calT$. Furthermore, by Theorem \ref{Thm:RemainUniformType},
the intermediate 3-manifolds have handle structures of uniform type. However, this only controls the local
complexity of each handle, not the total number of handles. Thus, it is not immediately clear how to 
encode the handle structures of these manifolds. Nevertheless, we will establish the following result.

\begin{theorem}[Efficient encoding of a hierarchy]
\label{Thm:VerifyEssHierarchy}
Let $M$ be a compact connected orientable irreducible 3-manifold with a boundary pattern $P$.
Let $\mathcal{H}$ be a handle structure for $(M,P)$ of uniform type. If $(M,P)$
admits an essential hierarchy, then there is a hierarchy for $(M,P)$ that may be
described by an amount of data that is at most a polynomial function of $|\mathcal{H}|$
and this hierarchy can be verified as essential by a polynomial-time algorithm.
Furthermore, for each decomposition 
$$(M_i,P_i) \xrightarrow{S_i} (M_{i+1}, P_{i+1})$$
in this hierarchy, one of the following holds:
\begin{enumerate}
\item $S_i$ is connected, orientable and non-separating and has smallest pattern-complexity
among all such surfaces, apart from when $M$ is closed and $i=1$, when $S_1$ may be separating and need not have smallest pattern-complexity, but it must not be a 2-sphere;
\item $S_i$ is a union of non-trivial squares;
\item $S_i$ is a union of properly embedded incompressible pattern-incompressible annuli disjoint from the pattern.
\end{enumerate}
\end{theorem}

Recall that the hierarchy construction loop given in Sections \ref{Sec:AlgorithmCompressible} and \ref{Sec:AlgorithmParallelity} used the following steps iteratively (where the numbering is as
given there):
\begin{itemize}
\item[(2)] Apply {\sc Create admissible handle structure}.
\item[(2$'$)] Replace each component of the interior parallelity bundle that is an $I$-bundle over a disc by a single 2-handle.
\item[(2$''$)] Cut along the vertical boundary components of the remaining $I$-bundles and
further decompose these using squares into 3-balls with essential boundary
patterns.  In the case where a component of the interior parallelity bundle is an entire component of the 3-manifold, decompose it along an essential vertical annulus, and then along squares into 3-balls with essential boundary patterns. Finally, discard all these 3-balls.
\item[(6)] Let $S$ be any connected, non-separating, properly embedded, oriented normal surface in the resulting 3-manifold. 
\item[(8)] Assuming that $S$ is not a violating disc, then cut the handle structure along $S$.
\end{itemize}

After an initial application of (2), (2$'$), and (2$''$), we can view the iterations of this loop as running in the order (6), (8), (2), (2$'$), (2$''$).
Furthermore, for the vertical boundary components of the interior parallelity involved in step (2$''$), compression or pattern-compression discs
may be discovered. These are dealt with using the algorithm {\sc (Pattern-)compressible vertical boundary}.
We will show that these processes may be completed efficiently.

Suppose that we are given an admissible handle structure $\calH$ of a compact orientable irreducible 3-manifold
$M$ with boundary pattern $P$, where $\calH$ has no 3-handles and is of uniform type.
Suppose that we are also given a compact orientable normal surface $S$ properly embedded in $M$.
It is a consequence of Theorem \ref{Thm:RemainUniformType} that the handle structure obtained by decomposing $\calH$ along $S$
has uniform type. However, it might have a large number of handles, when $S$ has very high extended weight.
But most handles will be of the following form.

\begin{definition} 
Let $\calH$ be a handle structure for the compact 3-manifold $M$ with boundary pattern $P$.
Let $S$ be a normal surface properly embedded in $M$. Let $\calH'$ be the induced handle
structure on $M \cut S$. Then a handle of $\calH'$ is a \emph{parallelity handle for $S$}
if it lies between parallel normal discs of $S$.
\end{definition}

Each parallelity handle for $S$ is a copy of $D^2 \times I$ where $(D^2 \times I) \cap S = D^2 \times \partial I$.
We view this product structure as an $I$-bundle over $D^2$. The $I$-bundle structures can be chosen so
they agree on the intersection of any collection of parallelity handles. Hence, the
union of the parallelity handles is an $I$-bundle, called the \emph{parallelity bundle} $\calB_S$ for $S$.

The vertical boundary $\partial_v \calB_S$ of $\calB_S$ is a union of annuli. However, these annuli
need not be properly embedded. Instead the intersection with $\partial M$ is a union of fibres in
the bundle. We say that $\partial_v \calB_S \cut \partial M$ is the \emph{inner vertical boundary}.
It is a disjoint union of properly embedded annuli disjoint from the pattern, and squares in $M \cut S$.

The following result is analogous to \cite[Theorems 9.2 and 9.3]{Lackenby:EfficientCertification}. 
It has essentially the same proof, which is omitted. The proof relies heavily on the algorithm of Agol, Hass and Thurston \cite{AgolHassThurston}.

\begin{theorem}[{\sc Cutting along a surface}]
\label{Thm:CutHS}
There is an algorithm that takes, as its input,
\begin{enumerate}
\item a handle structure $\mathcal{H}$ of uniform type, with $h$ handles, for a compact orientable 3-manifold $M$ with boundary pattern $P$;
\item a compact orientable properly embedded normal surface $S$;
\end{enumerate}
and provides as its output, the following data. If $(M',P')$ is the manifold with pattern obtained by decomposing along $S$, and $\mathcal{B}_S$ is its parallelity bundle for $S$, then the algorithm produces the handle structure that $M' \cut \mathcal{B}_S$ inherits, and for each component $B$ of $\mathcal{B}_S$, it determines:
\begin{enumerate}
\item the genus and number of boundary components of its base surface;
\item whether $B$ is a product or twisted $I$-bundle;
\item the way that $\partial_v B$ and $M' \cut \mathcal{B}_S$ intersect;
\item in each component of $\partial_v B \cap \partial M'$, its number of arcs of intersection with $P$.
\end{enumerate}
The algorithm runs in time that is bounded by a polynomial in $h$ and the logarithm of the number of handles of $S$.
\end{theorem}

As a result of this, we introduce the following definition. 

\begin{definition}
A \emph{generalised handle structure} on a compact orientable 3-manifold $M$ with boundary pattern $P$ is
a description of $M$ as the union of the following sets:
\begin{enumerate}
\item a compact 3-manifold $M'$ with boundary pattern $P'$ with a handle structure $\calH$;
\item a compact $I$-bundle $\calB$ with pattern $P''$ equal to the union of $\partial (\partial_h \calB)$ and some vertical arcs in $\partial_v \calB$. 
\end{enumerate}
The bundle $\calB$ is attached to $M'$ along a union of disjoint components of $\partial_v \calB \cut P''$.
These are identified with square and annular components of $\partial M' \cut P'$ lying in $\calH^0 \cup \calH^1$. Moreover, for each such component
of $\partial M' \cut P'$, its intersection with $\calH^1$ is a non-empty collection of discs, where the product structure
on the component of $\partial M' \cut P'$ agrees with the product structure on $\calH^1$. 
The data encoding a generalised handle structure includes, for each component $B$ of $\mathcal{B}$:
\begin{enumerate}
\item the genus and number of boundary components of its base surface;
\item whether $B$ is a product or twisted $I$-bundle;
\item the way that $\partial_v B$ and $M'$ intersect;
\item in each component of $\partial_v B \cut \partial M'$, its number of arcs of intersection with $P$.
\end{enumerate}
The boundary pattern $P$ is set to be the union of $P' \cut \partial_v \calB$ and $P'' \cut \partial M'$.
We write a generalised handle structure as $(\calH, \calB)$.
\end{definition}

Thus, the output from {\sc Cutting along a surface} is a generalised handle structure on $M \cut S$.

\begin{definition}
A generalised handle structure $(\calH, \calB)$ on $(M,P)$ is \emph{admissible} if the following all hold:
\begin{enumerate}
\item each 0-handle of $\mathcal{F}(\calH)$ either has positive index or has zero index and is disjoint from $P$;
\item for any simple closed curve $C$ lying in the boundary of a 0-handle of $\calH$, that is disjoint from $\calF \cup \partial_v \calB$, and that intersects $P$
at most 3 times, $C$ bounds a disc in $\partial \calH^0 - (\calF \cup \partial_v \calB)$ intersecting $P$ in the empty set, an arc or a tripod.
\item the interior of $\partial_v \calB$ lies in the interior of $M$.
\end{enumerate}
\end{definition}

There is an obvious version of the algorithm {\sc Create admissible handle structure}, which we term 
{\sc Create admissible generalised handle structure}. This takes, as its
input, a generalised handle structure $(\calH, \calB)$ on a compact orientable irreducible 3-manifold $M$ with boundary pattern $P$,
with no 3-handles, and it either creates an admissible generalised handle structure 
for a manifold $(M',P')$ that is obtained from $(M,P)$ by decomposing along disjoint non-trivial squares and discarding balls with essential boundary pattern,
or provides a violating disc for $(M,P)$.
Furthermore, as in Remark \ref{Rem:WhatIfMReducible}, this algorithm also admits a compact orientable reducible 3-manifold $M$
as input. Then we obtain either a correct declaration that $M$ is reducible, or a violating disc for $(M,P)$, or an admissible generalised handle structure on a 3-manifold $M'$ with boundary pattern $P'$ that is obtained from $(M,P)$ by decomposing along disjoint non-trivial squares and removing closed rational homology 3-sphere connected summands.
When $\calH$ has uniform type, this runs in polynomial time as a function of $|\calH|$ and $\log (\chi_-(\partial_h \calB)+2|\calB|+1)$.

When $(\calH, \calB)$ is an admissible generalised handle structure, its \emph{interior parallelity bundle} is the union of $\calB$ and the interior parallelity handles of $\calH$. The $I$-bundle structures on the interior parallelity handles and on $\calB$ agree, and so their union is again an $I$-bundle.

There is also an obvious generalisation of {\sc (Pattern-)compressible vertical boundary} to admissible generalised handle structures with no 3-handles. When $\calH$ has uniform type, its running time is also bounded
as above.

When $M$ has non-empty boundary, the certificate required by Theorem \ref{Thm:EssentialNPcoNP} is as follows:
\begin{enumerate}
\item a sequence of generalised handle structures $(\mathcal{H}_1, \calB_1), \dots, (\mathcal{H}_{\ell+1},\calB_{\ell+1})$, 
where $(\mathcal{H}_i, \calB_i)$ is a generalised handle structure for a 3-manifold $M_i$ 
with boundary pattern $P_i$;
\item  for some of the $i \leq \ell$, a surface $S_i$ or $D_i$ embedded in $M_i$ that is normal in $\mathcal{H}_i$ and horizontal in $\calB_i$; the surface $S_i$ is properly embedded in $M_i$, but in the case of $D_i$, it is not (see (4) below);
\end{enumerate}
These handle structures and surfaces are required to have the following properties:
\begin{enumerate}
\item the initial handle structure $\calH_1$ is obtained from the dual of the given triangulation $\calT$ by collapsing 2-handles and 3-handles until there are no 3-handles;
\item each handle structure $\calH_i$ has uniform type, as described in Theorem \ref{Thm:RemainUniformType};
\item when $S_i$ is provided, then $\calB_i$ is empty and $\calH_i$ is an admissible handle structure and $(\mathcal{H}_{i+1}, \calB_{i+1})$ is obtained from $\mathcal{H}_i$ using
Theorem \ref{Thm:CutHS};
\item when $D_i$ is provided, then $(\calH_i, \calB_i)$ is an admissible generalised handle structure
and $D_i$ is properly embedded in the manifold obtained from $M_i$ by removing the components of the interior parallelity bundle that are not $I$-bundles over discs and replacing each of the remaining components by 2-handles; it is a compression disc or pattern-compression disc for the vertical boundary of the interior parallelity bundle (minus the components that are $I$-bundles over discs), and
$(\mathcal{H}_{i+1}, \calB_{i+1})$ is obtained from $(\mathcal{H}_i, \calB_i)$ and $D_i$ by applying {\sc (Pattern-)compressible vertical boundary};
\item each $S_i$ and each $D_i$ has extended weight at most $k^{|\mathcal{H}_1|}$, where $k$ is a universal constant;
\item when no $D_i$ or $S_i$ is provided, then $(\mathcal{H}_{i+1}, \calB_{i+1})$ is obtained from $(\mathcal{H}_i, \calB_i)$ by applying {\sc Create admissible generalised handle structure}, then steps (2$'$) and (2$''$);
\item the length $\ell$ of this hierarchy is at most a polynomial function of $c_q(\mathcal{H}_1)$;
\item the final manifold is a collection of 3-balls with essential pattern and its generalised handle structure is an admissible handle structure with interior parallelity bundle consisting of 2-handles.
\end{enumerate}
When $M$ is closed, the certificate is slightly adjusted, as follows. The first handle structure is again dual to the given triangulation. The first surface $S_1$ is a closed connected orientable incompressible surface, other than a 2-sphere, with weight at most an exponential function of $|\calT|$. The next handle structure is obtained by applying Theorem \ref{Thm:CutHS}, then cancelling 3-handles and 2-handles (or part of the bundle) until there are no 3-handles, then applying {\sc Create admissible generalised handle structure}. The remaining part of the certificate is as described above.

The existence of the certificate when $M$ is irreducible and $P$ is essential is established as follows. We may assume that $M$ is connected. For the proof of Theorem \ref{Thm:EssentialNPcoNP}, we may also assume that it has non-empty boundary. In Theorem \ref{Thm:VerifyEssHierarchy}, $M$ may be closed. But in this case, by the proof of \cite[Theorem 4.1.30]{Matveev}, it contains a properly embedded orientable incompressible surface with no 2-sphere components, and with extended weight at most an exponential function of $|\calH|$. Indeed, such a surface arises as either a fundamental surface or the normal sum of a fundamental surface with itself. So we now assume that $M$ has non-empty boundary.

We run the algorithm given in Section \ref{Sec:AlgorithmCompressible} with the refinements from Section \ref{Sec:AlgorithmParallelity}. In Step (6) of the algorithm in Section \ref{Sec:AlgorithmCompressible}, we are free to pick any connected non-separating properly embedded orientable normal surface in the manifold. We pick one with smallest pattern complexity. Then Proposition \ref{Prop:ExpBoundExtendedWeight} guarantees that there is such a surface with extended weight at most $k^{|\mathcal{H}_i|}$, for some universal constant. Furthermore, since the surface has smallest pattern complexity, Corollary \ref{Cor:MinimalComplexityImpliesEssential} ensures that it is incompressible and pattern-incompressible. Hence, if the algorithm leaves the hierarchy construction loop and produces a violating disc, then this must form a compression or pattern-compression disc for some vertical boundary annulus for the interior parallelity bundle. This is provided in (4). We can arrange for the violating disc to have extended weight at most $k^{|\mathcal{H}_1|}$, for the following reason. The manifold in which the disc lives comes with a handle structure of uniform type and where its number of handles is bounded above by a linear function of $|\calH_1|$. A violating disc can be chosen to be fundamental in a triangulation refining this handle structure by Theorem \ref{Thm:ViolatingFundamental}. Hence its extended weight is at most $k^{|\mathcal{H}_1|}$ for some universal constant $k$. This violating disc is necessarily a compression or pattern-compression disc for the vertical boundary of the interior parallelity bundle of $M_i$ (minus the components that are $I$-bundles over discs).

Since $M$ is assumed to be irreducible and $P$ is essential, the algorithm terminates with an essential hierarchy, giving (8). Furthermore, the length of this hierarchy is at most 
a polynomial function of $c_q(\mathcal{H}_1)$, essentially by Proposition \ref{Prop:RefinedAlgorithmTerminates}, which gives (7). Finally, condition (2) is a consequence of Theorem \ref{Thm:RemainUniformType}.

We now explain how to verify the certificate. There are two main aspects to this. Firstly, we must check that each $\mathcal{H}_{i+1}$ really is obtained from $\mathcal{H}_i$ in the ways described above. This guarantees that we really do have a partial hierarchy. Secondly, we must check that the final manifold $M_{\ell+1}$ is balls with essential boundary pattern $P_{\ell+1}$. Since the final handle structure is admissible with interior parallelity bundle consisting of 2-handles, Theorem  \ref{Thm:RemainUniformType} implies that its number of handles is at most a uniform constant times $|\mathcal{H}_1|$. Furthermore each of its 0-handles is of uniform type. Hence, it is straightforward to verify, in polynomial time, that $\partial M_{\ell+1}$ is a collection of spheres and that its pattern $P_{\ell+1}$ is essential.
This implies that each decomposing surface is incompressible by Lemma \ref{Lem:EssentialPullsBack}, and $M_{\ell+1}$ is irreducible by Lemma \ref{Lem:IrreduciblePushesForward}. So $M_{\ell+1}$ is a collection of balls.

So we now focus on verifying that each $\mathcal{H}_{i+1}$ is obtained from $\mathcal{H}_i$ as above. Suppose first that the normal surface $S_i$ is given. We check that it is connected, orientable and not a 2-sphere using \cite{AgolHassThurston} and \cite[Theorem 9.7]{Lackenby:EfficientCertification}. We apply Theorem \ref{Thm:CutHS}
and check that the resulting generalised handle structure is $\calH_{i+1}$. Note that $S_i$ has extended weight at most $k^{|\mathcal{H}_i|}$. Since $|\calH_i|$ is bounded above by a universal constant times $|\calH_1|$ by (2) above, we deduce that the algorithm in Theorem \ref{Thm:CutHS} completes in polynomial time.

Suppose now that the disc $D_i$ is given in $\mathcal{H}_i$. This is a compression disc or pattern-compression disc for a vertical boundary component $A$ of the  interior parallelity bundle 
(minus the components that are $I$-bundles over discs). Since the handle structure containing $D_i$ is admissible and of uniform type, its number of 0-handles is at most a universal constant times $|\mathcal{H}_1|$. So we may confirm that $D_i$ is a compression disc or pattern-compression disc for $A$. Then the procedure in {\sc (Pattern-)compressible vertical boundary} may be implemented in polynomial time and the verifier may check that the resulting handle structure is $\mathcal{H}_{i+1}$. 

Finally, when neither $D_i$ nor $S_i$ is given, then the procedure {\sc Create admissible generalised handle structure} followed by steps (2$'$) and (2$''$) may be implemented and the verifier can check that the resulting handle structure is $\mathcal{H}_{i+1}$.

Thus, the verifier can check that this does give an essential hierarchy for $(M,P)$. This completes the proof of Theorem \ref{Thm:EssentialNPcoNP} in the case where $M$ is irreducible and $P$ is essential.

Suppose that we were not promised that $M$ is irreducible, but that we were provided with a certificate as above. As described in Remarks \ref{Rem:WhatIfMReducible} and \ref{Rem:MaybeReducible}, versions of {\sc (Pattern-)compressible vertical boundary} and
{\sc Create admissible generalised handle structure} still work in that setting. These do not declare that $M$ is reducible, because we are
assuming that this is not the case. The algorithms might remove connected summands of $M$ that are closed rational
homology 3-spheres but this does not affect whether $P$ is essential. Hence, if the verifier correctly verifies the above certificate
then they can be sure that $P$ is essential, even if they are not certain that $M$ is irreducible.

We now complete the proof of Theorem \ref{Thm:EssentialNPcoNP}. We need to deal with the general case where $M$ is reducible. Our certificate for establishing that $P$ is essential is:
\begin{enumerate}
\item a vector $(S)$ of a normal surface $S$ in $\mathcal{T}$ with weight at most $2^{740t^2}$; this will in fact be a maximal collection of disjoint embedded normal spheres;
\item a triangulation $\mathcal{T}'$ for a 3-manifold $M'$ with at most $1000t$ tetrahedra; this will be $M \cut S$ with a 3-ball attached to each spherical boundary component arising from $S$;
\item a certificate that $M'$ has essential boundary pattern.\end{enumerate}

Parts (1) and (2) of the certificate are verified in the same way as in Section \ref{Sec:CertifyingInessential}. One checks that $S$ is a collection of 2-spheres, using the
Agol-Hass-Thurston algorithm. A triangulation for $M \cut S$ is produced using \cite[Theorem 11.4]{Lackenby:EfficientCertification}, and then a coned 3-ball
is attached to each spherical boundary component arising from $S$, forming $\mathcal{T}'$. The copy of $P$ in $M'$ is essential if and only if $P$ was essential in $M$, because
any violating disc for $P$ in $M'$ can be chosen to avoid the attached 3-balls, and so forms a violating disc in $M$, and conversely, any violating disc in $M$ can be chosen to avoid $S$ and so becomes a violating disc in $M'$. Now $M'$ is irreducible. So we have a certificate, verifiable in polynomial time, that $(M',P)$ is essential. As mentioned above, the verifier does not need to be certain that $M'$ is irreducible for them to verify the certificate that $P$ is essential.

This completes the proof of Theorem \ref{Thm:EssentialNPcoNP}.

\bibliography{alg-incompressible-biblio}
\bibliographystyle{plain}

\end{document}